\documentclass{amsart}
\usepackage{amsmath,dsfont,amsfonts}
\usepackage{amssymb}
\usepackage{amsthm}
\usepackage{verbatim}
\usepackage{color}
\newcommand{\B}{{\mathcal{B}}}
\newcommand{\RR}{{\mathbb{R}}}
\newcommand{\ds}{\displaystyle}
\newcommand{\der}{\partial}

\newcommand{\eps}{\varepsilon}

\newcommand{\dive}{\text{\normalfont div}}
\def\qed{\hfill$\square$\vspace{0.5cm}}    
\numberwithin{equation}{section}

\newtheorem{theo}{Theorem}[section]
\newtheorem{lemma}[theo]{Lemma}
\newtheorem{pr}[theo]{Proposition}

\newtheorem{rem}[theo]{Remark}

\begin{document}

\title[Carleman estimates for the parabolic transmission problem ]{Carleman estimates for the parabolic transmission problem and H\"older propagation of smallness across an interface}

\author[]{Elisa~Francini}
\address{Dipartimento di Matematica e Informatica ``U. Dini'',
Universit\`{a} di Firenze, Italy}
\email{elisa.francini@unifi.it}
 \author[]{Sergio~Vessella}
 \address{Dipartimento di Matematica e Informatica ``U. Dini'',
Universit\`{a} di Firenze, Italy}
\email{sergio.vessella@unifi.it}

\date{\today}


\begin{abstract} In this paper we prove a H\"older  propagation of smallness  for solutions to  second order parabolic equations  whose general anisotropic leading coefficient has a jump at an interface. We assume that the leading coefficient  is Lipschitz continuous with respect to the parabolic distance on both sides of the interface. The main effort consists in proving a local Carleman estimate for this parabolic operator.
\end{abstract}

\maketitle


\section{Introduction}\label{intro}

The main purpose of this paper is to study  unique continuation properties and  propagation of smallness for solutions to  second order parabolic equations whose anisotropic leading coefficients have jumps at an interface. Although there exist good general books and papers about Carleman estimates, unique continuation properties and  related  propagation of smallness (see \cite{Be-LR-1}, \cite{cal}, \cite{Ho63}, \cite{Ho85}, \cite{Is1}, \cite{Tr}, \cite{Zu}) and a lot of surveys and introductory papers on the subject and  on its several applications (see \cite{FI}, \cite{Kl1}, \cite{KlTi}, \cite{KSU}, \cite{Is4}, \cite{LeRosseau-Lebeau}, \cite{Ve1}, \cite{Ve4}, \cite{Ya1}), we would like to give to the non expert reader some basic notions and quick historic panorama on the subject.

We say that a linear partial differential equation $\mathcal{L}(u)=0$, enjoys a unique continuation property (UCP) in a connected
open set $\Omega\subset\mathbb{R}^{N}$ if the following property holds true \cite{hTat}:
for any open subset $\widetilde{\Omega}$ of $\Omega$
\begin{equation}\label{intr-1}
\mathcal{L}(u)=0\mbox{ in }\Omega \mbox{ and }u=0\mbox{ in }\widetilde{\Omega}\quad
\mbox{ imply }\quad u=0\mbox{ in }\Omega.
\end{equation}

We call quantitative estimate of unique continuation (QEUC) or stability estimate related to the UCP property \eqref{intr-1} the following type of result:
\begin{equation*}
\mathcal{L}(u)=0\mbox{ in }\Omega\mbox{ and } u \mbox{ small in }\widetilde{\Omega} \quad\mbox{ imply } \quad u \mbox{ small in }\Omega.
\end{equation*}
Of course, the research in these topic is of some interest  if either the function $u$  is not analytic or if the operator $\mathcal{L}$ has nonanalytic coefficients in $\Omega$.

 In this sense  Carleman, in his paper \cite{ca} in 1939,  marked a true milestone, because he proved that the 2D elliptic operator $\mathcal{L}(u)=\Delta u+a(x)u$, where $a$ is a bounded function, enjoys the UCP. At the same time  Carleman conceived a highly constructive method that wide opened the doors to  quantitative estimates of unique continuation for equations with nonanalytic coefficients. Since the 1950s the investigation on UCP has been extended to more general differential operators with a special attention  to the regularity, in the first place, of the coefficients of the principal part of the operators.
For instance, it was proved in \cite{pl}, see also \cite{hMan}, \cite{Mi}, that the UCP for the second order elliptic equations doesn't hold true if the coefficients of principal part is  in $C^{0,\alpha}(\mathbb{R}^n)$ for $\alpha<1$ and $n>2$. On the other side,  the UCP applies when the coefficients of the principal part are Lipschitz continuous (see \cite{AKS}, \cite{ho}) and, consequently, a H\"older type propagation of smallness can be proved in the form of three-sphere inequality (\cite{Lan}). We refer to \cite{ARRV} for an extensive and detailed analysis of connection between the UCP and  propagation of smallness for second order elliptic equation. We  should mention that the UCP for the second order elliptic equations with two variables with $L^{\infty}$ coefficients can be deduced from the theory of quasiconformal mappings (\cite{bjs}, \cite{s}, \cite{alma} ).

\bigskip

In the parabolic context, broadly speaking, the investigation about UCP is focused on two main topics: (i) backward uniqueness and backward stability estimates,  (ii) spacelike unique continuation properties (which include the noncharacteristic Cauchy problem) and their quantitative versions.
In this paper we concentrate on the second issue. For backward uniqueness and stability we refer to \cite{isakov}, \cite{Ve4}, \cite{Ya1}.

Let us consider the operator
\begin{equation}\label{intr-2p}
\mathcal{L}(u)=\dive(A(x,t)Du)-\partial_tu,
\end{equation}
where $A(x,t)$, $(x,t)\in\mathbb{R}^{n+1}$, is a symmetric $n\times n$ matrix which we assume uniformly elliptic. The spacelike UCP has the following formulation: let $\widetilde{\Omega}$ be any open subset in $\Omega$ and let $J\subset (0,T)$ an interval or a single point; we say that  $\mathcal{L}$ enjoys the spacelike UCP if
\begin{equation}\label{intr-3p}
\mathcal{L}(u)=0\mbox{ in }\Omega\times (0,T)\mbox{ and }u=0 \mbox{ on } \widetilde{\Omega}\times J\quad \mbox{ imply }\quad u=0 \mbox{ on }\Omega\times J.
\end{equation}
There exists a broad literature about property \eqref{intr-3p}, so that we mention here the most meaningful papers (in the authors' opinion) in this topic and we refer to the survey paper \cite{Ve4} for more extensive references.

The first result of spacelike UCP was proved in \cite{Nir} for $n=1$, $J$ an interval, and for the solutions to the equation $u_{xx}-u_{t}=a(x,t)u_x+b(x,t)u$, $a,b\in L^{\infty}$.  For $n\geq 1$, $J$ is a single point, and the matrix $A$ in class $C^2$ and independent of $t$, the spacelike UCP was proved in \cite{ItYa} (see also \cite{Yamabe}). If $A(x,t)$ belongs to $C^{\infty}$ and $J$ is an interval, the spacelike UCP was proved in \cite{Miz}; such a result is still valid if $u$ is a solution to $\tilde{\mathcal{L}}u=0$ where $\tilde{\mathcal{L}}$ is a first order perturbation of $\mathcal{L}$, namely  $\tilde{\mathcal{L}}(u)=\mathcal{L}(u)+W(x,t)\cdot\nabla u+ V(x,t)\,u$, with $W,V\in C^{\infty}$. The result in \cite{Miz} was substantially improved in \cite{SS} and in \cite{So}: in \cite{SS} the matrix  $A(x,t)$ belongs to $C^{1}$, $W\in L^{\infty}$, $V\in L^{\infty}$ and in  \cite{So}  $V\in L_{loc}^{\frac{n+2}{2}}$, $W\in L^{\infty}$, but again $A(x,t)\in C^{\infty}$. Lees and Protter,
(\cite{LP}, \cite{Pr}) proved uniqueness for the Cauchy problem for the equation $\widetilde{\mathcal{L}}(u)=0$, when $A\in C^2$. A stability estimate of H\"{o}lder type (far from initial and final times) for the Cauchy problem under the same hypotheses of \cite{LP}, \cite{Pr}, was proved (perhaps for the first time) in \cite{AmSh}, see also \cite{Am} and \cite{LRS}.
We mention the quite recent estimate of log type up to the initial and final time proved in \cite{CY}. We refer to \cite{EsVe}, \cite{Ve3} for additional improvements on the regularity of the leading coefficients.

The following strong unique continuation property (SUCP) for parabolic equations whose coefficients are smooth and \textit{time independent} was proved in \cite{LanO}: if $(x_0,t_0)\in \Omega\times (0,T)$ and $\mathcal{L}(u)=0$ in $\Omega\times(0,T)$, then
\begin{equation}\label{intr-4p}
u(x,t_0)=O\left(|x-x_0|^k\right) \quad \forall k\in \mathbb{N}  \quad\mbox{ implies } \quad u(\cdot,t_0)=0 \mbox{ on }\Omega,
\end{equation}
(see also \cite{Lin} for weaker assumptions on the regularity). Notice that \eqref{intr-4p} trivially implies \eqref{intr-3p}.

It is rather obvious that, for the validity of \eqref{intr-3p} for $n\geq 3$, the minimum of regularity required on $A(x,t)$ with respect to space variables should be the same as the corresponding elliptic UCP. For this reason we assume that $A(x,t)$ is Lipschitz continuous with respect to the parabolic distance, namely
\begin{equation}\label{lipgen-intr}
|A(x,t)-A(y,s)|\leq M\left(|x-y|^2+|t-s|\right)^{1/2} \mbox{ for every } (x,t), (y,s)\in \mathbb{R}^{n+1},
\end{equation}
 for some given constant $M$.

As a matter of fact, under assumption \eqref{lipgen-intr}, the SUCP holds true, see \cite{AlVe}, \cite{EsFeVe}, \cite{EsFe}, \cite{Fe}, \cite{KoTa}. Some quantitative versions of this result were proved in \cite{EsFeVe}. Nevertheless, to the authors' knowledge, there is no counterexample in literature that forbid a substantial reduction of the regularity of $A(x,t)$ with respect to $t$.

\bigskip

To get closer to the main theme of the present paper, we emphasize that whenever the matrix $A(x,t)$ has a jump discontinuity at a smooth enough interface QEUC is much more interesting then UCP. In order to make this point clear let us restrict for a moment to the elliptic case. Let us consider a symmetric $n\times n$ matrix $A(x)$, $x=(x',x_n)\in\mathbb{R}^n$, whose entries have a jump discontinuity at the interface $\Gamma=\{(x',0)|x'\in\mathbb{R}^{n-1} \}$
and are Lipschitz continuous on both the sides of $\Gamma$, that is on $\mathbb{R}^n_{\pm}=\{(x',x_n)|x_n\gtrless0 \}$. We assume that $A(x)$ is also uniformly elliptic. Denote by $B_r(\widetilde{x})$ the ball centered at $\widetilde{x}$ with radius $r>0$, $B^{\pm}_r(\widetilde{x})=B_r(\widetilde{x})\cap \mathbb{R}^n_{\pm}$ and let $u$ be a weak solution to
\begin{equation}\label{intr-2}
\dive(A(x)Du)=0, \mbox{ in } B_4(0)
\end{equation}
 which satisfies
\begin{equation}\label{intr-3}
u=0, \mbox{ in } B_{1/2}(e_n),
\end{equation}
then, by the unique continuation property we have $u_{+}\equiv 0$ (here $u_{\pm}=u_{|B^{\pm}_4(0)}$) and by the homogeneous transmission conditions on  $\Gamma$ we have
\begin{equation*}
u_{-}(\cdot,0)=A_{-}(\cdot, 0)Du_{-}(\cdot,0)\cdot e_n=0 \mbox{ for } x\in \Gamma\cap B_4(0),
\end{equation*}
hence the uniqueness for Cauchy problem gives $u_{-}\equiv 0$. Now, broadly speaking, if we translate in a quantitative form the above procedure assuming that  $|u|\leq \varepsilon$  in  $B_{1/2}(e_n)$ instead of \eqref{intr-3}, and $|u|\leq 1$ in $B_{4}(0)$, we would obtain a logarithmic estimate of $u$ even in $B^{+}_r$ for $r<4$, see \cite{AV}. Clearly, this is a wrong way to perform  propagation of smallness across the interface because in this way we treat equation \eqref{intr-2} like two different equations, one in $B_4^+$ and the other one in  $B_4^-$, and the interface as part of the boundary on which only logarithmic estimates can be obtained \cite[Sect. 1.1]{ARRV}.

The right way to perform the smallness propagation estimate was provided for the first time, for isotropic coefficients (that is $A=aI$ for scalar $a$), in \cite[Sect. 3.1]{lr1} where a H\"older type propagation of  smallness   across the interface ("interpolation inequality" in the terminology of \cite{lr1}) was proved in the form of three-region inequality. This result was extended in \cite{flvw} to the case of general anisotropic and Lipschitz continuous matrix $A$. In both the papers \cite{lr1}, \cite{flvw} the three-region inequality was derived by Carleman estimates, see also \cite{Be-De-LR}, \cite{Be-De-Th} for some improvements of the results of \cite{lr1}. In \cite{ll} the Carleman estimate was proved for $A\in C^{\infty}$, whereas in \cite{dcflvw} it was proved when the matrix $A$ is Lipschitz continuous. We refer to \cite{Be-LR-2} for general results about elliptic transmission problems with complex coefficients across an interface.

Undoubtedly the investigation about Carleman estimates for the transmission problem was driven not only by its the intrinsic interest but also by the interest in the issue of exact controllability for parabolic equation and for inverse problems. Here we should mention the paper \cite{Do-Os-Pu} in which, perhaps for the first time, a Carleman estimate was proved in the parabolic context under the assumption that the leading coefficient it is independent on $t$ and it satisfies some monotonicity condition on the interface. Such a monotonicity condition was overcome in \cite{lr2}, as well as the time independence of coefficients, for $C^{\infty}$ isotropic coefficients. For some application to an inverse source parabolic transmission problem we refer to \cite{Be-Ya}.

\bigskip

The main effort in the present paper consists in proving a local Carleman estimate for the operator \eqref{intr-2p} when the matrix $A(x,t)$ have jumps at a flat interface orthogonal to the time direction and is Lipschitz continuous with respect to the parabolic distance on both the sides of the interface (see Theorem \ref{thm8.2} for the exact statement of the main result). Then in Section \ref{treregioni}, this Carleman estimate is applied to prove a H\"older type estimate for propagation of  smallness across the interface. It is easy to check that, by using standard  change of coordinates, this propagation of smallness continues to be true in the more general case of a $C^{1,1}$ interface whose normal vectors are never parallel to $t$ axis.

In order to prove our Carleman estimate (Theorem \ref{thm8.2}) we follow an approach similar to the one of \cite{dcflvw}. More precisely, in Section \ref{step1}, we consider the case in which the coefficients of the operator $\mathcal{L}$ depend only on the variable $x_n$ normal to the interface ($x_n$ is renamed $y$ in the rest of the paper). The simpler structure of operator  allows us, at a first stage, to prove a Carleman estimate (Theorem \ref{pr22}) with a weight function $\phi$ that is linear in all the space variables except the normal one. In this first step we perform the Fourier transform with respect to the tangential variables $x'$ and $t$ of the conjugate operator $e^{\tau\phi}\mathcal{L}(e^{-\tau\phi}\,\cdot\,)$ and, inspired by \cite{ll}, we factorize this conjugate operator  into two first order operators. By this approach we avoid the techniques of pseudodifferential operators (known to be "greedy" of regularity) used in \cite{ll}, and allows us to assume weak regularity on the leading matrix. Nevertheless, the Carleman estimate obtained in the first step, for the features of the level set of the weight function $\phi$, doesn't allow to find a smallness propagation estimate across the interface even in the most simple case of a constant matrix $A$. To get such a smallness propagation, say from an open set contained on one side of the interface to a set in the opposite side, we need a weight of type $-k|x'|^2+\phi$, $k>0$, or  the more comfortable  weight $-k|x'|^2-bt^2+\phi$, $k,b>0$. For this reason, in the second step of the proof of Carleman estimate (Section \ref{step2}), we consider an operator with general coefficients and a weight that is quadratic in $x^\prime$. In order to treat this general case we use a suitable partition of unity. Finally in the third step of the proof (Section \ref{step3}) we add in the weight the dependence on $t$.
Once the Carleman estimate has been proved, we show in Section \ref{treregioni} a three-region inequality. In the Appendix (Section \ref{appendice}) we prove a regularity result for the parabolic transmission problem that, although quite standard, is not present  in the literature (to the authors' knowledge).

\section{Notations and statement of the main theorem}\label{notstat}
\subsection{General notations and norms}
The functions we are interested in depend on $n$ space variables and one time variable. We denote the space variables as $(x,y)\in\RR^{n-1}\times\RR$ and the time variable as $t\in\RR$. We  assume the flat interface to be $\{y=0\}$.
Since the variable $y$, that is orthogonal to the interface, is the most important one, we use, instead of the usual notation $(x,y,t)$ the notation $(x,t,y)$. Also, the Fourier variables are denote by $(\xi,\eta,y)\in \RR^{n-1}\times\RR\times\RR$.
Sometimes we denote $X=(x,t)$  for $x\in\mathbb{R}^{n-1}$ and $t\in\mathbb{R}$. For sake of shortness we also use the notation $dX=dxdt$.

For any $r>0$ and $\widetilde{x}\in\mathbb{R}^{n-1}$, $\widetilde{t}\in\mathbb{R}$ we denote by $B_r(\widetilde{x})$ the ball centered at $\widetilde{x}$ with radius $r>0$ and we define $Q'_{r}(\widetilde{x})=\{x\in\mathbb{R}^{n-1}:|x_j-\widetilde{x}_j|\leq r,\,j=1,2,\cdots ,n-1\}$, $Q_{r}(\widetilde{X})=Q'_{r}(\widetilde{x})\times I_{r}(\widetilde{t})$, where $I_{r}(\widetilde{t})=[-r^2+\widetilde{t},r^2+\widetilde{t}]$.
Whenever $x=0$ we denote $B_r=B_r(0)$, $Q'_{r}=Q'_{r}(0)$ and $Q_{r}=Q_{r}(0)$.

By $H_{\pm}=\chi_{\RR^{n-1}\times\RR\times\RR_{\pm}}$ we denote the characteristic function of $\RR^{n-1}\times\RR\times\RR_{\pm}=\{(x,t,y)\in \RR^{n-1}\times\RR\times\RR\,|\,y\gtrless0\}$.

In places we use equivalently the symbols $D$, $\partial$
to denote the gradient of a function and we add the index $x$ or $y$
to denote the gradient in $\mathbb R^{n-1}$ and the derivative with respect to $y$,
respectively.

Let $u_\pm\in C^\infty(\RR^{n-1}\times\RR\times\RR)$; we define
\begin{equation*}
u=H_+u_++H_-u_-=\sum_\pm H_{\pm}u_{\pm},
\end{equation*}
where $\sum_\pm a_\pm=a_++a_-$.

For a function $f\in L^2(\mathbb{R}^{n})$, we define
\begin{equation*}
\mathcal{F}(f)=\hat{f}(\xi,\eta)=\int_{\mathbb{R}^{n}}f(x,t)e^{-ix\cdot\xi-it\eta}\,dx\,dt,\quad \xi\in \mathbb{R}^{n-1},\eta\in\RR.
\end{equation*}
Moreover we define (\cite{LiMa})
$$[f]_{\frac{1}{2},0,\mathbb{R}^{n}}=\left[\int_{\mathbb{R}}\int_{\mathbb{R}^{n-1}}\int_{\mathbb{R}^{n-1}}\frac{|f(x_1,t)-f(x_2,t)|^2}{|x_1-x_2|^n}dx_1dx_2dt\right]^{1/2},$$
and, for $\alpha\in[0,1)$,
$$[f]_{0,\alpha,\mathbb{R}^{n}}=\left[\int_{\mathbb{R}^{n-1}}\int_{\mathbb{R}}\int_{\mathbb{R}}\frac{|f(x,t_1)-f(x,t_2)|^2}
{|t_1-t_2|^{1+2\alpha}}dt_1dt_2dx\right]^{1/2}.$$

As usual we denote by $H^{\frac{1}{2},0}(\mathbb{R}^{n})$ and $H^{0,\alpha}(\mathbb{R}^{n}$ the spaces of the functions $f\in L^2(\mathbb{R}^{n})$ satisfying, respectively,
\[[f]_{\frac{1}{2},0,\mathbb{R}^{n}}<\infty,\]
and
\[[f]_{0,\alpha,\mathbb{R}^{n}}<\infty,\]
with the norms
\begin{equation}
||f||_{H^{\frac{1}{2},0}(\mathbb{R}^{n})}=||f||_{L^2(\mathbb{R}^{n})}+[f]_{\frac{1}{2},0,\mathbb{R}^{n}}
\end{equation}
and
\begin{equation}
||f||_{H^{0,\alpha}(\mathbb{R}^{n})}=||f||_{L^2(\mathbb{R}^{n})}+[f]_{0,\alpha,\mathbb{R}^{n}}.
\end{equation}

We also use the notation $H^{\frac{1}{2},\alpha}(\mathbb{R}^{n})=H^{\frac{1}{2},0}(\mathbb{R}^{n})\cap H^{0,\alpha}(\mathbb{R}^{n})$ where we consider  the norm $||f||_{H^{\frac{1}{2},\alpha}(\mathbb{R}^{n})}=||f||_{L^2(\mathbb{R}^{n})}+[f]_{\frac{1}{2},\alpha,\mathbb{R}^{n}}$ where $[f]_{\frac{1}{2},\alpha,\mathbb{R}^{n}}=[f]_{\frac{1}{2},0,\mathbb{R}^{n}}+[f]_{0,\alpha,\mathbb{R}^{n}}$.

Recall that there exist $C_n$, depending only on $n$, and $C_{n,\alpha}$, depending only on $n$ and $\alpha$, such that

\begin{equation}\label{nor1}
C_n^{-1}\int_{\mathbb{R}^{n}}|\xi||\hat{f}(\xi,\eta)|^2d\xi d\eta\leq[f]^2_{\frac{1}{2},0,\mathbb{R}^{n}}\leq C_n\int_{\mathbb{R}^{n}}|\xi||\hat{f}(\xi,\eta)|^2d\xi d\eta,
\end{equation}

\begin{equation}\label{nor2}
C_{n,\alpha}^{-1}\int_{\mathbb{R}^{n}}|\eta|^{2\alpha}|\hat{f}(\xi,\eta)|^2d\xi d\eta\leq[f]^2_{0,\alpha,\mathbb{R}^{n}}\leq C_{n,\alpha}\int_{\mathbb{R}^{n}}|\eta|^{2\alpha}|\hat{f}(\xi,\eta)|^2d\xi d\eta,
\end{equation}
so that the norm $||f||_{H^{\frac{1}{2},\alpha}(\mathbb{R}^{n})}$ is equivalent to
$\left(\int_{\mathbb{R}^{n}}(1+|\eta|^{\alpha}+|\xi|)|\hat{f}(\xi,\eta)|^2d\xi d\eta\right)^{1/2}$.

Moreover,
\begin{equation}\label{disug-seminorme}
\begin{aligned}
&[D_xf]^2_{0,\frac{1}{4},\mathbb{R}^{n}}\leq C_{n,\frac{1}{4}}\int_{\mathbb{R}^{n}}|\eta|^{1/2}|\xi|^2|\hat{f}(\xi,\eta)|^2d\xi d\eta\\
\leq & C_{n,\frac{1}{4}}\int_{\mathbb{R}^{n}}(\frac{1}{3}|\eta|^{3/2}+\frac{2}{3}|\xi|^{3/2})|\hat{f}(\xi,\eta)|^2d\xi d\eta\leq C\left([f]^2_{0,\frac{3}{4},\mathbb{R}^{n}}+[D_xf]^2_{\frac{1}{2},0,\mathbb{R}^{n}}\right).
\end{aligned}
\end{equation}
\subsection{Differential operator, weight and trace operators}
Let us define
\begin{equation}\label{Lgen}
\mathcal{L}(u):=\sum_{\pm}H_{\pm}\left(-\der_t u_\pm+\dive_{x,y}(A_{\pm}(x,t,y)D_{x,y}u_{\pm})\right),
\end{equation}
where
\begin{equation}\label{Agen}
A_{\pm}(x,t,y)=\{a^{\pm}_{ij}(x,t,y)\}^n_{i,j=1},\quad x\in \mathbb{R}^{n-1},t\in \mathbb{R},y\in \mathbb{R}
\end{equation}
are Lipschitz symmetric matrix-valued functions with real entries that satisfy, for given constants $\lambda_0\in (0,1]$ and $M_0>0$,
\begin{equation}\label{ellgen}
\lambda_0|z|^2\leq A_{\pm}(x,t,y)z\cdot z\leq \lambda^{-1}_0|z|^2,\, \quad \forall (x,t,y)\in \mathbb{R}^{n+1},\,\forall\, z\in \mathbb{R}^n
\end{equation}
and
\begin{equation}\label{lipgen}
|A_{\pm}(x',t',y')-A_{\pm}(x,t,y)|\leq M_0\left(|x'-x|^2+|y'-y|^2+|t'-t|\right)^{1/2}.
\end{equation}

We also define
\begin{equation}\label{7.5}
h_0(u):=u_+(x,t,0)-u_-(x,t,0), \ \quad\quad  \forall (x,t)\in \mathbb{R}^{n}
\end{equation}
\begin{equation}\label{7.6}
h_1(u):=A_+(x,t,0)D_{x,y}u_+(x,t,0)\cdot \nu-A_-(x,t,0)D_{x,y}u_-(x,t,0)\cdot \nu,\ \quad \forall\, (x,t)\in \mathbb{R}^{n},
\end{equation}
where $\nu=-e_n$.

\begin{rem}
Notice that, if $A(x,t,y)=\sum_\pm A_\pm(x,t,y)$ and $u$ is a solution to the equation
\[\der_tu=\dive_{x,y}\left(A(x,t,y)\nabla u\right)\mbox{ in }\mathbb{R}^{n-1}\times\mathbb{R}\times\mathbb{R},\]
then
\begin{equation*}
\left\{\begin{array}{rcl}\mathcal{L}(u)&=&0 \mbox{ in }\mathbb{R}^{n-1}\times\mathbb{R}\times\mathbb{R},\\
h_0(u)&=&0\mbox{ in }\mathbb{R}^{n-1}\times\mathbb{R},\\
h_1(u)&=&0\mbox{ in }\mathbb{R}^{n-1}\times\mathbb{R}.\end{array}
\right.\end{equation*}
\end{rem}

Let us now introduce  the weight function. Let $\varphi$ be
\begin{equation}\label{2.1}
\varphi(y)=
\begin{cases}
\begin{array}{l}
\varphi_+(y):=\alpha_+y+\beta y^2/2,\quad y\geq 0,\\
\varphi_-(y):=\alpha_-y+\beta y^2/2,\quad y< 0,
\end{array}
\end{cases}
\end{equation}
where $\alpha_+$, $\alpha_-$ and $\beta$ are positive numbers which will be determined later. In what follows we denote by $\varphi_{+}$ and $\varphi_{-}$ the restriction of the weight function $\varphi$ to $[0,+\infty)$ and to $(-\infty,0)$ respectively. We use similar notation for any other weight functions. 
Let, for $\delta>0$,
\begin{equation}\label{wei}
\phi_{\delta}(x,y):=\varphi(\delta^{-1}y)-\frac{|x|^2}{2\delta}.
\end{equation}

The more general weight function that we will use is of the form
\[
\Phi(x,t,y)=\phi_{\delta}(x,y)-\frac{b}{2}t^2
\]
for some positive $b$.

Let us introduce a notation for the  operators on the interface that appears in the Carleman estimates:
\begin{eqnarray}\label{BL}
&Y^L(u;\Phi)=\sum_{\pm}\left\{\sum_{k=0}^1\tau^{3-2k}\int_{\mathbb{R}^{n}}|D_{x,y}^k\left({u}_{\pm}e^{\tau\Phi}\right)(x,t,0)|^2dX
+\tau^2[(e^{\tau\Phi}u_{\pm})(\cdot,\cdot,0)]^2_{\frac{1}{2},\frac{1}{4},\mathbb{R}^{n}}\right.\nonumber\\
&+\left.\vphantom{\sum_\pm}[(e^{\tau\Phi}u_{\pm})(\cdot,\cdot,0)]^2_{0,\frac{3}{4},\mathbb{R}^{n}}
+[\partial_y(e^{\tau\Phi_\pm}u_{\pm})(\cdot,\cdot,0)]^2_{\frac{1}{2},\frac{1}{4},\mathbb{R}^{n}}
+[D_x(e^{\tau\Phi_\pm}u_{\pm})(\cdot,\cdot,0)]^2_{\frac{1}{2},0,\mathbb{R}^{n}}\right\},
\end{eqnarray}

\medskip

\begin{eqnarray}
\label{BR}
Y^R(u;\Phi)=&[e^{\tau\Phi(\cdot,0)}h_1(u)]^2_{\frac{1}{2},\frac{1}{4},\mathbb{R}^{n}}
+[e^{\tau\Phi(\cdot,0)}h_0(u)]^2_{0,\frac{3}{4},\mathbb{R}^{n}}+[D_x(e^{\tau\Phi(\cdot,0)}h_0(u))]^2_{\frac{1}{2},0,\mathbb{R}^{n}}\nonumber\\
&\ds{+\frac{\tau^{3}}{\delta^3}\int_{\mathbb{R}^{n}}|h_0(u)|^2e^{2\tau\Phi(x,t,0)}dX
+\frac{\tau}{\delta}\int_{\mathbb{R}^{n}}|h_1(u)|^2e^{2\tau\Phi(x,t,0)}dX},
\end{eqnarray}
where $h_0(u)$ and $h_1(u)$ are defined as in \eqref{7.5} and \eqref{7.6}.

We use the letters $C, C_0, C_1, \cdots$ to denote constants. The value of the constants may change from line to line, but it is always greater than $1$.

\bigskip
Let us now state our main result.
\begin{theo}\label{thm8.2}
Let $A_{\pm}(x,t,y)$ satisfy \eqref{Agen}-\eqref{lipgen}. There exist $\alpha_+$, $\alpha_-$, $\beta$, $b$, $\delta_0$, $r_0$, $\tau_0$ and $C$ depending on $\lambda_0, M_0$ such that if $\tau\geq \tau_0$, then
\begin{equation}\label{8.24}
\begin{aligned}
\sum_{\pm}&\left(\sum_{k=0}^2\tau^{3-2k}\int_{\mathbb{R}^{n+1}_{\pm}}|D_{x,y}^k{u}_{\pm}|^2e^{2\tau\Phi}dXdy+
\tau^{-1}
\int_{\mathbb{R}^{n+1}_{\pm}}|\partial_t{u}_{\pm}|^2e^{2\tau\Phi}dXdy\right)
+Y^L(u;\Phi)\\
&\leq C\left(\int_{\mathbb{R}^{n+1}}|\mathcal{L}(u)|^2\,e^{2\tau\Phi}dXdy+
Y^R(u;\Phi)\right).
\end{aligned}
\end{equation}
where $u=H_+u_++H_-u_-$,  $u_{\pm}\in C^\infty(\mathbb{R}^{n+1})$ and ${\rm supp}\, u\subset Q_{\delta_0/2}\times(-\delta_0 r_0,\delta_0 r_0)$, $\Phi(x,t,y)=\phi_{\delta_0}(x,y)-\frac{b}{2}t^2$ with $\phi_{\delta_0}$ given by \eqref{wei}, and $Y^L(u,\Phi)$ and $Y^R(u,\Phi)$ are defined in \eqref{BL} and \eqref{BR}, respectively.
\end{theo}
\begin{rem}\label{bassi}
It is clear that estimate \eqref{8.24} remains true for large enough $\tau_0$ if the operator $\mathcal{L}$ is substituted by
\[\tilde{\mathcal{L}}(u)=\mathcal{L}(u) + W\cdot \sum_\pm D_{x,y}u_\pm + Vu,\]
for bounded functions $W$ and $V$.
\end{rem}
We divide the proof of Theorem \ref{thm8.2} in 3 main steps. In Step 1, we consider a leading coefficient $A_\pm$ depending only on $y$ and a weight function linear in $x$ and independent of $t$. In Step 2, we take a general leading coefficient and a weight quadratic in $x$ but independent of $t$. In Step 3 we add the dependence on $t$ of the weight function.

\section{Step 1: leading coefficient depending on $y$ only}\label{step1}
In this section we consider the simple case of the leading matrices \eqref{Agen} independent of $x$ and $t$. Moreover, the weight function that we consider is linear with respect to $x$ variable, so that, as explained above, the Carleman estimates we get here are only preliminary to the one that we will get in the general case.

Assume that
\begin{equation*}
A_{\pm}(y)=\{a^{\pm}_{ij}(y)\}^n_{i,j=1}
\end{equation*}
are symmetric matrix-valued functions satisfying \eqref{ellgen} and \eqref{lipgen}, i.e.,
\begin{equation}\label{1.2}
\lambda_0|z|^2\leq A_{\pm}(y)z\cdot z\leq \lambda^{-1}_0|z|^2,\, \forall y\in \RR,\,\forall\, z\in \RR^n,
\end{equation}
\begin{equation}\label{1.3}
|A_{\pm}(y')-A_{\pm}(y'')|\leq M_0|y'-y''|, \quad\forall\, y',y''\in \RR.
\end{equation}
From \eqref{1.2}, we have
\begin{equation*}
a_{nn}^{\pm}(y)\geq\lambda_0\quad\forall\, y\in\RR.
\end{equation*}

In the present case the differential operator \eqref{Lgen} becomes
\begin{equation}\label{1.6}
\mathcal{L}(u):=\sum_{\pm}H_{\pm}\left(-\der_tu_\pm+{\rm div}_{x,y}(A_{\pm}(y)D_{x,y}u_{\pm})\right),
\end{equation}
and \eqref{7.6} becomes
\begin{equation*}
h_1(w):=A_+(0)D_{x,y}w_+(x,t,0)\cdot \nu-A_-(0)D_{x,y}w_-(x,t,0)\cdot \nu,\ \forall\, (x,t)\in \mathbb{R}^{n},
\end{equation*}
where $u=\sum_\pm H_{\pm}u_{\pm}$, $u_\pm\in C^\infty(\RR^{n+1})$ and $\nu=-e_n$.

We also set, for any $s\in[0,1]$ and $\gamma\in\mathbb{R}^{n-1}$ with $|\gamma|\le 1$
\begin{equation}\label{2.2}
\phi(x,y)=\varphi(y)+ s\gamma\cdot x=H_+\phi_++H_-\phi_-,
\end{equation}
where $\varphi$ is defined in \eqref{2.1}. Notice that this weight function is linear in $x$, hence this is not a special case of \eqref{wei}.

Our aim in this step is to prove:
\begin{theo}\label{pr22}
Let $\mathcal{L}$ be the operator \eqref{1.6}. There exist $\alpha_\pm$, $\beta$, $\tau_0$, $s_0$, $r_0$ and $C$ depending only on $\lambda_0$, $M_0$, such that for $\tau\geq\tau_0$, $0<s\le s_0<1$,
and for every  $w=\sum_{\pm}H_\pm w_\pm$ with ${\rm supp}\, w\subset Q_1\times [-r_0,r_0]$, we have that
\begin{equation}\label{5.16}
\begin{aligned}
\sum_{\pm}&\left(\sum_{k=0}^2\tau^{3-2k}\int_{\mathbb{R}^{n+1}_{\pm}}|D_{x,y}^k{w}_{\pm}|^2e^{2\tau\phi}dXdy+\tau^{-1}\int_{\mathbb{R}^{n+1}_{\pm}}|\der_t{w}_{\pm}|^2e^{2\tau\phi}dXdy\right)
+Y^L(w;\phi)\\
&\leq C\left(\int_{\mathbb{R}^{n+1}}|\mathcal{L}(w)|^2e^{2\tau\phi}dXdy
+Y^R(w;\phi)\right),
\end{aligned}
\end{equation}
where $\phi(x,y)$ is given by \eqref{2.2} and $Y^L$ and $Y^R$ where defined in \eqref{BL} and \eqref{BR}, respectively.
\end{theo}

\subsection{Fourier transform of the conjugate operator and its factorization}\label{Fourier}
To proceed further, we introduce some operators and find their properties. We use the notation $\partial_j=\partial_{x_j}$ for $1\leq j\leq n-1$.

Let us define
\begin{equation}\label{bnuovi}
b_{jk}^{\pm}(y)=\frac{1}{a_{nn}^\pm}\left(a_{jk}^{\pm}(y)-\frac{a_{nj}^{\pm}(y)a_{nk}^{\pm}(y)}{a_{nn}^{\pm}(y)}\right),\quad j,k=1,\cdots,n-1.
\end{equation}
and set
\[B_{\pm}(y)=\{b^{\pm}_{jk}(y)\}^{n-1}_{j,k=1}.\]
Let us define the operator
\begin{equation*}
\widetilde{T}_{\pm}(y,\partial_x)u_{\pm}=\sum_{j=1}^{n-1}\frac{a_{nj}^{\pm}(y)}{a_{nn}^{\pm}(y)}\,\partial_ju_{\pm}.
\end{equation*}

In view of \eqref{1.2} we have
\begin{equation}\label{1.15}
\lambda_1|z'|^2\leq B_{\pm}(y)z'\cdot z'\leq \lambda^{-1}_1|z'|^2,\, \quad\forall\, y\in \RR,\, \forall\, z'\in \RR^{n-1},
\end{equation}
\begin{equation}\label{1.15bis}
\left|B_\pm(y^\prime)-B_\pm(y^{\prime\prime})\right|\leq M_1|y^\prime-y^{\prime\prime}|,\,\forall y^\prime,\,y^{\prime\prime}\in\RR,
\end{equation}
where $\lambda_1\leq\lambda_0$  depends only on $\lambda_0$, and $M_1$ depends on $\lambda_0$ and $M_0$ only.

It is easy to show, by direct calculations (\cite{ll}), that
\begin{equation*}
{\rm div}_{x,y}\big(A_{\pm}(y)D_{x,y}u_{\pm}\big)=(\partial_y+\widetilde{T}_{\pm})a_{nn}^{\pm}(y)(\partial_y+\widetilde{T}_{\pm})u_{\pm}+a_{nn}^\pm(y)\,{\rm div}_{x}\big(B_{\pm}(y)D_{x}u_{\pm}\big).
\end{equation*}

 In order to derive the Carleman estimate \eqref{5.16} we  conjugate the operator $\mathcal{L}$ with $e^{\tau\phi}$ for $\phi$ given by \eqref{2.2} and get (see \cite{dcflvw} for further details)

\begin{eqnarray}\label{2.9}
&&e^{\tau\phi}\mathcal{L}(y,\partial)(e^{-\tau\phi}v)=\sum_{\pm}H_{\pm}\left\{\vphantom{\sum_{j,k=1}^{n-1}}-\der_tv_\pm\right.\nonumber\\
&&\quad+
\left(\partial_y-\tau\varphi_{\pm}'+\widetilde{T}_{\pm}(y,\partial_x-\tau s\gamma)\right)a_{nn}^{\pm}(y)\left(\partial_y-\tau\varphi_{\pm}'+\widetilde{T}_{\pm}(y,\partial_x-\tau s\gamma)\right)v_{\pm}\nonumber\\
&&\quad+\left.
a_{nn}^\pm(y)\left(\sum_{j,k=1}^{n-1}b_{jk}^{\pm}(y)\partial^2_{jk}v_{\pm}-2s\tau\sum_{j,k=1}^{n-1}b_{jk}^{\pm}(y)\partial_{j}v_{\pm}\gamma_k+s^2\tau^2\sum_{j,k=1}^{n-1}b_{jk}^{\pm}(y)\gamma_j\gamma_kv_{\pm}\right)\right\}.
\end{eqnarray}

Now, we  focus on the analysis of $e^{\tau\phi}\mathcal{L}(e^{-\tau\phi}v)$ and introduce some notations:


\begin{equation*}
\B_{\pm}(\xi,\gamma,y)=\sum_{j,k=1}^{n-1}b_{jk}^{\pm}(y)\xi_j\gamma_k,\quad \xi\in \mathbb{R}^{n-1},
\end{equation*}
for $b_{jk}^{\pm}$ as in \eqref{bnuovi},
\begin{equation}\label{2.14}
\zeta_{s,\pm}(\xi,\eta,y)=\B_{\pm}(\xi,\xi,y)-s^2\tau^2\B_{\pm}(\gamma,\gamma,y)+i\left(2s\tau \B_{\pm}(\xi,\gamma,y)+\frac{\eta}{a_{nn}^\pm(y)}\right),
\end{equation}
\begin{equation*}
	R_{s,\pm}(\xi,\eta,y)=\Re\sqrt{\zeta_{s,\pm}(\xi,\eta,y)},\quad J_{s,\pm}(\xi,\eta,y)=\Im\sqrt{\zeta_{s,\pm}(\xi,\eta,y)}
\end{equation*}
\begin{equation}\label{L}
	L=\sup	\frac{R_{0,+}(\xi,\eta,0)}{R_{0,-}(\xi,\eta,0)}
\end{equation}
and
\begin{equation}\label{2.12}
T_{\pm}(\xi,y)=\sum_{j=1}^{n-1}\frac{a_{nj}^{\pm}(y)}{a_{nn}^{\pm}(y)}\xi_j.
\end{equation}
By \eqref{2.9}, we have
\begin{equation*}
\mathcal{F}\left(e^{\tau\phi}\mathcal{L}(e^{-\tau\phi}v)\right)=\sum_{\pm}H_{\pm}P_{s,\pm}\hat{v}_{\pm},
\end{equation*}
where $\mathcal{F}(\cdot)$ is the Fourier transform with respect to $(x,t)$ and
\begin{equation*}
\begin{aligned}
P_{s,\pm}\hat{v}_\pm:=&\big(\partial_y-\tau\varphi_{\pm}^\prime+iT_{\pm}(\xi+i\tau s\gamma,y)\big)a_{nn}^{\pm}(y)\big(\partial_y-\tau\varphi_{\pm}^\prime+iT_{\pm}(\xi+i\tau s\gamma,y)\big)\hat{v}_{\pm}\\
&-a_{nn}^{\pm}(y)\zeta_{s,\pm}(\xi,\eta,y)\hat{v}_{\pm}.
\end{aligned}
\end{equation*}

Our aim is to estimate $\mathcal{F}\left(e^{\tau\phi}\mathcal{L}(e^{-\tau\phi}v)\right)$ from below.  
In order to do this, we want to factorize the operators $P_{s,\pm}$.

For any $z=a+ib$ with $(a,b)\neq(0,0)$, we define the square root of $z$,
$$\sqrt{z}=\sqrt{\frac{a+\sqrt{a^2+b^2}}{2}}+i\frac{b}{\sqrt{2(a+\sqrt{a^2+b^2})}}.$$
It should be noted that $\Re \sqrt{z}\geq 0$.

Let us set

\begin{equation}\label{sigma0}
\sigma_0(\xi,\eta)=\mathcal{F}\left(e^{\tau\phi(\cdot,0)}h_0(e^{-\tau\phi(\cdot,0)}v)\right)=\hat{v}_+(\xi,\eta,0)-\hat{v}_-(\xi,\eta,0),
\end{equation}
and
\begin{equation}\label{sigma1}
\sigma_1(\xi,\eta)=\mathcal{F}\left(-e^{\tau\phi(\cdot,0)}h_1(e^{-\tau\phi(\cdot,0)}v)\right)=V_+(\xi,\eta)-V_-(\xi,\eta),
\end{equation}
where
\begin{equation}\label{Vpm}
V_\pm(\xi,\eta)=a_{nn}^{\pm}(0)\left[\der_y\hat{v}_{\pm}(\xi,\eta,0)-\tau\alpha_\pm\hat{v}_\pm(\xi,\eta,0)+iT_\pm(\xi+i\tau s\gamma,0)\hat{v}_\pm(\xi,\eta,0)\right].
\end{equation}

We denote by
\begin{equation*}
	\Lambda_0=(|\xi|^2+|\eta|)^{1/2}
\end{equation*}
and
\begin{equation*}
	\Lambda_1=(|\xi|^2+|\eta|+\tau^2)^{1/2}
\end{equation*}

Notice that, since $|\gamma|\leq 1$ and $0\leq s\leq 1$, we have:
\begin{equation}\label{1-15bis}
	|\zeta_{s,\pm}|\leq C\Lambda_1^2,
\end{equation}
\begin{equation}\label{2-15bis}
	|R_{s,\pm}|, |J_{s,\pm}|\leq C\Lambda_1,
\end{equation}
\begin{equation}\label{derivt}
	|T_{\pm}(y,\xi+i\tau s\gamma)|\leq C(|\xi|+s\tau),
\end{equation}
where $C$ depends only on  $\lambda_0$.

Moreover, for $s=0$ we have,
\begin{equation}\label{stimeRzer}
	\frac{\lambda_0^{1/2}\lambda_1^{1/2}\Lambda_0}{2}\leq R_{0,\pm}\leq\lambda_0^{-1/2}\lambda_1^{-1/2}\Lambda_0.
\end{equation}

We always assume that the constants $\alpha_+$ and $\alpha_-$ in the weight \eqref{2.1} are fixed in such a way that
\begin{equation}\label{ipalpha}
	\frac{\alpha_+}{\alpha_-}>L,
\end{equation}
where $L$ was given in \eqref{L}.

\begin{pr}\label{prop3.1}
There exist $\tau_0$, $\overline{s}$, $\rho$, $\beta$ and $C$, depending only on $\lambda_0$ and $M_0$ such that for $\tau\geq\tau_0$, ${\rm supp}\, \hat{v}_{\pm}(\xi,\cdot)\subset [-\rho,\rho]$, $s\leq \overline{s}<1$ we have:
\begin{eqnarray}\label{4.4}
&&\frac{1}{\tau}\sum_{\pm}||\partial^2_y\hat{v}_{\pm}(\xi,\eta,\cdot)||^2_{L^2(\mathbb{R}_\pm)}
+\frac{\Lambda_1^2}{\tau}\sum_{\pm}||\partial_y\hat{v}_{\pm}(\xi,\eta,\cdot)||^2_{L^2(\mathbb{R}_\pm)}
\nonumber\\
&&\quad+\frac{\Lambda_1^4}{\tau}\sum_{\pm}||\hat{v}_{\pm}(\xi,\eta,\cdot)||^2_{L^2(\mathbb{R}_\pm)}+\Lambda_1\sum_{\pm}|V_{\pm}(\xi,\eta)|^2+\Lambda_1^3\sum_{\pm}|\hat{v}_{\pm}(\xi,\eta,0)|^2\nonumber\\
 &&\quad\leq C\left(\sum_{\pm}||P_{s,\pm}\hat{v}_{\pm}(\xi,\eta,\cdot)||^2_{L^2(\mathbb{R}_\pm)}+\Lambda_1|\sigma_1(\xi,\eta)|^2+\Lambda_1^3|\sigma_0(\xi,\eta)|^2\right).
\end{eqnarray}

\end{pr}
\textit{Proof of Theorem \ref{pr22}.}
We integrate \eqref{4.4} with respect to $\xi$ and $\eta$.
Since $\Lambda_1$ is positive and 1-homogeneous with respect to $|\xi|$, $\tau$ and $|\eta|^{1/2}$, $\Lambda_1$ and its the powers appearing in \eqref{4.4} can be bounded from below and  above by polynomials with the same degree (with respect to  $|\xi|$, $\tau$ and $|\eta|^{1/2}$). We can then choose the suitable polynomials that, thanks to   \eqref{nor1} and \eqref{nor2}, give \eqref{5.16}.
\qed

\subsection{Proof of Proposition \ref{prop3.1}}
Let us define two operators
\begin{equation*}
E_{s,\pm}=\der_y+iT_{\pm}(\xi+i\tau s\gamma,y)-\left(\tau\varphi^\prime_{\pm}(y)+\sqrt{\zeta_{s,\pm}(\xi,\eta,y)}\right),
\end{equation*}
\begin{equation*}
F_{s,\pm}=\der_y+iT_{\pm}(\xi+i\tau s\gamma,y)-\left(\tau\varphi^\prime_{\pm}(y)-\sqrt{\zeta_{s,\pm}(\xi,\eta,y)}\right).
\end{equation*}
With all the definitions given above, we thus obtain that
\begin{equation}\label{3.12}
P_{s,+}\hat{v}_+=E_{s,+}a_{nn}^+(y)F_{s,+}\hat{v}_+-\hat{v}_+\der_y\left(a_{nn}^+(y)\sqrt{\zeta_{s,+}(\xi,\eta,y)}\right),
\end{equation}
\begin{equation}\label{3.13}
P_{s,-}\hat{v}_-=F_{s,-}a_{nn}^-(y)E_{s,-}\hat{v}_-+\hat{v}_-\partial_y\left(a_{nn}^-(y)\sqrt{\zeta_{s,-}(\xi,\eta,y)}\right).
\end{equation}

Similarly to the elliptic case, we distinguish three cases:
\begin{eqnarray*}
	1^{st}\mbox{ case}&\tau\geq \frac{\Lambda_0}{s_0}& \\
		2^{nd}\mbox{ case}& \frac{R_{0,+}(\xi,\eta,0)}{(1-\kappa)\alpha_+}\leq\tau\leq \frac{\Lambda_0}{s_0}&\\
			3^{rd}\mbox{ case}& \tau\leq\frac{R_{0,+}(\xi,\eta,0)}{(1-\kappa)\alpha_+}&
			\end{eqnarray*}
where $s_0$ will be chosen later and
\begin{equation*}
	\kappa=\frac{1}{2}\left(1-L\frac{\alpha_-}{\alpha_+}\right),
\end{equation*}
notice that by \eqref{ipalpha} we have $\kappa>0$.
\subsubsection{First case }\label{firstcase}
In this case we assume
\begin{equation}\label{case1}
	\tau\geq \frac{\Lambda_0}{s_0}.
\end{equation}
\begin{pr}
There exist a constant $C$ depending on $\lambda_0$ and $M_0$ such that
\begin{equation}\label{1-8p}
\left|\left(P_{s,\pm}-P_{0,\pm}\right)\hat{v}_\pm\right|	\leq Cs\Lambda_1\left(|\der_y \hat{v}_\pm|+\Lambda_1|\hat{v}_\pm|\right),
\mbox{ for }|y|\leq 1/\beta,
\end{equation}
\begin{equation}\label{2-8p}
	\left|P_{0,+}\hat{v}_+-E_{0,+}a_{nn}^+(y)F_{0,+}\hat{v}_+\right|\leq C\Lambda_0|\hat{v}_+|,
\end{equation}
\begin{equation}\label{3-8p}
	\left|P_{0,-}\hat{v}_--F_{0,-}a_{nn}^-(y)E_{0,-}\hat{v}_-\right|\leq C\Lambda_0|\hat{v}_-|.
\end{equation}
\end{pr}
\begin{proof}
The proof of estimate \eqref{1-8p} follows the same lines of the proof of Lemma 3.2 in \cite{dcflvw}.
Estimates \eqref{2-8p} and \eqref{3-8p} easily follow from \eqref{stimeRzer}, \eqref{1.2} and \eqref{1.15bis}.
\end{proof}

\begin{lemma}\label{lem4.2}
Let $\tau\geq1$ and assume \eqref{case1}. There exists a positive constant $C$ depending only on $\lambda_0$ and $M_0$ such that,
if $0<s\leq s_0\leq 1/C$, we have
\begin{eqnarray}\label{4.19}
&&\Lambda_1|a_{nn}^+(0)F_{0,+}\hat{v}_+(\xi,\eta,0)|^2+\Lambda_1^3|\hat{v}_+(\xi,\eta,0)|^2+\Lambda_1^4||\hat{v}_+(\xi,\eta,\cdot)||^2_{L^2(\mathbb{R}_+)}+\Lambda_1^2||\partial_y\hat{v}_+(\xi,\eta,\cdot)||^2_{L^2(\mathbb{R}_+)}\nonumber\\
&&\leq C ||P_{s,+}\hat{v}_{+}(\xi,\eta,\cdot)||^2_{L^2(\mathbb{R}_+)}
\end{eqnarray}
and
\begin{eqnarray}\label{4.20}
&&-\Lambda_1|a_{nn}^-(0)E_{0,-}\hat{v}_-(\xi,\eta,0)|^2-\Lambda_1^3|\hat{v}_-(\xi,\eta,0)|^2+\Lambda_1^4||\hat{v}_-(\xi,\eta,\cdot)||^2_{L^2(\mathbb{R}_-)}+\Lambda_1^2||\partial_y\hat{v}_-(\xi,\eta,\cdot)||^2_{L^2(\mathbb{R}_-)}\nonumber\\
&&\quad\leq C ||P_{s,-}\hat{v}_{-}(\xi,\eta,\cdot)||^2_{L^2(\mathbb{R}_-)},
\end{eqnarray}
for ${\rm supp}(\hat{v}_-(\xi,\cdot))\subset[-\frac{1}{C\beta},\frac{1}{C\beta}]$.
\end{lemma}
\begin{proof}
Define
\begin{equation}\label{omega0}
\omega_{0,+}(\xi,\eta,y)=a_{nn}^+(y)F_{0,+}\hat{v}_+(\xi,\eta,y),\quad	\omega_{0,-}(\xi,\eta,y)=a_{nn}^-(y)E_{0,-}\hat{v}_-(\xi,\eta,y).
\end{equation}

We have
\begin{eqnarray}\label{4.21}
&&\int_0^\infty\left|E_{0,+}\omega_{0,+}\right|^2dy=\int_0^\infty|(\partial_y+i(T_{+}+J_{0,+})\omega_{0,+}|^2dy
+\int_0^\infty\left(\tau\alpha_{+}+\tau\beta y+R_{0,+}\right)^2|\omega_{0,+}|^2dy\nonumber\\
&&\qquad\qquad-2\Re \int_0^\infty\left(\tau\alpha_{+}+\tau\beta y+R_{0,+}\right)\overline{\omega}_{0,+}\left(\left(\partial_y+i(T_{+}+J_{0,+}\right)\omega_{0,+}\right)dy\nonumber\\
&&\qquad\geq\int_0^\infty\left(\tau\alpha_{+}+\tau\beta y+R_{0,+}\right)^2|\omega_{0,+}|^2dy-
\int_0^\infty\left(\tau\alpha_{+}+\tau\beta y+R_{0,+}\right)\partial_y\left|\omega_{0,+}\right|^2dy\nonumber\\
&&\qquad= \int_0^\infty\left[\left(\tau\alpha_{+}+\tau\beta y+R_{0,+}\right)^2+(\tau\beta+\partial_yR_{0,+})\right]|\omega_{0,+}|^2dy
+\left(|\omega_{0,+}|^2(\tau\alpha_++R_{0,+})\right)_{|_{y=0}}
\end{eqnarray}
where we omit the arguments $(\xi,\eta,y)$ for sake of shortness.
Since, by \eqref{1.3} and \eqref{stimeRzer}, there exists a constant $C_0$, depeding only on $\lambda_0$ and $M_0$, such that
\begin{equation}\label{stimeRzero}
	 C_0^{-1}\Lambda_0\leq R_{0,+}\leq C_0\Lambda_0,\quad \left|\partial_y R_{0,+}\right|\leq C_0\Lambda_0,
\end{equation}
and by \eqref{case1} (recalling that $y>0$) we can estimate
\begin{eqnarray}\label{4.23}
(\tau\alpha_++\tau\beta y+R_{0,+})^2+(\tau\beta y+\partial_yR_{0,+})&\geq&
\frac{1}{2}(\tau^2\alpha_+^2+C_0^{-2}\Lambda_0^2)+\tau\beta-C_0\Lambda_0\nonumber\\
&\geq& \frac{1}{2}\min\{\alpha_+^2,C_0^{-2}\}(\tau^2+\Lambda_0^2)+\tau\beta-C_0s_0\tau\nonumber\\
&\geq& \frac{1}{2}\min\{\alpha_+^2,C_0^{-2}\}(\tau^2+\Lambda_0^2)+\frac{\tau\beta}{2}.
\end{eqnarray}
provided

\begin{equation}\label{szero}
s_0\leq \frac{\beta}{2C_0}.\end{equation}

Combining \eqref{4.21}, \eqref{stimeRzero} and \eqref{4.23}  and the fact that, by \eqref{case1},
\[\tau\lesssim\Lambda_1\lesssim\tau,\]
yields
\begin{equation}\label{4.24}
||E_{0,+}\omega_{0,+}(\xi,\eta,\cdot)||^2_{L^2(\mathbb{R}_{+})}
\geq C^{-1}\left(\Lambda_1^2\int_0^\infty|\omega_{0,+}(\xi,\eta,y)|^2dy+\Lambda_1|\omega_{0,+}(\xi,\eta,0)|^2\right),
\end{equation}
where $C$ depends only on $\lambda_0$ and $M_0$ and provided \eqref{szero} holds true.

Similarly, we have that
\begin{eqnarray}\label{4.25}
&&\lambda_0^{-2}\int_0^\infty\left|\omega_{0,+}\right|^2dy
\geq\int_0^\infty|\partial_y\hat{v}_++i\left(T_{+}+J_{0,+}\right)\hat{v}_+|^2dy\nonumber\\
&&+\int_0^\infty\left[\left(\tau\alpha_{+}+\tau\beta y-R_{0,+}\right)^2+\tau\beta-\partial_yR_{0,+}\right]|\hat{v}_+|^2dy
+\left((\tau\alpha_{+}-R_{0,+})|\hat{v}_+|^2\right)_{|_{y=0}}.
\end{eqnarray}
The assumption \eqref{case1} and \eqref{stimeRzero} imply
\begin{equation}\label{4.26}
\tau\alpha_{+}+\tau\beta y-R_{0,+}
\geq\tau\alpha_+ -C_0\Lambda_0\geq \tau(\alpha_+-C_0s_0)\geq \frac{\tau\alpha_+}{2}
\end{equation}
provided
\begin{equation}\label{szerobis}
s_0\leq \frac{C_0\alpha_+}{2}.\end{equation}
Thus, if \eqref{szero} and \eqref{szerobis} hold true, by \eqref{4.26} and \eqref{stimeRzero}
we have
\begin{equation}\label{boh}
	\left(\tau\alpha_{+}+\tau\beta y-R_{0,+}\right)^2+\tau\beta-\partial_yR_{0,+}\geq \frac{\tau^2\alpha_+^2}{4} +\frac{\tau\beta}{2}\geq C\Lambda_1^2.
\end{equation}
Also   by \eqref{1.2}, \eqref{2.12}, and the fact that
\[|J_{0,+}|\leq |\zeta_{0,+}|\leq C \Lambda_0\]
for some $C$ depending only on $\lambda_0$, we have that, for any $\eps\leq 1$
\begin{eqnarray}\label{4.28}
\int_0^\infty|\partial_y\hat{v}_++i\left(T_{+}+J_{0,+}\right)\hat{v}_+|^2dy&\geq&\eps\int_0^\infty|\partial_y\hat{v}_++i\left(T_{+}+J_{0,+}\right)\hat{v}_+|^2dy\nonumber\\
&\geq&\eps\left\{\frac{1}{2}\int_0^\infty\left|\partial_y\hat{v}_+\right|^2dy-2\int_0^\infty|T_{+}+J_{0,+}|^2|\hat{v}_+|^2dy\right\}\nonumber\\
&\geq&\frac{\eps}{2}\int_0^\infty|\partial_y\hat{v}_+|^2dy-C\eps\Lambda_0^2\int_0^\infty|\hat{v}_+|^2dy.
\end{eqnarray}
 Choosing $\eps$ sufficiently small, we obtain, from \eqref{4.25}, \eqref{boh} and \eqref{4.28},
\begin{equation}\label{4.29}
||\omega_{0,+}||^2_{L^2(\mathbb{R}_{+})}
\geq\frac{1}{C}\left( \int_0^\infty|\partial_y\hat{v}_+|^2dy+\Lambda_1^2\int_0^\infty|\hat{v}_+|^2dy+\Lambda_1|\hat{v}_+|^2_{|_{y=0}}\right),
\end{equation}
where $C$ depends only on $\lambda_0$ and $M_0$.

Recalling \eqref{omega0} and combining \eqref{4.24} and \eqref{4.29} yields
\begin{eqnarray}\label{4.30}
\left\|E_{0,+}a_{nn}^+F_{0,+}\hat{v}_+\right\|^2_{L^2(\mathbb{R}_{+})}&\geq&\frac{1}{C}\left\{\Lambda_1^2
\left\|\partial_y\hat{v}_+\right\|^2_{L^2(\mathbb{R}_{+})}
+\Lambda_1^4\left\|\hat{v}_+\right\|^2_{L^2(\mathbb{R}_{+})}\right.\nonumber\\&&+\left.\Lambda_1^3|\hat{v}_
+|^2_{|_{y=0}}+\Lambda_1|a_{nn}^+F_{0,+}\hat{v}_+|^2_{|_{y=0}}\right\},
\end{eqnarray}
where $C$ depends only on $\lambda_0$ and $M_0$.

By \eqref{4.30},  \eqref{1-8p}, and by \eqref{2-8p},
for $s_0$  small enough \eqref{4.19} follows.

The proof of\eqref{4.20} follows the same path, the only difference is that, in the proof of  \eqref{4.26} the assumption
that ${\rm supp}(\hat{v}_-(\xi,\cdot))\subset[-\frac{1}{C\beta},\frac{1}{C\beta}]$ comes into play.
\end{proof}

In order to conclude the proof of Proposition \ref{prop3.1} in this case, we need to connect the traces of the function for $y=0$ to the transmission conditions $\sigma_0$ and $\sigma_1$. This is done in next Lemma.

\begin{lemma}\label{lem4.3}
Let  $\tau\geq1$ and assume \eqref{case1}. There exists a positive constant $C$ depending only on $\lambda_0$ and $M_0$ such that if $s_0\leq 1/C$ then
\begin{eqnarray}\label{4.37}
\Lambda_1\sum_{\pm}|V_{\pm}(\xi,\eta)|^2\!\!\!\!&+\!\!\!\!&\Lambda_1^3\sum_{\pm}|\hat{v}_{\pm}(\xi,\eta,0)|^2+\Lambda_1^4\sum_{\pm}||\hat{v}_{\pm}(\xi,\eta,\cdot)||^2_{L^2(\mathbb{R}_{\pm})}
+\Lambda_1^2\sum_{\pm}||\partial_y\hat{v}_{\pm}(\xi,\eta,\cdot)||^2_{L^2(\mathbb{R}_{\pm})}\nonumber\\
&&\leq C\left\{\sum_{\pm}||P_{s,\pm}\hat{v}_{\pm}(\xi,\eta,\cdot)||^2_{L^2(\mathbb{R}_{\pm})}+\Lambda_1|\sigma_1(\xi,\eta)|^2+C\Lambda_1^3|\sigma_0(\xi,\eta)|^2\right\},
\end{eqnarray}
if ${\rm supp}(\hat{v}_\pm(\xi,\cdot))\subset[-\frac{1}{C\beta},\frac{1}{C\beta}]$.
\end{lemma}
\begin{proof} It follows from \eqref{4.19}  that, for some $C$ depending only on $\lambda_0$ and $M_0$,
\begin{equation}\label{4.38}
\Lambda_1|a_{nn}^+(0)F_{0,+}\hat{v}_+(\xi,\eta,0)|^2+\Lambda_1^3|\hat{v}_+(\xi,\eta,0)|^2\leq C \left\|P_{s,+}\hat{v}_{+}(\xi,\eta,\cdot)\right\|^2_{L^2(\mathbb{R}_{+})}.
\end{equation}
Since  by \eqref{Vpm} we have
\[V_+(\xi,\eta)=a_{nn}^+(0)F_{0,+}\hat{v}_+(\xi,\eta,0)-a_{nn}^+(0)\left(\tau s T_+(\gamma,0)+\sqrt{\zeta_{+,0}(\xi,\eta,0)}\right)\hat{v}_+(\xi,\eta,0),\]
and since by \eqref{1-15bis}
\[\left|\tau s T_+(\gamma,0)+\sqrt{\zeta_{+,0}(\xi,\eta,0)}\right|\leq C\Lambda_1,\]
hence, using \eqref{4.38},
\begin{equation}\label{4.39}
\Lambda_1|V_+(\xi,\eta)|^2\leq2\Lambda_1|a_{nn}^+(0)F_{0,+}\hat{v}_+(\xi,\eta,0)|^2+C\Lambda_1^3|\hat{v}_+(\xi,\eta,0)|^2\leq C||P_{s,+}\hat{v}_{+}(\xi,\eta,\cdot)||^2_{L^2(\mathbb{R}_{+})},
\end{equation}
where $C$ depends only on $\lambda_0$ and $M_0$.

By \eqref{sigma1} and \eqref{4.39},
\begin{equation}\label{4.42}
\Lambda_1|V_-(\xi,\eta)|^2\leq 2\Lambda_1|V_+(\xi,\eta)|^2+2\Lambda_1|\sigma_1(\xi,\eta)|^2\leq C ||P_{s,+}\hat{v}_{+}(\xi,\eta,\cdot)||^2_{L^2(\mathbb{R}_{+})}+2\Lambda_1|\sigma_1(\xi,\eta)|^2.
\end{equation}

In a similar way, By \eqref{sigma0} and \eqref{4.39}, we have that
\begin{equation}\label{4.40}
\Lambda_1^3|\hat{v}_-(\xi,\eta,0)|^2\leq 2\Lambda_1^3|\hat{v}_+(\xi,\eta,0)|^2+2\Lambda_1^3|\sigma_0(\xi,\eta)|^2
\leq C ||P_{s,+}\hat{v}_{+}(\xi,\eta,\cdot)||^2_{L^2(\mathbb{R}_{+})}+2\Lambda_1^3|\sigma_0(\xi,\eta)|^2.
\end{equation}

By putting together  \eqref{4.39}, \eqref{4.42} and \eqref{4.40} , we obtain
\begin{equation*}
\Lambda_1^3\sum_{\pm}|\hat{v}_{\pm}(\xi,\eta,0)|^2+\Lambda_1\sum_{\pm}|V_{\pm}(\xi,\eta)|^2\leq C||P_{s,+}\hat{v}_{+}(\xi,\eta,0)||^2_{L^2(\mathbb{R}_{+})}+2\Lambda_1^3|\sigma_0(\xi,\eta)|^2+2\Lambda_1|\sigma_1(\xi,\eta)|^2,
\end{equation*}
that, together with \eqref{4.19} and \eqref{4.20} of Lemma \ref{lem4.2}, gives \eqref{4.37}.\end{proof}
\begin{rem}
Since $\tau\geq 1$, we can write \eqref{4.37} in the following weaker form
\begin{eqnarray*}
&&\Lambda_1\sum_{\pm}|V_{\pm}(\xi,\eta)|^2+\Lambda_1^3\sum_{\pm}|\hat{v}_{\pm}(\xi,\eta,0)|^2+\frac{\Lambda_1^4}{\tau}\sum_{\pm}||\hat{v}_{\pm}||^2_{L^2(\mathbb{R}_{\pm})}
+\frac{\Lambda_1^2}{\tau}\sum_{\pm}||\partial_y\hat{v}_{\pm}||^2_{L^2(\mathbb{R}_{\pm})}\nonumber\\
&&\quad\leq C\left(\sum_{\pm}||P_{s,\pm}\hat{v}_{\pm}||^2_{L^2(\mathbb{R}_{\pm})}+\Lambda_1|\sigma_1(\xi,\eta)|^2+\Lambda_1^3|\sigma_0(\xi,\eta)|^2\right),
\end{eqnarray*}
where $C$ depends on $\lambda_0$ and $M_0$ only.
\end{rem}
\subsubsection{Some useful estimates}
In this section we write down some estimates that will be useful in second and third cases of the main proof. In both these cases we have
\begin{equation}\label{4-16p}
\tau\leq C_1\Lambda_0	
\end{equation}
for $C_1$ independent on $s$. Notice that
\[\Lambda_0\lesssim\Lambda_1\lesssim\Lambda_0,\]

\begin{lemma}\label{zetabelow}
If \eqref{4-16p} holds, there are two constants $C$ and $s_1$ depending on $\lambda_0$ and $C_1$ only, such that, if $s\leq s_1$, then
  \begin{equation}\label{3-15bis}
	 |\zeta_{s,\pm}|\geq\frac{1}{C}\Lambda_1^2.
 \end{equation}
\end{lemma}
\begin{proof}
For a fixed $y$ two cases occur

\begin{subequations}
\begin{equation}
\label{B-case-a}
\B_{\pm}(\xi,\xi,y)\geq 2s^2\tau^2|\gamma|^{-2}\B_{\pm}(\gamma,\gamma,y),
\end{equation}
\begin{equation}
\label{B-case-b}
\B_{\pm}(\xi,\xi,y)\leq 2s^2\tau^2|\gamma|^{-2}\B_{\pm}(\gamma,\gamma,y).
\end{equation}
\end{subequations}
In case \eqref{B-case-a}, by \eqref{1.15} and \eqref{2.14},  we have

\begin{equation}\label{14-4-17}
\begin{aligned}
&|\zeta_{s,\pm}(\xi,\eta,y)|^2\geq \left(\frac{1}{2}\B_{\pm}(\xi,\xi,y)\right)^2+\left(2s\tau \B_{\pm}(\xi,\gamma,y)+\frac{\eta}{a_{nn}^{\pm}(y)}\right)^2\\
&\geq \frac{\lambda^2_1}{4}|\xi|^2+\left(2s\tau \B_{\pm}(\xi,\gamma,y)+\frac{\eta}{a_{nn}^{\pm}(y)}\right)^2.
\end{aligned}
\end{equation}
Moreover by \eqref{1.15} we have

\begin{equation*}
|\B_{\pm}(\xi,\gamma,y)|\leq \lambda_1^{-1}|\xi||\gamma|.
\end{equation*}
Now either
\begin{equation*}
(i) \quad \quad \quad \quad \frac{|\eta|}{a_{nn}^{\pm}(y)}\geq 4s\tau \lambda_1^{-1}|\xi||\gamma|
\end{equation*}
or
\begin{equation*}
(ii) \quad \quad \quad \quad \frac{|\eta|}{a_{nn}^{\pm}(y)}\leq 4s\tau \lambda_1^{-1}|\xi||\gamma|.
\end{equation*}

If case (i) occurs then we have

\begin{equation*}
\left|2s\tau \B_{\pm}(\xi,\gamma,y)+\frac{\eta}{a_{nn}^{\pm}(y)}\right|\geq \frac{|\eta|}{a_{nn}^{\pm}(y)}-2s\tau \lambda_1^{-1}|\xi||\gamma|\geq \frac{\lambda_0}{2}|\eta|,
\end{equation*}
hence by \eqref{4-16p} and \eqref{14-4-17} we get \eqref{3-15bis}
where $C$ depends on $\lambda_0$ only.

If case (ii) occurs then, by \eqref{1.15} and \eqref{B-case-a} we have
\[2\lambda_1s^2\tau^2\leq \B_{\pm}(\xi,\xi,y)\leq \lambda_1^{-1}|\xi|^2\]
hence
\[s\tau\leq\frac{\lambda_1^{-1}}{\sqrt{2}}|\xi|\]
so that
\[|\eta|\leq 4s\tau \lambda_1^{-1}|\xi||\gamma|\leq 4(\lambda_0\lambda_1)^{-1}|\xi|^2,\]
this inequality combined with \eqref{4-16p} yields again \eqref{3-15bis}.

Now, let us consider case \eqref{B-case-b}. By this condition and by \eqref{1.15} we have

\begin{equation*}
\lambda_1|\xi|^2\leq\B_{\pm}(\xi,\xi,y)\leq 2s^2\tau^2|\gamma|^{-2}\B_{\pm}(\gamma,\gamma,y)\leq 2\lambda^{-1}_1s^2\tau^2,	
\end{equation*}
hence

\begin{equation}\label{14-4-172}
\lambda_1|\xi|^2\leq 2\lambda^{-1}_1s^2\tau^2.	
\end{equation}

This inequality and \eqref{4-16p} give

\begin{equation*}
\tau^2\leq C^2_1\Lambda^2_0=C^2_1(\xi|^2+|\eta|)\leq 2C^2_1 \lambda_1^{-2}s^2\tau^2+C^2_1|\eta|	
\end{equation*}
by this inequality we have, for $s\leq\frac{\lambda_1}{2C_1}$,
\begin{equation}\label{14-4-173}
\tau^2\leq 2C^2_1|\eta|,
\end{equation}
that combined with \eqref{14-4-172} gives, for $s\leq\frac{\sqrt{\lambda_0}\lambda_1}{4C_1}$

\begin{equation}\label{14-4-174}
\begin{aligned}
&|\zeta_{s,\pm}(\xi,\eta,y)|\geq\left|2s\tau \B_{\pm}(\xi,\gamma,y)+\frac{\eta}{a_{nn}^{\pm}(y)}\right|\geq \lambda_0|\eta|
-2\lambda_1^{-1}|\xi|s\tau\\
&\geq (\lambda_0-8\lambda_1^{-2}C_1^2s^2)|\eta|\geq \frac{\lambda_0}{2}|\eta|.
\end{aligned}
\end{equation}
By \eqref{14-4-172}, \eqref{14-4-173}, \eqref{14-4-174} we get again \eqref{3-15bis}.
\end{proof}

\begin{lemma}\label{above}
If \eqref{4-16p} holds, there are two constants $C$ and $s_1$ depending on $\lambda_0$, $M_0$ and $C_1$ only, such that, if $s\leq s_1$, then
  \begin{equation}\label{4-15bis}
	 |\partial_y \sqrt{\zeta_{s,\pm}}|\leq C\Lambda_0,
 \end{equation}
 \begin{equation}\label{5-15bis}
	 |\partial_y R_{s,\pm}|\leq C\Lambda_0,\quad|\partial_y J_{s,\pm}|\leq C\Lambda_0,
 \end{equation}
 \begin{equation}\label{1-15}
	 R_{s,\pm}\geq\frac{1}{C}\Lambda_1.
 \end{equation}
\end{lemma}
\begin{proof}
By \eqref{1.15} and \eqref{1.15bis} we have
\begin{equation*}
	 |\partial_y \zeta_{s,\pm}|\leq C\Lambda_1^2.
 \end{equation*}
by this inequality, \eqref{4-16p} and \eqref{3-15bis} we obtain \eqref{4-15bis}.

Inequalities \eqref{5-15bis} follow immediately by \eqref{4-15bis}.

In order to prove \eqref{1-15} we denote

\begin{equation*}
\begin{aligned}
&a=\Re\zeta_{s,\pm}(\xi,\eta,y)=B_{\pm}(\xi,\xi,y)-s^2\tau^2\B_{\pm}(\gamma,\gamma,y),\\
&b=\Im\zeta_{s,\pm}(\xi,\eta,y)=2s\tau \B_{\pm}(\xi,\gamma,y)+\frac{\eta}{a_{nn}^\pm(y)},
\end{aligned}
\end{equation*}
so that we have
\begin{equation*}
2R^2_{s,\pm}=a+\left(a^2+b^2\right)^{1/2}=a+|\zeta_{s,\pm}(\xi,\eta,y)|.
 \end{equation*}
 For a fixed $y$ we distinguish two cases
\begin{subequations}
\begin{equation}
\label{B-case-a-1}
\B_{\pm}(\xi,\xi,y)\geq 2s^2\tau^2|\gamma|^{-2}\B_{\pm}(\gamma,\gamma,y),
\end{equation}
\begin{equation}
\label{B-case-b-1}
\B_{\pm}(\xi,\xi,y)\leq 2s^2\tau^2|\gamma|^{-2}\B_{\pm}(\gamma,\gamma,y).
\end{equation}
\end{subequations}
If case \eqref{B-case-a-1} occurs then by \eqref{3-15bis} we have

\begin{equation*}
2R^2_{s,\pm}\geq|\zeta_{s,\pm}(\xi,\eta,y)|\geq\frac{\Lambda^2_1}{C}.
 \end{equation*}
Hence, if \eqref{B-case-a-1} is satisfied then \eqref{1-15} is true.

Now, let us consider case \eqref{B-case-b-1}. First, let us notice that in such a case \eqref{B-case-b-1}, by \eqref{1.15} we have

\begin{equation}\label{14-4-175}
|\xi|\leq\sqrt{2}\lambda^{-1}s\tau.
 \end{equation}
By \eqref{4-16p} and \eqref{14-4-175}, we get

\[\tau^2\leq C^2_1\left(|\xi|^2+|\eta|\right)\leq 2\lambda^{-2}_1C^2_1s^2\tau^2+C^2_1|\eta|,\]

so that, for $s\leq\frac{\lambda_1}{2C_1}$, we have

\begin{equation}\label{14-4-176}
\tau^2\leq 2C^2_1|\eta|.
\end{equation}

Now, by \eqref{14-4-175} and \eqref{14-4-176} we have, for $s\leq\frac{\sqrt{\lambda_0}\lambda_1}{4C_1}$,

\begin{equation*}
2R^2_{s,\pm}\geq-s^2\tau^2\B_{\pm}(\gamma,\gamma,y)+|b|\geq \left(\lambda_0-8\lambda_1^{-2}C^2_1s^2\right)|\eta|\geq \frac{\lambda_0}{2}|\eta|
\end{equation*}
this inequality combined with \eqref{14-4-175} and \eqref{14-4-176} gives \eqref{1-15} whenever \eqref{B-case-b-1} is satisfied. The proof is completed.
\end{proof}

\begin{lemma}\label{s}
If \eqref{4-16p} holds, there is a constant $C$ depending on $\lambda_0$ and $C_1$ only, such that
  \begin{equation}\label{6-15bis}
	 |\zeta_{s,\pm}(\xi,\eta,y)-\zeta_{0,\pm}(\xi,\eta,y)|\leq Cs\Lambda_0^2 ,
 \end{equation}
 \begin{equation}\label{8-15bis}
	\left|\sqrt{\zeta_{s,\pm}(\xi,\eta,y)}-\sqrt{\zeta_{0,\pm}(\xi,\eta,y)}\right|\leq Cs\Lambda_0,
 \end{equation}
\begin{equation}\label{11-15bis}
	|R_{s,\pm}(\xi,\eta,y)-R_{0,\pm}(\xi,\eta,y)|\leq Cs\Lambda_0,\quad |J_{s,\pm}(\xi,\eta,y)-J_{0,\pm}(\xi,\eta,y)|\leq Cs\Lambda_0.
 \end{equation}
\end{lemma}
\begin{proof}
By \eqref{2.14} and \eqref{4-16p} we have
\begin{equation*}
|\zeta_{s,\pm}(\xi,\eta,y)-\zeta_{0,\pm}(\xi,\eta,y)|=|-s^2\tau^2\B_{\pm}(\gamma,\gamma,y)+2is\tau \B_{\pm}(\xi,\gamma,y)|\leq Cs(\tau^2+|\xi|^2)\leq Cs\Lambda^2_0
\end{equation*}
and \eqref{6-15bis} follows.
Inequalities \eqref{8-15bis} \eqref{11-15bis} are immediate consequences of \eqref{1-15} and \eqref{6-15bis}.
\end{proof}

\begin{lemma}\label{y}
If \eqref{4-16p} holds, there are two constants $C$ and $s_1$ depending on $\lambda_0$, $M_0$ and $C_1$ only, such that, if $s\leq s_1$, then
  \begin{equation}\label{9-15bis}
	 |\zeta_{s,\pm}(\xi,\eta,y)-\zeta_{s,\pm}(\xi,\eta,0)|\leq C|y|\Lambda^2_0,
 \end{equation}
 \begin{equation}\label{7-15bis}
|\zeta_{s,\pm}(\xi,\eta,y)-\zeta_{0,\pm}(\xi,\eta,0)|\leq C(|y|+s)\Lambda^2_0,
\end{equation}
\begin{equation}\label{10-15bis}
\left|\sqrt{\zeta_{s,\pm}(\xi,\eta,y)}-\sqrt{\zeta_{0,\pm}(\xi,\eta,0)}\right|\leq C(|y|+s)\Lambda_0,\end{equation}
	\begin{equation}\label{12-15bis}
	|R_{s,\pm}(\xi,\eta,y)-R_{0,\pm}(\xi,\eta,0)|\leq  C(|y|+s)\Lambda_0.
  \end{equation}
\end{lemma}
\begin{proof}
Inequality \eqref{9-15bis} is an immediate consequence of \eqref{1.15}, \eqref{1.15bis} and \eqref{4-16p}. Inequality \eqref{7-15bis} follows by  \eqref{8-15bis} and \eqref{9-15bis}. Inequalities \eqref{10-15bis} and \eqref{12-15bis} follow by \eqref{4-15bis} and \eqref{6-15bis}.
\end{proof}

\subsubsection{Second case}\label{secondcase}
In this case we assume
\begin{equation}\label{case2}
	\frac{R_{0,+}(\xi,\eta,0)}{(1-\kappa)\alpha_+}\leq\tau\leq \frac{\Lambda_0}{s_0}.
\end{equation}
Let us point out that in this case by \eqref{stimeRzer}
\begin{equation}\label{27p}
	\Lambda_0\lesssim \tau \mbox{ and }\Lambda_1\lesssim \Lambda_0.
\end{equation}
By \eqref{3.12}, \eqref{3.13}, \eqref{1-15bis}, \eqref{4-15bis} and \eqref{27p}, we have
\begin{equation}\label{6-27p}
	|P_{s,+}\hat{v}_+-E_{s,+} a^{+}_{nn}(y) F_{s,+}\hat{v}_+|\leq C\Lambda_0|\hat{v}_+|,
\end{equation}
\begin{equation}\label{7-27p}
		|P_{s,-}\hat{v}_--F_{s,-} a^{-}_{nn}(y) E_{s,-}\hat{v}_-|\leq C\Lambda_0|\hat{v}_-|,
\end{equation}
where $C$ depends on $\lambda_0$ and $M_0$ only.

\begin{lemma}\label{lem3.5}
Assume \eqref{case2}. There exist $C$, $s_2\leq s_1$ depending only on $\lambda_0$ and $M_0$ such that,
if $0\leq s\leq s_2$, $\beta\geq C$ and $\tau\geq C$, then we have
\begin{equation}\label{4.53}
\begin{aligned}
\Lambda_1&|V_+(\xi,\eta)+a_{nn}^+(0)\sqrt{\zeta_{s,+}(\xi,\eta,0)}\hat{v}_+(\xi,\eta,0)|^2+\Lambda_1^2||a_{nn}^+(y)F_{s,+}\hat{v}_+(\xi,\eta,\cdot)||^2_{L^2(\mathbb{R}_{+})}\\
&\leq C||E_{s,+}a_{nn}^+(y)F_{s,+}\hat{v}_+(\xi,\eta,\cdot)||^2_{L^2(\mathbb{R}_{+})}
\end{aligned}
\end{equation}
and
\begin{equation}\label{4.54}
\begin{aligned}
\Lambda_1&|V_+(\xi,\eta)+a_{nn}^+(0)\sqrt{\zeta_{s,+}(\xi,\eta,0)}\hat{v}_+(\xi,\eta,0)|^2+\Lambda_1^3|\hat{v}_+(\xi,\eta,0)|^2\\
&+\Lambda_1^4||\hat{v}_+(\xi,\eta,\cdot)||^2_{L^2(\mathbb{R}_{+})}+\Lambda_1^2||\partial_y\hat{v}_+(\xi,\eta,\cdot)||^2_{L^2(\mathbb{R}_{+})}\leq C||P_{s,+}\hat{v}_{+}(\xi,\eta,\cdot)||^2_{L^2(\mathbb{R}_{+})}
\end{aligned}
\end{equation}
provided ${\rm supp}(\hat{v}_+(\xi,\eta,\cdot))\subset[0,\frac{1}{C}]$.
\end{lemma}
\begin{proof}
Define
\begin{equation*}
\omega_{s,+}(\xi,\eta,y)=a_{nn}^+(y)F_{s,+}\hat{v}_+(\xi,\eta,y).
\end{equation*}
For sake of shortness we omit arguments unless they are necessary.

We have, by integration by parts,
\begin{equation*}
	\begin{aligned}
\left\|E_{s,+}\omega_{s,+}\right\|^2_{L^2(\mathbb{R}_{+})}=&\int_0^{+\infty}\!\!\!\!\left|\left(\partial_y+i(T_{+}(\xi,y)-J_ {s,+})\right)\omega_{s,+}-\left(\tau\alpha_++\tau\beta y+R_{s,+}+\tau s T_+(\gamma,y)\right)\omega_{s,+}\right|^2dy\\
=& \int_0^{+\infty}\!\!\!\!\left|\partial_y\omega_{s,+}+i(T_{+}(\xi,y)-J_ {s,+})\omega_{s,+}\right|^2\\
&+\int_0^{+\infty}\!\!\!\!\left[\tau\alpha_++\tau\beta y+R_{s,+}+\tau s T_+(\gamma,y)\right]^2|\omega_{s,+}|^2dy\\
&-2\Re\int_0^{+\infty}\!\!\!\!\left(\partial_y\omega_{s,+}+i(T_{+}(\xi,y)-J_ {s,+})\omega_{s,+}\right)\left(
\tau\alpha_++\tau\beta y+R_{s,+}+\tau s T_+(\gamma,y)\right)\overline{\omega}_{s,+}\,dy\\
\geq& \int_0^{+\infty}\!\!\!\! \left\{\left[\tau\alpha_++\tau\beta y+R_{s,+}+\tau s T_+(\gamma,y)\right]^2+
\tau\beta +\partial_yR_{s,+}+\tau s \partial_yT_+(\gamma,y)\right\}|\omega_{s,+}|^2dy\\
&+(\tau\alpha_++R_{s,+}(\xi,\eta,0)+\tau s T_+(\gamma,0))|\omega_{s,+}(\xi,\eta,0)|^2.
		\end{aligned}
\end{equation*}
Since $R_{s,+}\geq 0$, for $s$ small enough
\[\tau\alpha_++\tau\beta y+R_{s,+}+\tau s T_+(\gamma,y)\geq \tau\alpha_+-Cs\tau\geq \frac{\tau\alpha_+}{2}\]
and, hence, by \eqref{5-15bis} and \eqref{case2},
\begin{equation*}
\begin{aligned}
&\left[\tau\alpha_++\tau\beta y+R_{s,+}+\tau s T_+(\gamma,y)\right]^2+
\tau\beta +\partial_yR_{s,+}+\tau s \partial_yT_+(\gamma,y)\\
& \geq \left(\frac{\tau\alpha_+}{2}\right)^2-C\Lambda_0-Cs\tau\geq \left(\frac{\tau\alpha_+}{2}\right)^2-C\tau\gtrsim
\frac{\tau^2\alpha_+^2}{8}\gtrsim \Lambda_1^2.
\end{aligned}
\end{equation*}
Moreover, for $s$ small enough
\begin{equation*}
	\tau\alpha_++R_{s,+}(\xi,\eta,0)+\tau s T_+(\gamma,0)\geq \tau\alpha_++R_{s,+}(\xi,\eta,0)-C s \tau\geq \frac{\tau\alpha_+}{2}\gtrsim \Lambda_1,
\end{equation*}
hence
\begin{equation*}
\left\|E_{s,+}\omega_{s,+}\right|^2_{L^2(\mathbb{R}_{+})}\geq\frac{\Lambda_1^2}{C}
\left\|\omega_{s,+}\right|^2_{L^2(\mathbb{R}_{+})}+\frac{\Lambda_1}{C}|\omega_{s,+}(\xi,\eta,0)|^2,
\end{equation*}
that is \eqref{4.53} because
\[\omega_{s,+}(\xi,\eta,0)=V_+(\xi,\eta)+a_{nn}^+(0)\sqrt{\zeta_{+}(\xi,\eta,0)}\hat{v}_+(\xi,\eta,0).\]
Let us now consider \eqref{4.54}. Let us write
\begin{equation}\label{1-18}
\begin{aligned}
\left\|F_{s,+}\hat{v}_{+}(\xi,\cdot)\right\|^2_{L^2(\mathbb{R}_{+})}=&
\int_0^{+\infty}\!\!\!\!\left|\left(\partial_y+i(T_{+}(\xi,y)+J_ {s,+})\right)\hat{v}_{+}-\left(\tau\alpha_++\tau\beta y-R_{s,+}+\tau s T_{+}(y,\gamma)\right)\hat{v}_{+}\right|^2dy\\
=&\int_0^{+\infty}\!\!\!\!\left|\partial_y\hat{v}_{+}+i(T_{+}(\xi,y)+J_ {s,+})\hat{v}_{+}\right|^2\\
&+ \int_0^{+\infty}\!\!\!\! \left\{\left[\tau\alpha_++\tau\beta y-R_{s,+}+\tau s T_+(\gamma,y)\right]^2+
\tau\beta -\partial_yR_{s,+}+\tau s \partial_yT_+(\gamma,y)\right\}|\hat{v}_{+}|^2dy\\
&+(\tau\alpha_+-R_{s,+}(\xi,\eta,0)+\tau s T_+(\gamma,0))|\hat{v}_{+}(\xi,\eta,0)|^2.
\end{aligned}
\end{equation}
By \eqref{case2} and \eqref{12-15bis} we have
\begin{equation}\label{2-18}
\begin{aligned}
	\tau\alpha_+-R_{s,+}(\xi,\eta,0)+\tau s T_+(\gamma,0)&=\tau\alpha_+-R_{0,+}(\xi,\eta,0)+R_{0,+}(\xi,\eta,0)-R_{s,+}(\xi,\eta,0)+\tau s T_+(\gamma,0)\\
	&\geq \tau\alpha_+-(1-\kappa)\alpha_+\tau-C s \Lambda_0-C\tau s\\
	&\geq \tau(\kappa \alpha_+-Cs)\geq \frac{\Lambda_1}{C},
	\end{aligned}
\end{equation}
for $s$ small enough and $C$ depending on $\lambda_0$ and $\alpha_\pm$ only.

Moreover, by \eqref{case2} and \eqref{12-15bis} ,
\begin{equation}\label{2-30p}
	\begin{aligned}
\tau\alpha_++\tau\beta y-R_{s,+}(\xi,\eta,y)+\tau s T_+(\gamma,y)\geq&
	\tau\alpha_++\tau\beta y-R_{0,+}(\xi,\eta,0)\\&-\left|R_{0,+}(\xi,\eta,0)-R_{s,+}(\xi,\eta,y)\right|+\tau s T_+(\gamma,y)\\
	\geq& \kappa \alpha_+\tau -C(s+y)\Lambda_0\geq \frac{\kappa\alpha_+}{2}\tau,
	\end{aligned}
\end{equation}
for positive and small enough $y$ and $s$.

Hence, by \eqref{5-15bis}, \eqref{27p} and \eqref{2-30p} we have
\begin{equation}\label{1-31p}
	\left[\tau\alpha_++\tau\beta y-R_{s,+}+\tau s T_+(\gamma,y)\right]^2+
\tau\beta -\partial_yR_{s,+}+\tau s \partial_yT_+(\gamma,y) \geq \left(\frac{\kappa\alpha_+}{2}\right)^2\tau^2+\beta\tau-C\tau\geq
\left(\frac{\kappa\alpha_+}{2}\right)^2\tau^2
\end{equation}
for $\beta\geq C$.

By putting together \eqref{1-18}, \eqref{2-18}, \eqref{1-31p} and \eqref{2-15bis}, we finally get
\begin{equation*}
	\|F_{s,+}\hat{v}_+\|^2_{L^2(\mathbb{R}_{+})}\geq\frac{1}{C}\left\{\|\partial_y \hat{v}_+\|^2_{L^2(\mathbb{R}_{+})}+\tau^2
	\|\hat{v}_+\|^2_{L^2(\mathbb{R}_{+})}+\tau|\hat{v}_+|^2_{|_{y=0}}\right\}
\end{equation*}
 that, combined with \eqref{6-27p} and \eqref{4.53} gives \eqref{4.54}.
\end{proof}

\begin{lemma}\label{lem3.6}
Assume \eqref{case2}. There exist $C$ depending only on $\lambda_0$ and $M_0$ such that,
if $0\leq s\leq s_1$, $\beta\geq C$ and $\tau\geq C$, then we have
\begin{equation}\label{4.67}
\begin{aligned}
-\Lambda_1&|V_-(\xi,\eta)-a_{nn}^-(0)\sqrt{\zeta_{s,-}(\xi,\eta,0)}\hat{v}_-(\xi,\eta,0)|^2+\Lambda_1||a_{nn}^-(y)E_{s,-}\hat{v}_-(\xi,\eta,\cdot)||^2_{L^2(\mathbb{R}_{+})}\\
&\leq C||F_{s,-}a_{nn}^+(y)F_{s,-}\hat{v}_+(\xi,\eta,\cdot)||^2_{L^2(\mathbb{R}_{+})}
\end{aligned}
\end{equation}
and
\begin{equation}\label{4.68}
\begin{aligned}
-\Lambda_1&|V_-(\xi,\eta)+a_{nn}^-(0)\sqrt{\zeta_{s,-}(\xi,\eta,0)}\hat{v}_-(\xi,\eta,0)|^2-\Lambda_1^3|\hat{v}_-(\xi,\eta,0)|^2\\
&+\Lambda_1^3||\hat{v}_-(\xi,\eta,\cdot)||^2_{L^2(\mathbb{R}_{-})}+\Lambda_1||\partial_y\hat{v}_-(\xi,\eta,\cdot)||^2_{L^2(\mathbb{R}_{-})}\leq C||P_{s,-}\hat{v}_{-}(\xi,\eta,\cdot)||^2_{L^2(\mathbb{R}_{-})}
\end{aligned}
\end{equation}
provided ${\rm supp}(\hat{v}_-(\xi,\eta,\cdot))\subset[-\frac{\alpha_-}{2\beta},0]$.
\end{lemma}
\begin{proof}
Define
\begin{equation*}
	\omega_{s,-}(\xi,\eta,y)=a_{nn}^-(y)E_{s,-}\hat{v}_-(\xi,\eta,y).
\end{equation*}
We have
\begin{equation*}
	\begin{aligned}
	 \int_{-\infty}^0\!\!\!|F_{s,-}\omega_{s,-}|^2=&\int_{-\infty}^0\!\!\!\left|\partial_y\omega_{s,-}+i(T_-(\xi,y)+J_{s,-})\omega_{s,-}
-(\tau\alpha_-+\tau\beta y-R_{s,-}+\tau s T_-(\gamma,y))\omega_{s,-}\right|^2dy\\
	\geq& -2\Re \int_{-\infty}^0\!\!\!(\tau\alpha_-+\tau\beta y-R_{s,-}+\tau s T_-(\gamma,y))\overline{\omega}_{s,-}(\partial_y\omega_{s,-}+i(T_-(\xi,y)+J_{s,-})\omega_{s,-})dy\\
	=& -2\Re \int_{-\infty}^0\!\!\!(\tau\alpha_-+\tau\beta y-R_{s,-}+\tau s T_-(\gamma,y))\overline{\omega}_{s,-}\partial_y\omega_{s,-}dy\\
	=& -\Re \int_{-\infty}^0\!\!\!(\tau\alpha_-+\tau\beta y-R_{s,-}+\tau s T_-(\gamma,y))\partial_y\left|\omega_{s,-}\right|^2dy\\
	=&  \int_{-\infty}^0\!\!\!(\tau\beta -\partial_yR_{s,-}+\tau s \partial_yT_-(\gamma,y))\left|\omega_{s,-}\right|^2dy\\
	&-(\tau\alpha_--R_{s,-}(\xi,\eta,0)+\tau s T_-(\gamma,0))\left|\omega_{s,-}(\xi,\eta,0)\right|^2.
	\end{aligned}
\end{equation*}

By \eqref{5-15bis}, \eqref{derivt} and \eqref{27p} we have
\begin{equation*}
	\tau\beta -\partial_yR_{s,-}+\tau s \partial_yT_-(\gamma,y)\geq \tau\beta -C\Lambda_0-C\tau s\geq (\beta -C^\prime)\tau\geq \frac{\beta}{2}\tau
\end{equation*}
for $\beta$ sufficiently large,
and
\begin{equation*}
\tau\alpha_--R_{s,-}(\xi,\eta,0)+\tau s T_-(\gamma,0)\leq C\Lambda_1,
\end{equation*}
and, since
\begin{equation*}
	\omega_{s,-}(\xi,\eta,0)=V_-(\xi,\eta)-a_{nn}^-(0)\sqrt{\zeta_{s,-}(\xi,\eta,0)}\hat{v}_-(\xi,\eta,0),
\end{equation*}
we finally get \eqref{4.67}.

Let us now estimate
\begin{equation*}
	\begin{aligned} \int_{-\infty}^0\!\!\!\!\!|E_{s,-}\hat{v}_{-}|^2=&\int_{-\infty}^0\!\!\!\left|\partial_y\hat{v}_{-}
+i(T_-(\xi,y)-J_{s,-})\hat{v}_{-}-(\tau\alpha_-+\tau\beta y+R_{s,-}+\tau s T_-(\gamma,y))\hat{v}_{-}\right|^2dy\\
	=& \!\!\int_{-\infty}^0\!\!\!\!\!\left|\partial_y\hat{v}_{-}+i(T_-(\xi,y)-J_{s,-})\hat{v}_{-}\right|^2 \!\!dy
	-(\tau\alpha_-+R_{s,-}(\xi,\eta,0)+\tau s T_-(\gamma,0))\left|\hat{v}_{-}\right(\xi,\eta,0)|^2\\
	&+\int_{-\infty}^0\!\!\!\left[(\tau\alpha_-+\tau\beta y+R_{s,-}+\tau s T_-(\gamma,y))^2+\tau\beta +\partial_yR_{s,-}+\tau s \partial_yT_-(\gamma,y)\right]\left|\hat{v}_{-}\right|^2dy\\
	\end{aligned}
\end{equation*}
Then, since
if ${\rm supp}(\hat{v}_-(\xi,\eta,\cdot))\subset[-\frac{\alpha_-}{2\beta},0]$, provided $s$ small enough, we have
\begin{equation*}
(\tau\alpha_-+\tau\beta y+R_{s,-}+\tau s T_-(\gamma,y))^2+\tau\beta +\partial_yR_{s,-}+\tau s \partial_yT_-(\gamma,y)\geq C\Lambda_1^2,
\end{equation*}
and estimating the remaing terms as we did before, by \eqref{4.67} and \eqref{7-27p} we finally get
\eqref{4.68}.
\end{proof}
\begin{lemma}\label{lem3.7}
Assume \eqref{case2}. there exist $C$, $s_2\leq s_1$ depending on $\lambda_0$ and $M_0$ such that if $0\leq s\leq s_0$, $\beta\geq C$, $\tau\geq C$ and ${\rm supp}(\hat{v}(\xi,\eta,\cdot))\subset[-\frac{1}{C},\frac{1}{C}]$ then we have
\begin{equation}\label{1-34p}
\begin{aligned}
\Lambda_1\sum_{\pm}|V_{\pm}(\xi,\eta)|^2&+\Lambda_1^3\sum_{\pm}|\hat{v}_{\pm}(\xi,\eta,0)|^2+\Lambda_1^3\sum_{\pm}||\hat{v}_{\pm}(\xi,\eta,\cdot)||^2_{L^2(\mathbb{R}_{\pm})}
+\Lambda_1\sum_{\pm}||\partial_y\hat{v}_{\pm}(\xi,\eta,\cdot)||^2_{L^2(\mathbb{R}_{\pm})}\\
&\leq C\left\{\sum_{\pm}||P_{s,\pm}\hat{v}_{\pm}(\xi,\eta,\cdot)||^2_{L^2(\mathbb{R}_{\pm})}+\Lambda_1|\sigma_1(\xi,\eta)|^2+C\Lambda_1^3|\sigma_0(\xi,\eta)|^2\right\},
\end{aligned}
\end{equation}
\end{lemma}
\begin{proof}
The proof is the same of Lemma \ref{lem4.3}.
\end{proof}

\begin{rem}
Since, by \eqref{case2}, $\frac{\Lambda_1}{\tau}\leq C$, we can write \eqref{1-34p} in the following more convenient form
\begin{eqnarray*}
&&\Lambda_1\sum_{\pm}|V_{\pm}(\xi,\eta)|^2+\Lambda_1^3\sum_{\pm}|\hat{v}_{\pm}(\xi,\eta,0)|^2+\frac{\Lambda_1^4}{\tau}\sum_{\pm}||\hat{v}_{\pm}||^2_{L^2(\mathbb{R}_{\pm})}
+\frac{\Lambda_1^2}{\tau}\sum_{\pm}||\partial_y\hat{v}_{\pm}||^2_{L^2(\mathbb{R}_{\pm})}\nonumber\\
&&\quad\leq C\left(\sum_{\pm}||P_{s,\pm}\hat{v}_{\pm}||^2_{L^2(\mathbb{R}_{\pm})}+\Lambda_1|\sigma_1(\xi,\eta)|^2+\Lambda_1^3|\sigma_0(\xi,\eta)|^2\right),
\end{eqnarray*}
where $C$ depends on $\lambda_0$ and $M_0$ only.
\end{rem}

\subsubsection{Third case}\label{thirdcase}
In this case we assume
\begin{equation}
	\label{case3}
	\tau\leq\frac{R_{0,+}(\xi,\eta,0)}{(1-\kappa)\alpha_+}.
\end{equation}
Notice that, by \eqref{stimeRzero}, in this case condition \eqref{4-16p} is verified. Moreover
\begin{equation}
	\label{equivlambda}
	\Lambda_0\leq\Lambda_1\leq \sqrt{1+C_1^2}\Lambda_0.
\end{equation}

\begin{lemma}\label{lem3.8}
Assume \eqref{case3}. There exists a positive constant $C$ depending on $\lambda_0$, $M_0$ such that, if $0\leq s\leq s_1$ and $\tau\geq C$, then
\begin{equation}\label{3-35p}
	\Lambda_1\left|\omega_+(\xi,\eta,0)\right|^2+\Lambda_1^2\int_0^{+\infty}\left|\omega_+(\xi,\eta,y)\right|^2dy+
	\int_0^{+\infty}\left|\partial_y\omega_+(\xi,\eta,y)\right|^2dy\leq C\left\|E_{s,+}\omega_+(\xi,\eta,\cdot)\right\|^2_{L^2(\RR^+)}.
\end{equation}
Furthermore, if ${\rm supp}(\hat{v}_-(\xi,\eta,\cdot))\subset[-\frac{1}{C},0]$, then
\begin{equation}\label{4-35p}
	\Lambda_1^2\int_{-\infty}^0\left|\hat{v}_-(\xi,\eta,y)\right|^2dy+
	\int_{-\infty}^0\left|\partial_y\hat{v}_-(\xi,\eta,y)\right|^2dy\leq C\left\|E_{s,-}\hat{v}_-(\xi,\eta,\cdot)\right\|^2_{L^2(\RR^+)}+C\Lambda_1\left|\hat{v}_-(\xi,\eta,0)\right|^2.
\end{equation}
\end{lemma}
\begin{proof}
We have, integrating by parts,
\begin{equation}\label{1-21}
	\begin{aligned}
&\left\|E_{s,+}\omega_{s,+}\right\|^2_{L^2(\mathbb{R}_{+})}=\\& \int_0^{+\infty}\!\!\!\!\left|\partial_y\omega_{s,+}+i(T_{+}(\xi,y)-J_ {s,+})\omega_{s,+}\right|^2\\
&+\int_0^{+\infty}\!\!\!\!\left[(\tau\alpha_++\tau\beta y+R_{s,+}+\tau s T_+(\gamma,y))^2+\tau\beta +\partial_yR_{s,+}+\tau s \partial_yT_+(\gamma,y)\right]|\omega_{s,+}|^2dy\\
&+(\tau\alpha_++R_{s,+}(\xi,\eta,0)+\tau s T_+(\gamma,0))|\omega_{s,+}(\xi,\eta,0)|^2.
		\end{aligned}
\end{equation}
By  \eqref{1-15} and \eqref{5-15bis}, for $s$ small, we can write
\begin{equation}\label{3-36p}
\begin{aligned}
(\tau\alpha_++&\tau\beta y+R_{s,+}+\tau s T_+(\gamma,y))^2+\tau\beta +\partial_yR_{s,+}+\tau s \partial_yT_+(\gamma,y)\geq
(R_{s,+})^2-C\Lambda_0\\&\geq \left(\frac{1}{C}-\frac{C}{\Lambda_0}\right)\Lambda_0^2\geq
	\left(\frac{1}{C}-\frac{C}{\tau}\right)\Lambda_0^2	\geq \frac{1}{2C}\Lambda_1^2,
	\end{aligned}
	\end{equation}
for $\tau\geq 2C^2$.

The term $ \int_0^{+\infty}\!\!\!\!\left|\partial_y\omega_{s,+}+i(T_{+}(\xi,y)-J_ {s,+})\omega_{s,+}\right|^2$ can be estimated from velow by the usual trick. Therefore \eqref{1-21} and \eqref{3-36p} yield \eqref{3-35p}.
The proof of \eqref{4-35p} is essentially the same provided ${\rm supp}(\hat{v}_-(\xi,\eta,\cdot))\subset[-\frac{1}{C},0]$.
\end{proof}

\begin{lemma}\label{lem3.9} Assume \eqref{case3}. There exists a positive constant $C$, depending on $\lambda_0,M_0$, such that if $0\leq s\leq s_1$, $\tau\geq C$, and $\beta\geq C$, then, for ${\rm supp}(\hat{v}_+(\xi,\cdot))\subset[0,\frac{1}{\beta}]$, we have that
\begin{equation}\label{1-37p}
\frac{\Lambda_1^2}{\tau}\int^{\infty}_0|\hat{v}_+(\xi,\eta,y)|^2dy+\frac{1}{\tau}\int_0^{\infty}|\partial_y\hat{v}_+(\xi,\eta,y)|^2dy
\leq C\left(||F_{s,+}\hat{v}_{+}(\xi,\eta,\cdot)||^2_{L^2(\mathbb{R}_{+})}+\Lambda_1|\hat{v}_+(\xi,\eta,0)|^2\right).
\end{equation}
\end{lemma}
\begin{proof}
We have, integrating by parts,
\begin{equation}\label{1-22}
	\begin{aligned}
&\left\|F_{s,+}\hat{v}_+\right\|^2_{L^2(\mathbb{R}_{+})}=\\& \int_0^{+\infty}\!\!\!\!\left|\partial_y\hat{v}_++i(T_{+}(\xi,y)+J_ {s,+})\hat{v}_+\right|^2\\
&+\int_0^{+\infty}\!\!\!\!\left[(-\tau\alpha_+-\tau\beta y+R_{s,+}-\tau s T_+(\gamma,y))^2+\tau\beta -\partial_yR_{s,+}+\tau s \partial_yT_+(\gamma,y)\right]|\hat{v}_+|^2dy\\
&+(\tau\alpha_+-R_{s,+}(\xi,\eta,0)+\tau s T_+(\gamma,0))|\hat{v}_+(\xi,\eta,0)|^2.
		\end{aligned}
\end{equation}
Notice that, for $0\leq y\leq \frac{1}{\beta}$ and $0\leq s \leq 1/C$,
\[-\tau\alpha_+-\tau\beta y+R_{s,+}-\tau s T_+(\gamma,y)\geq R_{s,+}-\tau(\alpha_++1+Cs)
\geq R_{s,+}-\tau(\alpha_++2)\geq \frac{\Lambda_1}{C_*}-\tau(\alpha_++2).
\]
We want now to estimate
\[p:=(-\tau\alpha_+-\tau\beta y+R_{s,+}-\tau s T_+(\gamma,y))^2+\tau\beta -\partial_yR_{s,+}+\tau s \partial_yT_+(\gamma,y).\]
We distinguish two cases:

if $\tau\leq \frac{\Lambda_1}{2C_*(\alpha_++2)}$, we can write
\[p\geq \frac{\Lambda_1^2}{(2C_*)^2}+\tau\beta-\overline{C}\Lambda_1\geq \Lambda_1^2\left(\frac{1}{(2C_*)^2}-\frac{\overline{C}}{\Lambda_1}\right)\geq
\Lambda_1^2\left(\frac{1}{(2C_*)^2}-\frac{\overline{C}}{\tau}\right)\geq \frac{\Lambda_1^2}{2(2C_*)^2}\geq \frac{\Lambda^2_1}{2\tau (2C_*)^2} \]
if $\tau\geq 2(2C_*)^2\overline{C}$.

On the other hand, if $\tau\geq \frac{\Lambda_1}{2C_*(\alpha_++2)}$, then
\[p\geq \tau\beta -\overline{C}\Lambda_1\geq \Lambda_1\left(\frac{\beta}{2C_*(\alpha_++1)}-\overline{C}\right)\geq \frac{\beta}{4C_*(\alpha_++1)}\Lambda_1
\geq  \frac{\beta}{4C_*(\alpha_++1)}\frac{\Lambda_1^2}{C\tau}\]
for $\beta\geq 4C_*\overline{C}(\alpha_++1)$ and since $\frac{\Lambda_1}{\tau}\leq C$.

In each of these two cases we have
\begin{equation}\label{1-38p}
p\geq \frac{1}{C}\frac{\Lambda_1^2}{\tau}.
\end{equation}
By \eqref{1-22} and \eqref{1-38p} and by the fact that $\tau\alpha_+-R_{s,+}\geq -C\Lambda_1$, we have
	\begin{equation*}
	\begin{aligned}
\left\|F_{s,+}\hat{v}_+\right\|^2_{L^2(\mathbb{R}_{+})}&\geq \int_0^{+\infty}\!\!\!\!\left|\partial_y\hat{v}_++i(T_{+}(\xi,y)+J_ {s,+})\hat{v}_+\right|^2dy+\frac{\Lambda_1^2}{C\tau}\int_0^{+\infty}\!\!\!\!|\hat{v}_+|^2dy-C\Lambda_1|\hat{v}_+(\xi,\eta,0)|^2\\
&\geq \frac{\varepsilon}{\tau}\int_0^{+\infty} \!\!\!\!\left|\partial_y\hat{v}_++i(T_{+}(\xi,y)+J_ {s,+})\hat{v}_+\right|^2dy+\frac{\Lambda_1^2}{C\tau}\int_0^{+\infty}\!\!\!\!|\hat{v}_+|^2dy-C\Lambda_1|\hat{v}_+(\xi,\eta,0)|^2\\
&\geq \frac{\varepsilon}{2\tau}\int_0^{+\infty} \left|\partial_y\hat{v}_+\right|dy-\frac{2\varepsilon}{\tau}C\Lambda_1^2\int_0^{+\infty}\left|\hat{v}_+\right|^2dy+\frac{\Lambda_1^2}{C\tau}\int_0^{+\infty}\!\!\!\!|\hat{v}_+|^2dy-C\Lambda_1|\hat{v}_+(\xi,\eta,0)|^2\\
&\geq \frac{\varepsilon}{2\tau}\int_0^{+\infty} \left|\partial_y\hat{v}_+\right|dy+\frac{\Lambda_1^2}{2C\tau}\int_0^{+\infty}\!\!\!\!|\hat{v}_+|^2dy-C\Lambda_1|\hat{v}_+(\xi,\eta,0)|^2
		\end{aligned}
\end{equation*}
if $\varepsilon$ is small enough.\end{proof}
\begin{lemma}\label{lem3.10} Assume \eqref{case3}. There exist positive constants $C$, $s_3$ depending on $\lambda_0$, $M_0$, such that if $0\leq s\leq s_3$, and for ${\rm supp} \omega_{s,-}(\xi,\eta,\cdot)\subset[-1/C,0]$, we have that
\begin{equation}\label{1-40p}
\Lambda_1|\omega_{s,-}(\xi,\eta,0)|+\Lambda_1^2\int_{-\infty}^0|\omega_{s,-}(\xi,\eta,y)|^2dy\leq
C\|F_{s,-}\omega_{s,-}(\xi,\eta,\cdot)||_{L^2(\mathbb{R}_{-})}.
\end{equation}
\end{lemma}
\begin{proof}
We first compute by integration by parts,
\begin{equation}\label{4.97}
\begin{aligned}
&\Re\int_{-\infty}^0\Lambda_0 (F_{s,-}\omega_{s,-})\bar{\omega}_{s,-}\,dy\\
=&\frac{1}{2}\Lambda_0|\omega_{s,-}(\xi,\eta,0)|^2+\Lambda_0\int_{-\infty}^0\left(R_{s,-}(\xi,\eta,y)-\tau\alpha_{-}-\tau\beta y-\tau sT_{-}(y,\gamma)\right)|\omega_{s,-}|^2dy.
\end{aligned}
\end{equation}
Since $y\leq 0$  and by \eqref{derivt}  and \eqref{12-15bis} we can write
\begin{equation}\label{1-41p}
\begin{aligned}
R_{s,-}&(\xi,\eta,y)-\tau\alpha_{-}-\tau\beta y-\tau sT_{-}(y,\gamma)\geq R_{s,-}(\xi,\eta,y)-\tau\alpha_{-}-C\tau s\\
&=R_{0,-}(\xi,\eta,0)-\tau\alpha_{-}-\left( R_{0,-}(\xi,\eta,0)-R_{s,-}(\xi,\eta,y)\right)-C\tau s\\
&\geq R_{0,-}(\xi,\eta,0)-\tau\alpha_{-}-C(|y|+s)-C\tau s.
\end{aligned}
\end{equation}
On the other hand, by \eqref{L} and \eqref{case3},
\begin{equation}\label{4.101}
\begin{aligned}
R_{0,-}(\xi,\eta,0)-\tau\alpha_{-}
&\geq \frac{1}{L}\left[R_{0,+}(\xi,\eta,0)-\tau\alpha_{-}L\right]\\
&\geq \frac{1}{L}\left[R_{0,+}(\xi,\eta,0)-\frac{R_{0,+}(\xi,\eta,0)}{(1-\kappa)\alpha_+}\alpha_{-}L\right]
=\frac{R_{0,+}(\xi,\eta,0)}{L}\left[1-\frac{\alpha_{-}L}{(1-\kappa)\alpha_+}\right]\\
&=\frac{\kappa}{1-\kappa}\frac{R_{0,+}(\xi,\eta,0)}{L}\geq \frac{\Lambda_0}{C}.
\end{aligned}\end{equation}

By \eqref{4.97}, \eqref{1-41p} and \eqref{4.101} and for sufficiently small $|y|$ and $s$ we have
\begin{equation*}
\Re\int_{-\infty}^0\Lambda_0 (F_{s,-}\omega_{s,-})\bar{\omega}_{s,-}\,dy\geq \frac{\Lambda_0}{2}|\omega_{s,-0}(\xi,\eta,0)|^2 +\frac{\Lambda_0^2}{C}\int_{-\infty}^0|\omega_{s,-}|^2dy
\end{equation*}
which implies \eqref{1-40p} by \eqref{equivlambda}.\end{proof}

\begin{lemma}\label{lem3.11}
Assume \eqref{case3}. There exist constants $C$ and $s_4=\min\{s_1,s_2,s_3\}$, depending only on $\lambda_0$ and $M_0$, such that if $0\leq s\leq s_4$,  $\tau\geq C$, $\beta\geq C$, then for ${\rm supp}(\hat{v}(\xi,\cdot))\subset\left[-\frac{1}{C},\frac{1}{C}\right]$ we have
\begin{equation}\label{4.102}
\begin{aligned}
&\Lambda_1|V_+(\xi,\eta)+a_{nn}^+(0)\sqrt{\zeta_{s,+}(\xi,\eta,0)}v_+(\xi,\eta,0)|^2+\Lambda_1^2||F_{s,+}\hat{v}_{+}(\xi,\eta,\cdot)|
|^2_{L^2(\mathbb{R}_{+})}\\
&\quad\leq C\left(||P_{s,+}\hat{v}_{+}||^2_{L^2(\mathbb{R}_{+})}+\Lambda_1^2||\hat{v}_{+}||^2_{L^2(\mathbb{R}_{+})}\right).
\end{aligned}
\end{equation}
and
\begin{equation}\label{4.103}
\begin{aligned}
&\Lambda_1|V_-(\xi,\eta)-a_{nn}^-(0)\sqrt{\zeta_{s,-}(\xi,\eta,0)}\hat{v}_-(\xi,\eta,0)|^2+\Lambda_1^2|
|E_{s,-}\hat{v}_{-}(\xi,\eta,\cdot)||^2_{L^2(\mathbb{R}_{-})}\\
&\quad\leq C\left(||P_{s,-}\hat{v}_{-}||^2_{L^2(\mathbb{R}_{-})}+\Lambda_1^2||\hat{v}_{-}||^2_{L^2(\mathbb{R}_{-})}\right).
\end{aligned}
\end{equation}
\end{lemma}
\begin{proof}
Inequality \eqref{4.102} follows from \eqref{3-35p} and \eqref{6-27p}. Similarly, \eqref{4.103} follows from \eqref{1-40p} and \eqref{7-27p}.
\end{proof}

\begin{lemma}\label{lem3.12}
Assume \eqref{case3}. There exist constants $C$ and $s_4$, depending only on $\lambda_0$ and $M_0$, such that if $0\leq s\leq s_4$,  $\tau\geq C$, $\beta\geq C$, then for ${\rm supp}(\hat{v}(\xi,\eta,\cdot))\subset\left[-\frac{1}{C},\frac{1}{C}\right]$ we have
\begin{equation}\label{1-44p}
\begin{aligned}
\Lambda_1&\sum_\pm|V_\pm(\xi,\eta)|^2+\Lambda_1^3\sum_\pm|\hat{v}_{\pm}(\xi,\eta,0)|^2+\frac{\Lambda_1^4}{\tau} \sum_\pm\|\hat{v}_{\pm}(\xi,\eta,\cdot)\|^2_{L^2(\mathbb{R}_{\pm})}\\
&+\frac{\Lambda_1^2}{\tau} \sum_\pm\|\partial_y\hat{v}_{\pm}(\xi,\eta,\cdot)\|^2_{L^2(\mathbb{R}_{\pm})}
\leq C\left(\sum_\pm||P_{s,+}\hat{v}_{+}||^2_{L^2(\mathbb{R}_{\pm})}+\Lambda_1^3|\sigma_0(\xi,\eta)|^2 +\Lambda_1|\sigma_1(\xi,\eta)|^2\right).
\end{aligned}
\end{equation}
\end{lemma}
\begin{proof}
By recalling \eqref{sigma0} and \eqref{sigma1} and by \eqref{1-15bis}, we have
\begin{equation}\label{3-44p}
\begin{aligned}
&\left|\left(a^+_{nn}(0)\sqrt{\zeta_{s,+}(\xi,\eta,0)}+a^-_{nn}(0)\sqrt{\zeta_{s,-}(\xi,\eta,0)}\right)\hat{v}_+(\xi,\eta,0)\right|\\
&\quad=\left|\left(V_++a^+_{nn}(0)\sqrt{\zeta_{s,+}}\,\hat{v}_+\right)-\left(V_--a^-_{nn}(0)\sqrt{\zeta_{s,-}}\,\hat{v}_-\right)-\sigma_1+a^-_{nn}(0)\sqrt{\zeta_{s,-}}\,\sigma_0\right|	 \\
&\quad\quad\leq \left|V_++a^+_{nn}(0)\sqrt{\zeta_{s,+}}\,\hat{v}_+\right|+\left|V_--a^-_{nn}(0)\sqrt{\zeta_{s,-}}\,\hat{v}_-\right|+|\sigma_1|+C\Lambda_1|\sigma_0|.
 \end{aligned}
\end{equation}
On the other hand, by  \eqref{1-15}
\begin{equation}\label{2-44p}
\begin{aligned}
&\left|\left(a^+_{nn}(0)\sqrt{\zeta_{s,+}(\xi,\eta,0)}+a^-_{nn}(0)\sqrt{\zeta_{s,-}(\xi,\eta,0)}\right)\hat{v}_+(\xi,\eta,0)\right|\\
&\quad\geq\left(a^+_{nn}(0)R_{s,+}
(\xi,\eta,0)+a^-_{nn}(0)R_{s,-}(\xi,\eta,0)\right)\left|\hat{v}_+(\xi,\eta,0)\right|\geq\frac{1}{C}\Lambda_1\left|\hat{v}_+(\xi,\eta,0)\right|.
 \end{aligned}
\end{equation}
By \eqref{3-44p} and \eqref{2-44p} we have
\begin{equation}\label{1-45p}
\Lambda_1^3\left|\hat{v}_+(\xi,\eta,0)\right|^2\leq C \Lambda_1\left\{\left|V_++a^+_{nn}(0)\sqrt{\zeta_{s,+}}\,\hat{v}_+\right|^2
+\left|V_--a^-_{nn}(0)\sqrt{\zeta_{s,-}}\,\hat{v}_-\right|^2+|\sigma_1|^2+\Lambda_1^2|\sigma_0|^2\right\}.
\end{equation}
Recalling \eqref{sigma0}, by \eqref{1-45p}, \eqref{4.102}, \eqref{4.103}, \eqref{3-44p} and \eqref{2-44p} we have
\begin{equation}\label{2-45p}
\Lambda_1^3\left|\hat{v}_+(\xi,\eta,0)\right|^2\leq C\left(\sum_\pm\left(\|P_{s,\pm}\hat{v}\pm\|^2_{L^2(\RR_\pm)}+\Lambda_1^2\|\hat{v}\pm\|^2_{L^2(\RR_\pm)}\right)+ \Lambda_1|\sigma_1|^2+\Lambda_1^3|\sigma_0|^2\right).
\end{equation}
By triangle inequality
\begin{equation*}
\begin{aligned}
&\left|V_+\right|^2\leq\left(\left|V_++a^+_{nn}(0)\sqrt{\zeta_{s,+}}\,\hat{v}_+\right|+
\left|a^+_{nn}(0)\sqrt{\zeta_{s,+}}\,\hat{v}_+\right|\right)^2\\
&\quad
\leq 2 \left|V_++a^+_{nn}(0)\sqrt{\zeta_{s,+}}\,\hat{v}_+\right|^2+C\Lambda_1^2 \left|\hat{v}_+(\xi,\eta,0)\right|^2,
\end{aligned}\end{equation*}
and, by \eqref{sigma1}, \eqref{2-45p} and Lemma \ref{lem3.11}, we have
\begin{equation}\label{2-46p}
\Lambda_1\sum_\pm\left|V_+(\xi,\eta)\right|^2\leq C\left(\sum_\pm\left(\|P_{s,\pm}\hat{v}\pm\|^2_{L^2(\RR_\pm)}+\Lambda_1^2\|\hat{v}\pm\|^2_{L^2(\RR_\pm)}\right)+ \Lambda_1|\sigma_1|^2+\Lambda_1^3|\sigma_0|^2\right).
\end{equation}
Now, by \eqref{4-35p}, \eqref{1-37p}, \eqref{4.102}, \eqref{4.103} and \eqref{2-45p}, we get
\begin{equation}\label{3-46p}
\begin{aligned}
&\frac{\Lambda_1^4}{\tau}\sum_\pm\|\hat{v}\pm\|^2_{L^2(\RR_\pm)}+\frac{\Lambda_1^2}{\tau}\sum_\pm\|\partial_y\hat{v}\pm\|^2_{L^2(\RR_\pm)}
\\
&\quad
\leq C\left(\sum_\pm\left(\|P_{s,\pm}\hat{v}\pm\|^2_{L^2(\RR_\pm)}+\Lambda_1^2\|\hat{v}\pm\|^2_{L^2(\RR_\pm)}\right)+ \Lambda_1|\sigma_1|^2+\Lambda_1^3|\sigma_0|^2\right).
\end{aligned}\end{equation}
Now, by adding up \eqref{2-45p}, \eqref{2-46p} and \eqref{3-46p} and by absorbing the term $C\Lambda_1^2\sum_\pm\|\hat{v}_\pm\|$ by the left hand side, we finally get \eqref{1-44p}.
\end{proof}
\subsubsection{Conclusion}
By putting together the three cases
and noticing that, by definition of $P_{s,\pm}$, we can write
\[|\partial^2_y \hat{v}_\pm|\leq C\left(\left|P_{s,\pm}\hat{v}_\pm\right|+\Lambda_1\left|\partial_y\hat{v}_\pm\right|+\Lambda_1^2
\left|\hat{v}_\pm\right|\right)\]
we finally proved Proposition \ref{prop3.1}.
\section{Step 2 - Carleman estimate for general coefficients with weight independent of $t$}\label{step2}

In the previous section we have proved the Carleman estimate when $A_\pm=A_\pm(y)$. Now we want to derive it for $A_\pm(x,t,y)$.
To achieve this purpose, similarly to the elliptic case,  we approximate $A_\pm(x,y,t)$ with coefficients depending on $y$ only and we make use of a special kind of partition of unity introduced in the next section. At the same time we consider a weight function that is quadratic in $x$.
\subsection{Partition of unity and auxiliary results}\label{preliminary}
In this section we collect some results on a partition of unity that we  use in our proof and we describe how this partition of unity behaves with respect to the function spaces that we use.

\noindent Let $\vartheta_0, \widetilde{\vartheta}_0\in C^\infty_0(\mathbb{R})$ such that
\begin{equation*}
0\leq\vartheta_0, \widetilde{\vartheta}_0\leq 1\mbox{ for }z\in\mathbb{R},\quad  \vartheta_0(z)=\widetilde{\vartheta}_0(z)=1\mbox{ for }z\in[-1,1] \end{equation*}
and
\begin{equation*}
{\rm supp}\,\vartheta_0\subset (-3/2,3/2),\quad  {\rm supp}\,\widetilde{\vartheta}_0\subset (-9/4,9/4).
\end{equation*}
Let  $\vartheta_{n-1}(x)=\vartheta_0(x_1)\cdots\vartheta_0(x_{n-1})$.
Given $\mu\geq 1$ and $g\in \mathbb{Z}^{n}$, $g=(g',g_n)$, where $g'\in \mathbb{Z}^{n-1}$ and $g_n\in \mathbb{Z}$, we define
\[X_g=(x_{g'},t_{g_n})=(g'/\mu,g_n/\mu^2)\]
and
\[\vartheta_{g,\mu}(X)=\vartheta_{n-1}(\mu(x-x_{g'}))\widetilde{\vartheta}_0(\mu^2(t-t_{g_n})).\]
Notice that
\begin{equation*}
{\rm supp}\,\vartheta_{g,\mu}\subset \stackrel{\circ}{Q}_{3/2\mu}(X_g)\subset Q_{2/\mu}(X_g),
\end{equation*}
\begin{equation}\label{6.4}
|D_x^k\vartheta_{g,\mu}|\leq C_1\mu^k(\chi_{ Q_{3/2\mu}(X_g)}-\chi_{Q_{1/\mu}(X_g)}),\quad k=0,1,2,
\end{equation}
and
\begin{equation*}
|\partial_t\vartheta_{g,\mu}|\leq C_1\mu^2(\chi_{ Q_{3/2\mu}(X_g)}-\chi_{Q_{1/\mu}(X_g)}),
\end{equation*}
where $C_1\geq 1$ depends only on $n$.

Since, for any $\overline{g}\in\mathbb{Z}^n$,
\begin{equation}\label{cardAg}
card\left(\{g\in \mathbb{Z}^{n}\,:\,{\rm supp}\,\vartheta_{g,\mu}\cap{\rm supp}\,\vartheta_{\overline{g},\mu}\neq\emptyset\}\right)=5^{n},\end{equation}
we can define
\begin{equation*}
\bar{\vartheta}_{\mu}(X):=\sum_{g\in \mathbb{Z}^{n}}\vartheta_{g,\mu}(X)\geq 1,\quad X\in \mathbb{R}^{n}
\end{equation*}
and
\[\eta_{g,\mu}(X)=\vartheta_{g,\mu}(X)/\bar{\vartheta}_{\mu}(X),\quad X\in \mathbb{R}^{n}.\]
Hence we have that
\begin{equation}\label{6.7}
\begin{cases}
\sum_{g\in \mathbb{Z}^{n-1}}\eta_{g,\mu}(X)= 1,\quad X\in \mathbb{R}^{n},\\
{\rm supp}\,\eta_{g,\mu}\subset Q_{3/2\mu}(X_g)\subset Q_{2/\mu}(X_g),\\
|D_x^k\partial^m_t\eta_{g,\mu}|\leq C_2\mu^{k+2m}\chi_{Q_{3/2\mu}(X_g)},\quad k=0,1,2,\quad m=0,1,
\end{cases}
\end{equation}
where $C_2\geq 1$ depends on $n$.
In Section \ref{notstat} we have recalled the definition of $H^{\frac{1}{2},0}(\mathbb{R}^{n})$, $H^{0,\alpha}(\mathbb{R}^{n})$, $\alpha\in[0,1)$ and their seminorms $[\cdot]_{\frac{1}{2},0,\mathbb{R}^{n}}$, $[\cdot]_{0,\alpha,\mathbb{R}^{n}}$, respectively, in what follows
we also need the seminorms
\begin{equation*}
[f]_{\frac{1}{2},0,Q_{r}}=\left[\int_{I_r}\int_{Q'_{r}}\int_{Q'_{r}}\frac{|f(x_1,t)-f(x_2,t)|^2}{|x_1-x_2|^n}dx_1dx_2dt\right]^{1/2},
\end{equation*}

\begin{equation*}
[f]_{0,\alpha, Q_{r}}=\left[\int_{Q'_r}\int_{I_{r}}\int_{I_{r}}\frac{|f(x,t_1)-f(x,t_2)|^2}{|t_1-t_2|^{1+2\alpha}}dt_1dt_2dx\right]^{1/2},
\end{equation*}
where $Q_{r}=Q_{r}(0)$.

In the rest of this subsection we give the statements of some lemmas and propositions that we use in the sequel, their proofs are the same of the ones given in \cite{dcflvw} with the obvious changes. Since the constants of various inequalities always depends on $n$ or on $\alpha$ we will omit such dependence.
\begin{lemma}\label{lem5.2}
Let $f\in C^\infty(\mathbb{R}^{n})$ and ${\rm supp} f\subset Q_{3r/4}$ for some $r\leq 1$. There exists a constant C such that we have
\begin{equation*}
[f]^2_{\frac{1}{2},0,Q_{r}}+\frac{C^{-1}}{r}\int_{Q_{r}}|f(x,t)|^2dX\leq[f]^2_{\frac{1}{2},0,\mathbb{R}^{n}}
\leq [f]^2_{\frac{1}{2},0,Q_{r}}+\frac{C}{r}\int_{Q_{r}}|f(x,t)|^2dX,
\end{equation*}

\begin{equation}\label{5.6t}
[f]^2_{0,\alpha,Q_{r}}+\frac{C^{-1}}{r^{4\alpha}}\int_{Q_{r}}|f(x,t)|^2dX\leq[f]^2_{0,\alpha,\mathbb{R}^{n}}
\leq [f]^2_{0,\alpha,Q_{r}}+\frac{C}{r^{4\alpha}}\int_{Q_{r}}|f(x,t)|^2dX,
\end{equation}
\end{lemma}
\begin{pr}\label{pr6.1}
Let  $\{\varsigma_g\}_{g\in \mathbb{Z}^{n}}$ be a family of smooth functions such that ${\rm supp}\,\varsigma_{g}$ in contained in the interior of $Q_{3/2\mu}(X_g)$, then
\begin{equation*}
[\sum_{g\in \mathbb{Z}^{n}}\varsigma_{g}]^2_{\frac{1}{2},0,\mathbb{R}^{n}}\leq C\left(\sum_{g\in \mathbb{Z}^{n}}[\varsigma_{g}]^2_{\frac{1}{2},0,Q_{\frac{2}{\mu}}(X_g)}
+\sum_{g\in \mathbb{Z}^{n}}\mu\int_{Q_{\frac{2}{\mu}}(X_g)}|\varsigma_{g}|^2\right),
\end{equation*}
\begin{equation}\label{6.8t}
[\sum_{g\in \mathbb{Z}^{n}}\varsigma_{g}]^2_{0,\alpha,\mathbb{R}^{n}}\leq C\left(\sum_{g\in \mathbb{Z}^{n}}[\varsigma_{g}]^2_{0,\alpha,Q_{\frac{2}{\mu}}(X_g)}
+\sum_{g\in \mathbb{Z}^{n}}\mu^{4\alpha}\int_{Q_{\frac{2}{\mu}}(X_g)}|\varsigma_{g}|^2\right).
\end{equation}
\end{pr}
\begin{pr}\label{pr6.4}
Let  $F\in C^\infty(\mathbb{R}^{n})$ with ${\rm supp}\, F\subset Q_{3/2\mu}(X_g)$, let $\beta\in(\alpha,1]$ and let $a$ be a function satisfying
\begin{equation}\label{6.25}
|a(x,t)|\leq E_a,\quad
|a(x',t')-a(x'',t'')|\leq K_a|x'-x''|+\widetilde{K}_a|t'-t''|^{\beta},
\end{equation}
for $(x,t),(x',t'),(x'',t'')\in Q_{3/2\mu}(X_g)$ and $E_a$, $K_a$, $\widetilde{K}_a$  positive constants. Then we have
\begin{equation}\label{6.26}
[a F]^2_{\frac{1}{2},0,Q_{\frac{2}{\mu}}(X_g)}\leq C\left( E_a^2[F]^2_{1/2,Q_{\frac{2}{\mu}}(X_g)}+K_a^2\mu^{-1}\int_{Q_{\frac{2}{\mu}}(X_g)}|F(x,t)|^2dX\right)
\end{equation}
and
\begin{equation}\label{6.26t}
[a F]^2_{0,\alpha,Q_{\frac{2}{\mu}}(X_g)}\leq C\left( E_a^2[F]^2_{0,\alpha,Q_{\frac{2}{\mu}}(X_g)}+\frac{\widetilde{K}_a^2\mu^{-4(\beta-\alpha)}}{\beta-\alpha}\int_{Q_{\frac{2}{\mu}}(X_g)}|F(x,t)|^2dX\right).
\end{equation}
\end{pr}
\begin{proof}
The proof of \eqref{6.26} is exactly the same to the proof of \cite[Proposition 4.2]{dcflvw}, hence we limit ourselves to the proof of \eqref{6.26t}.

We have
\begin{equation*}
\begin{aligned}
&[a F]^2_{0,\alpha,Q_{\frac{2}{\mu}}(X_g)}
=\int_{Q'_{\frac{2}{\mu}}(x_g')}\int_{I_{\frac{2}{\mu}}(t_{g_n})}\int_{I_{\frac{2}{\mu}}(t_{g_n})}
\frac{|a(x,t_1)F(x,t_1)-a(x,t_2)F(x,t_2)|^2}{|t_1-t_2|^{1+2\alpha}}dt_1dt_2dx\\
&\leq2\int_{Q'_{\frac{2}{\mu}}(x_g')}\int_{I_{\frac{2}{\mu}}(t_{g_n})}\int_{I_{\frac{2}{\mu}}(t_{g_n})}\left(\frac{|a(x,t_1)|^2\,|F(x,t_1) -F(x,t_2)|^2}{|t_1-t_2|^{1+2\alpha}}+\frac{|F(x,t_2)|^2\,| a(x,t_1)-a(x,t_2)|^2}{|t_1-t_2|^{1+2\alpha}}\right)dt_1dt_2dx\\
&\leq 2 E_a^2[F]^2_{0,\alpha,Q_{\frac{2}{\mu}}(X_g)}+
2\widetilde{K}^2_a\int_{Q'_{\frac{2}{\mu}}(x_g')}\int_{I_{\frac{2}{\mu}}(t_{g_n})}|F(x,t_2)|^2\int_{I_{\frac{2}{\mu}}(t_{g_n})}
|t_1-t_2|^{-1+2(\beta-\alpha)}dt_1dt_2dx\\
&\leq 2 E_a^2[F]^2_{0,\alpha,Q_{\frac{2}{\mu}}(X_g)}+C\frac{\widetilde{K}^2_a\mu^{-4(\beta-\alpha)}}{\beta-\alpha}
\int_{Q_{\frac{2}{\mu}}(X_g)}|F(x,t)|^2dX.
\end{aligned}
\end{equation*}
\end{proof}
\begin{pr}\label{pr6.2}
Let  $f\in C^\infty(\mathbb{R}^{n})\cap H^{\frac{1}{2},\alpha}(\mathbb{R}^{n})$. Then
\begin{equation}\label{6.16}
\sum_{g\in \mathbb{Z}^{n}}[f \eta_{g,\mu}]^2_{\frac{1}{2},0,Q_{\frac{2}{\mu}}(X_g)}\leq C\left( [f]^2_{\frac{1}{2},0,\mathbb{R}^{n}}+\mu\int_{\mathbb{R}^{n}}|f(x,t)|^2dX\right),
\end{equation}
and
\begin{equation}\label{6.16t}
\sum_{g\in \mathbb{Z}^{n}}[f \eta_{g,\mu}]^2_{0,\alpha,Q_{\frac{2}{\mu}}(X_g)}\leq C\left( [f]^2_{0,\alpha,\mathbb{R}^{n}}+\mu^{4\alpha}\int_{\mathbb{R}^{n}}|f(x,t)|^2dX\right).
\end{equation}
\end{pr}
\begin{pr}\label{pr6.3}
Let  $f\in C^\infty(\mathbb{R}^{n-1})\cap H^{\frac{1}{2},\alpha}(\mathbb{R}^{n})$. Then
\begin{equation}\label{6.20}
\sum_{g\in \mathbb{Z}^{n}}[f\, D_x\eta_{g,\mu}]^2_{\frac{1}{2},0,Q_{\frac{2}{\mu}}(X_g)}\leq C\left(\mu^2[f]^2_{\frac{1}{2},0,\mathbb{R}^{n}}+\mu^3\int_{\mathbb{R}^{n}}|f(x,t)|^2dX\right).
\end{equation}
\begin{equation}\label{6.20t}
\sum_{g\in \mathbb{Z}^{n}}[f\, D_x\eta_{g,\mu}]^2_{0,\alpha,Q_{\frac{2}{\mu}}(X_g)}\leq C\left(\mu^2[f]^2_{0,\alpha,\mathbb{R}^{n}}+\frac{\mu^{2+4\alpha}}{1-\alpha}\int_{\mathbb{R}^{n}}|f(x,t)|^2dX\right).
\end{equation}
\end{pr}
\subsection{Estimate of the left hand side of the Carleman estimate, I}
In the present subsection and in the next one we derive the Carleman estimate for general coefficients. In order to make clear the procedure that we follow let us introduce and recall some notation and some definitions. For any $0<\delta\leq 1$ we define
\begin{equation}\label{7.7}
A^\delta_{\pm}(x,t,y):=A_{\pm}(\delta x,\delta^2t,\delta y),
\end{equation}
\begin{equation*}
\mathcal{L}_\delta(w):=\sum_{\pm}H_{\pm}\left({-\der_t w_\pm+\rm div}_{x,y}(A^\delta_{\pm}(x,t,y)D_{x,y}w_{\pm})\right),
\end{equation*}
and
\begin{equation*}
\begin{cases}
h_{0}(w):=w_+(x,t,0)-w_-(x,t,0),\\
h_{1}(w):=A^\delta_+(x,t,0)D_{x,y}w_+(x,t,0)\cdot \nu-A^\delta_-(x,t,0)(x,t,0)D_{x,y}w_-(x,t,0)\cdot \nu,
\end{cases}
\end{equation*}
where $\nu=-e_n$.

Next, for any $g=(g',g_n)\in \mathbb{Z}^{n}$ we define
\begin{equation*}
\begin{cases}
A^{\delta,g}_{\pm}(y):=A^\delta_{\pm}(X_g,y)=A_{\pm}(\delta x_{g'},\delta^2t_{g_n}\delta y),\\
\mathcal{L}_{\delta,g}(w):=\sum_{\pm}H_{\pm}\left(-\der_t w_\pm+{\rm div}_{x,y}(A^{\delta,g}_{\pm}(y)D_{x,y}w_{\pm})\right).
\end{cases}
\end{equation*}
Notice that
\begin{equation*}
\lambda_0|z|^2\leq A^{\delta,g}_{\pm}(y)z\cdot z\leq \lambda^{-1}_0|z|^2,\, \forall y\in \mathbb{R},\,\forall\, z\in \mathbb{R}^n
\end{equation*}
and
\begin{equation*}
|A^{\delta,g}_{\pm}(y')-A^{\delta,g}_{\pm}(y)|\leq M_0\delta|y'-y|.
\end{equation*}
Concerning the weight functions, let us introduce the following notation.
\[
\begin{cases}
\kappa_\varepsilon(x):=-\varepsilon|x|^2/2,\\
\psi_\varepsilon(x,y):=\varphi(y)+\kappa_\varepsilon(x),\\
\psi_{\varepsilon,g'}(x,y):=\varphi(y)+D_x\kappa_\varepsilon(x_{g'})\cdot(x-x_{g'})+\kappa_\varepsilon(x_{g'}),
\end{cases}
\]
where $\varphi(y)$ is defined in \eqref{2.1}. In addition we assume that $\alpha_{+}, \alpha_{-}, \beta$ are fixed positive numbers such that Theorem \ref{pr22} holds for the operator $\mathcal{L}_{\delta,g}$.

Notice that
\begin{equation}\label{aggiunta}
e^{\tau\psi_\varepsilon}\leq e^{\tau\psi_{\varepsilon,g'}}\leq e^{2(n-1)\frac{\varepsilon\tau}{\mu^2}}e^{\tau\psi_\varepsilon}\mbox{ in }Q'_{\frac{2}{\mu}}(x_{g'}).
\end{equation}

In order to estimate the left hand side of \eqref{8.24} we define

\begin{equation*}
\begin{aligned}
\Xi_E(w):=&\sum_{\pm}\sum_{k=0}^2\tau^{3-2k}\int_{\mathbb{R}^{n+1}_{\pm}}|D_{x,y}^k{w}_{\pm}|^2e^{2\tau\psi_{\varepsilon}(x,y)}dXdy
+\sum_{\pm}\sum_{k=0}^1\tau^{3-2k}\int_{\mathbb{R}^{n}}|D_{x,y}^k{w}_{\pm}(x,t,0)|^2e^{2\psi_{\varepsilon}(x,0)}dX\\
&+\sum_{\pm}\tau^2[(e^{\tau\psi_{\varepsilon}}w_{\pm})(\cdot,\cdot,0)]^2_{\frac{1}{2},0,\mathbb{R}^{n}}
+\sum_{\pm}[\partial_y(e^{\tau\psi_{\varepsilon,\pm}}w_{\pm})(\cdot,\cdot,0)]^2_{\frac{1}{2},0,\mathbb{R}^{n}}
+\sum_{\pm}[D_x(e^{\tau\psi_{\varepsilon}}w_{\pm})(\cdot,\cdot,0)]^2_{\frac{1}{2},0,\mathbb{R}^{n}},\\
\end{aligned}
\end{equation*}

\begin{equation}\label{l.s.t-2}
\begin{aligned}
\Xi_N(w):=&\tau^{-1}\sum_{\pm}\int_{\mathbb{R}^{n+1}_{\pm}}|\partial_t{w}_{\pm}|^2e^{2\tau\psi_{\varepsilon}(x,y)}dXdy
+\sum_{\pm}\tau^2[(e^{\tau\psi_\varepsilon}w_{\pm})(\cdot,\cdot,0)]^2_{0,\frac{1}{4},\mathbb{R}^{n}}\\
&+\sum_{\pm}[(e^{\tau\psi_\varepsilon}w_{\pm})(\cdot,\cdot,0)]^2_{0,\frac{3}{4},\mathbb{R}^{n}}
+\sum_{\pm}[\partial_y(e^{\tau\psi_{\varepsilon,\pm}}w_{\pm})(\cdot,\cdot,0)]^2_{0,\frac{1}{4},\mathbb{R}^{n}}\\
\end{aligned}
\end{equation}
and
\begin{equation}\label{l.s.}
\Xi(w):=\Xi_E(w)+\Xi_N(w).
\end{equation}

Roughly speaking, $\Xi_E(w)$ behaves similarly to the corresponding elliptic term defined in \cite[(4.23)]{dcflvw} and $\Xi_N(w)$ is the additional contribution that arises from the parabolic operator.

If we assume that ${\rm supp}\, w\subset \mathfrak{U}:=Q_{1/2}\times [-r_0,r_0]$ and that
\begin{equation}\label{7.17}
\tau\geq 1/\varepsilon\quad\text{and}\quad  \mu=(\varepsilon\tau)^{1/2},
\end{equation}
then arguing as in \cite[Sect. 4.2]{dcflvw} we have
\begin{equation}\label{7.29e}
\Xi_E(w)\leq C\sum_{g\in \mathbb{Z}^{n}} \Xi_E(w\eta_{g,\mu})+C\widetilde{R}_{0,E},
\end{equation}
where
\begin{equation*}
\widetilde{R}_{0,E}:= (\varepsilon\tau)^{1/2}\sum_{\pm}\int_{\mathbb{R}^{n}}(|D_{x,y}w_{\pm}(x,t,0)|^2
+\tau^2|w_{\pm}(x,t,0)|^2)e^{2\tau\psi_{\varepsilon}(x,0)}dX
\end{equation*}
and $C$ depends only on $\lambda_0,M_0$.

In the rest of the present subsection we prove that
\begin{equation}\label{7.29t}
\Xi_N(w)\leq C\sum_{g\in \mathbb{Z}^{n}} \Xi_N(w\eta_{g,\mu})+C\widetilde{R}_{0,N},
\end{equation}
where
\begin{equation*}
\widetilde{R}_{0,N}:= (\varepsilon\tau)^{1/2}\sum_{\pm}\int_{\mathbb{R}^{n}}(|\partial_yw_{\pm}(x,t,0)|^2+
\tau^2|w_{\pm}(x,t,0)|^2)e^{2\tau\psi_{\varepsilon}(x,0)}dX
\end{equation*}
and $C$ depends only on $\lambda_0,M_0$.

By \eqref{6.7}, we can write
\begin{equation}\label{7.12}
w_{\pm}(x,t,y)=\sum_{g\in \mathbb{Z}^{n}}w_{\pm}(x,t,y)\eta_{g,\mu}(x,t).
\end{equation}
From \eqref{cardAg}, \eqref{aggiunta} and \eqref{7.12}, we get the following estimate from above of the first term at the righthand side of \eqref{l.s.t-2}
\begin{equation}\label{7.15}
\tau^{-1}\int_{\mathbb{R}^{n+1}_{\pm}}|\partial_t{w}_{\pm}|^2e^{2\tau\psi_{\varepsilon}(x,y)}dXdy\leq	C \tau^{-1}\sum_{g\in \mathbb{Z}^{n}}\int_{\mathbb{R}^{n+1}_{\pm}}|\partial_t({w}_{\pm}\eta_{g,\mu})|^2e^{2\tau\psi_{\varepsilon,g'}(x,y)}dXdy
\end{equation}

In the next Lemma we give some estimates that will be useful in this subsection as well as in subsection \ref{subsEII}

\begin{lemma}\label{lemma7.2t}
If $\alpha\in[0,1)$ and ${\rm supp}\, f\subset Q_{3/2\mu}(X_g)$, then we have that
\begin{equation}\label{7.191t}
[f e^{\tau\psi_{\varepsilon}(\cdot,0)}]^2_{0,\alpha,Q_{2/\mu}(X_g)}\leq[f e^{\tau\psi_{\varepsilon,g'}(\cdot,0)}]^2_{0,\alpha,Q_{2/\mu}(X_g)}
\leq C [f e^{\tau\psi_{\varepsilon}(\cdot,0)}]^2_{0,\alpha,Q_{2/\mu}(X_g)},
\end{equation}
\begin{equation}\label{4.32e}
[f e^{\tau\psi_{\varepsilon}(\cdot,0)}]^2_{\frac{1}{2},0,Q_{2/\mu}(X_g)}\leq[f e^{\tau\psi_{\varepsilon,g'}(\cdot,0)}]^2_{\frac{1}{2},0,Q_{2/\mu}(X_g)}
\leq C [f e^{\tau\psi_{\varepsilon}(\cdot,0)}]^2_{\frac{1}{2},0,Q_{2/\mu}(X_g)},
\end{equation}

\end{lemma}
\begin{proof} For sake of shortness, we only show the proof the inequality on right of \eqref{7.191t}. The proof of inequality on the left is similar and the proof of \eqref{4.32e} is the same of \cite[Lemma 4.2]{dcflvw}.

Denote by
\[
F=fe^{\tau\psi_{\varepsilon}(\cdot,0)}.
\]

By \eqref{aggiunta} and by the second of \eqref{7.17} we have
\begin{equation*}
\begin{aligned}
&[fe^{\tau\psi_{\varepsilon,g'}(\cdot,0)}]^2_{0,\alpha,Q_{2/\mu}(X_g)}
=[Fe^{\tau(\psi_{\varepsilon,g'}-\psi_{\varepsilon})(\cdot,0)}]^2_{0,\alpha,Q_{2/\mu}(X_g)}\\
&\quad
\leq e^{4(n-1)}[F]^2_{0,\alpha,Q_{2/\mu}(X_g)}= e^{4(n-1)}[f e^{\tau\psi_{\varepsilon}(\cdot,0)}]^2_{0,\alpha,Q_{2/\mu}(X_g)},
\end{aligned}
\end{equation*}
that is \eqref{7.191t}.
\end{proof}

By \eqref{5.6t}, \eqref{7.12}, \eqref{6.8t} and \eqref{7.191t}  we have
 \begin{equation*}
\begin{aligned}
\tau^2[(e^{\tau\psi_\varepsilon}w_{\pm})(\cdot,\cdot,0)]^2_{0,\frac{1}{4},\mathbb{R}^{n}}
=& C\tau^2\sum_{g\in \mathbb{Z}^{n}}\left([(e^{\tau\psi_\varepsilon,g'}w_{\pm}\eta_{g,\mu})\right.(\cdot,\cdot,0)]^2_{0,\frac{1}{4},Q_{2/\mu}(X_g)}\\
&+(\varepsilon\tau)^{1/2}\int_{Q_{2/\mu}(X_g)}\left.|(w_{\pm}\eta_{g,\mu})(x,t,0)|^2e^{2\tau\psi_\varepsilon}(x,0)dX\right)\\
\leq C\tau^2\sum_{g\in \mathbb{Z}^{n}}[(e^{\tau\psi_\varepsilon,g'}&w_{\pm}\eta_{g,\mu})(\cdot,\cdot,0)]^2_{0,\frac{1}{4},\mathbb{R}^n}+
(\varepsilon\tau)^{1/2}\tau^2\int_{\mathbb{R}^n}|w_{\pm}(x,t,0)|^2e^{2\tau\psi_{\varepsilon}(x,0)}dX.
\end{aligned}
\end{equation*}

Similarly we estimate the third and the fourth term at the righthand side of \eqref{l.s.t-2}  so that, taking into account \eqref{7.15}, \eqref{7.29t} follows.

Finally, \eqref{7.29e} and \eqref{7.29t} give
\begin{equation}\label{7.29final}
\Xi(w)\leq C\sum_{g\in \mathbb{Z}^{n}} \Xi(w\eta_{g,\mu})+C\widetilde{R}_0,
\end{equation}
where
\begin{equation*}
\widetilde{R}_0:= (\varepsilon\tau)^{1/2}\sum_{\pm}\int_{\mathbb{R}^{n}}(|D_{x,y}w_{\pm}(x,t,0)|^2
+\tau^2|w_{\pm}(x,t,0)|^2)e^{2\tau\psi_{\varepsilon}(x,0)}dX
\end{equation*}
and $C$ depends only on $\lambda_0,M_0$.

\subsection{Estimate of the left hand side of the Carleman estimate, II}\label{subsEII}
In this section, we continue to estimate $\Xi(w)$ from above using \eqref{7.29final}. To this aim we apply Theorem \ref{pr22} to the function $w\eta_{g,\mu}$ with the weight function $\psi_{\eps,g'}=\varphi(y)-\eps x_{g'}\cdot x+\eps|x_{g'}|^2/2$. First we notice that if ${\rm supp}\, w\subset \mathfrak{U}:=Q_{1/2}\times [-r_0,r_0]$ and $\mu\geq 4$ then either $|x_g'|\leq 1$ and $|t_{g_n}|\leq 1$ or ${\rm supp}\, \eta_{g,\mu}\cap Q_{1/2}=\emptyset$ so that, in both the cases, we can apply Theorem \ref{pr22}.

By applying \eqref{5.16} and by adding up with respect to $g\in\mathbb{Z}^{n-1}$, we obtain that
\begin{equation}\label{8.1}
\sum_{g\in \mathbb{Z}^{n}}\Xi(w\eta_{g,\mu})\leq C \sum_{g\in \mathbb{Z}^{n}}(d^{(0)}_{g,\mu}+d^{(1)}_{g,\mu}+d^{(2)}_{g,\mu}+d^{(3)}_{g,\mu}),
\end{equation}
where
\begin{equation*}
\begin{aligned}
d^{(0)}_{g,\mu}=&[e^{\tau\psi_{\varepsilon,g'}}h_{0;g,\mu}(w)]^2_{0,\frac{3}{4},\mathbb{R}^{n}}
+[e^{\tau\psi_{\varepsilon,g'}}h_{1;g,\mu}(w)]^2_{0,\frac{1}{4},\mathbb{R}^{n}},\\
d^{(1)}_{g,\mu}=& \int_{\mathbb{R}^{n+1}}|\mathcal{L}_{\delta,g}(w\eta_{g,\mu})|^2e^{2\tau\psi_{\varepsilon,g'}(x,y)}dXdy,\\
d^{(2)}_{g,\mu}=&\tau^{3}\int_{\mathbb{R}^{n}}|e^{\tau\psi_{\varepsilon,g'}(x,0)}h_{0;g,\mu}(w)|^2dX+
[D_x(e^{\tau\psi_{\varepsilon,g'}(\cdot,0)}h_{0;g,\mu}(w))]^2_{\frac{1}{2},0,\mathbb{R}^{n}},\\
d^{(3)}_{g,\mu}=&\tau\int_{\mathbb{R}^{n}}|e^{\tau\psi_{\varepsilon,g'}(x,0)}h_{1;g,\mu}(w)|^2dX+
[e^{\tau\psi_{\varepsilon,g'}(\cdot,0)}h_{1;g,\mu}(w)]^2_{\frac{1}{2},0,\mathbb{R}^{n}},
\end{aligned}
\end{equation*}
where we set
\begin{equation}\label{theta0g}
h_{0;g,\mu}(w):=(w_+(x,t,0)-w_-(x,t,0))\eta_{g,\mu}(x,t)=\widetilde{h}_0(w)\eta_{g,\mu},
\end{equation}
\begin{equation}\label{theta1g}
h_{1;g,\mu}(w):=A^{\delta,g}_+(0)D_{x,y}(w_+\eta_{g,\mu})\cdot \nu-A^{\delta,g}_-(0)D_{x,y}(w_-\eta_{g,\mu})\cdot \nu.\end{equation}
In order to estimate from above the four terms of \eqref{8.1} we would like to point out that $\sum_{g\in \mathbb{Z}^{n}}d^{(0)}_{g,\mu}$ is the "new" term that arises in this parabolic context, whereas the other terms are basically the same of the corresponding terms of the elliptic case, \cite[(4.36)]{dcflvw} as soon as we notice that by \eqref{lipgen} and \eqref{7.7} we have

\begin{equation}\label{lipgen-delta}
|A^{\delta}_{\pm}(x,t,y)-A^{\delta,g}_{\pm}(y)|\leq M_0\delta, \quad \forall (x,t,y)\in Q_{2/\mu}(X_g)\times \mathbb{R}.
\end{equation}

We begin to estimate from above the comparatively new term $\sum_{g\in \mathbb{Z}^{n}}d^{(0)}_{g,\mu}$. First we estimate $\sum_{g\in \mathbb{Z}^{n}}[e^{\tau\psi_{\varepsilon,g'}(\cdot,0)}h_{0;g,\mu}(w)]^2_{0,\frac{3}{4},\mathbb{R}^{n}}$. By \eqref{5.6t}, \eqref{6.16t}, \eqref{7.191t} and \eqref{theta0g} we have

\begin{equation}\label{d4-1}
\begin{aligned}
&\sum_{g\in \mathbb{Z}^{n}}[e^{\tau\psi_{\varepsilon,g'}(\cdot,0)}h_{0;g,\mu}(w)]^2_{0,\frac{3}{4},\mathbb{R}^{n}}\\
\leq & C\sum_{g\in \mathbb{Z}^{n}}[e^{\tau\psi_{\varepsilon,g'}(\cdot,0)}\widetilde{h}_0(w)\eta_{g,\mu}]^2_{0,\frac{3}{4},Q_{2/\mu}(X_g)}+
C\sum_{g\in \mathbb{Z}^{n}}(\varepsilon\tau)^{3/2}\int_{Q_{2/\mu}(X_g)}|e^{\tau\psi_{\varepsilon,g'}(x,0)}\widetilde{h}_0(w)\eta_{g,\mu}(x,0)|^2dX\\
\leq & C[e^{\tau\psi_{\varepsilon}(\cdot,0)}\widetilde{h}_0(w)]^2_{0,\frac{3}{4},\mathbb{R}^n}+ C(\varepsilon\tau)^{3/2}\sum_{\pm}\int_{\mathbb{R}^n}|w_{\pm}(x,t,0)|^2e^{2\tau\psi_{\varepsilon}(x,0)}dX.
\end{aligned}
\end{equation}

Now we estimate $\sum_{g\in \mathbb{Z}^{n}}[e^{\tau\psi_{\varepsilon,g'}(\cdot,0)}h_{1;g,\mu}(w)]^2_{0,\frac{1}{4},\mathbb{R}^{n}}$.
By \eqref{theta1g} it is easy to write $h_{1;g,\mu}(w)$ as
\begin{equation*}
h_{1;g,\mu}(w)=\widetilde{h}_1(w)\eta_{g,\mu}+J^{(1)}_{g,\mu}+J^{(2)}_{g,\mu}+J^{(3)}_{g,\mu},
\end{equation*}
where
\begin{equation*}
\begin{aligned}
J^{(1)}_{g,\mu}=&w_+A^{\delta}_+(x,t,0)D_{x,y}\eta_{g,\mu}\cdot \nu -w_-A^{\delta}_-(x,t,0)D_{x,y}\eta_{g,\mu}\cdot \nu ,\\
J^{(2)}_{g,\mu}=&\eta_{g,\mu}(A^{\delta,g}_+(0)-A^{\delta}_+(x,t,0))D_{x,y}w_+\cdot \nu\\
&-\eta_{g,\mu}(A^{\delta,g}_-(0)-A^{\delta}_-(x,t,0))D_{x,y}w_-\cdot \nu,\\
J^{(3)}_{g,\mu}=&w_+(A^{\delta,g}_+(0)-A^{\delta}_+(x,t,0))D_{x,y}\eta_{g,\mu}\cdot \nu\\
&-w_-(A^{\delta,g}_-(0)-A^{\delta}_-(x,t,0))D_{x,y}\eta_{g,\mu}\cdot \nu.
\end{aligned}
\end{equation*}
By \eqref{cardAg}, \eqref{6.16t} and \eqref{7.191t} 
we have

\begin{equation}\label{seminorma-h1}
\begin{aligned}
\sum_{g\in \mathbb{Z}^{n}}[e^{\tau\psi_{\varepsilon,g'}(\cdot,0)}\widetilde{h}_1(w)\eta_{g,\mu}]^2_{0,\frac{1}{4},\mathbb{R}^{n}}&\leq
C[e^{\tau\psi_{\varepsilon}(\cdot,0)}\widetilde{h}_1(w)]^2_{0,\frac{1}{4},\mathbb{R}^{n}}\\
+& C(\varepsilon\tau)^{1/2}\sum_{\pm}\int_{\mathbb{R}^n}\left(|D_{x,y}w_{\pm}(x,t,0)|^2
+\tau^2|w_{\pm}(x,t,0)|^2\right)e^{2\tau\psi_{\varepsilon}(x,0)}dX.
\end{aligned}
\end{equation}
From \eqref{cardAg}, \eqref{6.7} and \eqref{5.6t} we have
\begin{equation}\label{J1}
\begin{aligned}
&\sum_{g\in \mathbb{Z}^{n}}[e^{\tau\psi_{\varepsilon,g'}(\cdot,0)}J^{(1)}_{g,\mu})]^2_{0,\frac{1}{4},\mathbb{R}^{n}}\\
\leq &C\sum_{g\in \mathbb{Z}^{n}}[e^{\tau\psi_{\varepsilon}(\cdot,0)}J^{(1)}_{g,\mu}]^2_{0,\frac{1}{4},Q_{2/\mu}(X_g)}
+C(\tau\varepsilon)^{1/2}\sum_{\pm}\int_{\mathbb{R}^{n}}|w_{\pm}(x,t,0)|^2e^{2\tau\psi_{\varepsilon}(x,0)}dX.\\
\end{aligned}
\end{equation}
Now, in order to estimate the first term at the righthand side of \eqref{J1}, after using the triangle inequality, we apply \eqref{6.7}, \eqref{6.26t} (we choose $\beta=1/2$, $\alpha=1/4$) and \eqref{6.20t} and we get
\begin{equation}\label{J1-1}
\begin{aligned}
&\sum_{g\in \mathbb{Z}^{n}}[e^{\tau\psi_{\varepsilon}(\cdot,0)}J^{(1)}_{g,\mu}]^2_{0,\frac{1}{4},Q_{2/\mu}(X_g)}
\leq C\sum_{\pm}\sum_{g\in \mathbb{Z}^{n}}\left([(e^{\tau\psi_{\varepsilon}}w_{\pm}D_x\eta_{g,\mu})(\cdot,\cdot,0)]^2_{0,\frac{1}{4},Q_{2/\mu}(X_g)}\right.\\
&+\delta^2(\varepsilon\tau)^{1/2}\int_{Q_{2/\mu}(X_g)}\left.|w_{\pm}(x,t,0)|^2e^{2\tau\psi_{\varepsilon}(x,0)}dX\right)\\
\leq & C\sum_{\pm}\left(\varepsilon\tau[(e^{\tau\psi_{\varepsilon}}w_{\pm})(\cdot,\cdot,0))]^2_{0,\frac{1}{4},\mathbb{R}^n}+
(\varepsilon\tau)^{3/2}\int_{\mathbb{R}^n}|w_{\pm}(x,t,0)|^2e^{2\tau\psi_{\varepsilon}(x,0)}dX\right).
\end{aligned}
\end{equation}

Hence by \eqref{J1} and \eqref{J1-1} we have

\begin{equation}\label{J1-2}
\begin{aligned}
&\sum_{g\in \mathbb{Z}^{n}}[e^{\tau\psi_{\varepsilon,g'}(\cdot,0)}J^{(1)}_{g,\mu}]^2_{0,\frac{1}{4},\mathbb{R}^{n}}\\
\leq &C\sum_{\pm}\left(\varepsilon\tau[(e^{\tau\psi_{\varepsilon}}w_{\pm})(\cdot,\cdot,0))]^2_{0,\frac{1}{4},\mathbb{R}^n}+
(\varepsilon\tau)^{3/2}\int_{\mathbb{R}^n}|w_{\pm}(x,t,0)|^2e^{2\tau\psi_{\varepsilon}(x,0)}dX\right).
\end{aligned}
\end{equation}

In order to estimate $\sum_{g\in \mathbb{Z}^{n}}[e^{\tau\psi_{\varepsilon,g'}(\cdot,0)}J^{(2)}_{g,\mu}]^2_{0,\frac{1}{4},\mathbb{R}^{n}}$ we use \eqref{disug-seminorme}, \eqref{cardAg}, \eqref{6.7}, \eqref{5.6t}, \eqref{6.26t}, \eqref{6.16t} and \eqref{lipgen-delta}   
we have

\begin{equation}\label{J2}
\begin{aligned}
&\sum_{g\in \mathbb{Z}^{n}}[e^{\tau\psi_{\varepsilon,g'}(\cdot,0)}J^{(2)}_{g,\mu}]^2_{0,\frac{1}{4},\mathbb{R}^{n}}\\
\leq & C(\tau\varepsilon)^{1/2}\sum_{\pm}\int_{\mathbb{R}^{n}}\left(|D_{x,y}w_{\pm}(x,t,0)|^2
+\tau^2|w_{\pm}(x,t,0)|^2\right)e^{2\tau\psi_{\varepsilon}(x,0)}dX\\
+&C\sum_{\pm}\sum_{g\in \mathbb{Z}^{n}}\left[\left(e^{\tau\psi_{\varepsilon}}\eta_{g,\mu}(A^{\delta,g}_{\pm}-A^{\delta}_{\pm})D_{x,y}w_{\pm})(\cdot,\cdot,0)\right)
\cdot \nu\right]^2_{0,\frac{1}{4},Q_{2/\mu}(X_g)}\\
\leq & C\left(\sqrt{\varepsilon\tau}+\frac{\delta^2}{\sqrt{\varepsilon\tau}}\right)\sum_{\pm}\int_{\mathbb{R}^n}\left(
|D_{x,y}w_{\pm}(x,t,0)|^2+\tau^{2}|w_{\pm}(x,t,0)|^2\right)e^{2\tau\psi_{\varepsilon}(x,0)}dX\\
+&C\frac{\delta^2\tau}{\varepsilon}\sum_{\pm}[(e^{\tau\psi_{\varepsilon}}w_{\pm})(\cdot,\cdot,0)]^2_{0,\frac{1}{4},\mathbb{R}^n}+
C\frac{\delta^2}{\varepsilon\tau}\sum_{\pm}[(e^{\tau\psi_{\varepsilon}}w_{\pm})(\cdot,\cdot,0)]^2_{0,\frac{3}{4},\mathbb{R}^n}\\
+&C\frac{\delta^2}{\varepsilon\tau}\sum_{\pm}[D_{x}(e^{\tau\psi_{\varepsilon}}w_{\pm})(\cdot,\cdot,0)]^2_{\frac{1}{2},0,\mathbb{R}^n}
+C\frac{\delta^2}{\varepsilon\tau}\sum_{\pm}[\partial_y(e^{\tau\psi_{\varepsilon}}w_{\pm})(\cdot,\cdot,0)]^2_{0,\frac{1}{4},\mathbb{R}^n}.
\end{aligned}
\end{equation}

Arguing as before it is simple to obtain

\begin{equation}\label{J3}
\begin{aligned}
&\sum_{g\in \mathbb{Z}^{n}}[e^{\tau\psi_{\varepsilon,g'}(\cdot,0)}J^{(3)}_{g,\mu}]^2_{0,\frac{1}{4},\mathbb{R}^{n}}\\
\leq & C\delta^2 \sum_{\pm}\left(\varepsilon\tau [(e^{\tau\psi_{\varepsilon}}w_{\pm})(\cdot,\cdot,0))]^2_{0,\frac{1}{4},\mathbb{R}^n}+
(\varepsilon\tau)^{3/2}\int_{\mathbb{R}^{n}}|w_{\pm}(x,t,0)|^2e^{2\tau\psi_{\varepsilon}(x,0)}dX\right).
\end{aligned}
\end{equation}

Finally , combining \eqref{d4-1}, \eqref{seminorma-h1}, \eqref{J1-2}, \eqref{J2} and \eqref{J3} we have

\begin{equation}\label{d0}
\begin{aligned}
\sum_{g\in \mathbb{Z}^{n}}d^{(0)}_{g,\mu}\leq C[e^{\tau\psi_{\varepsilon}(\cdot,0)}\widetilde{h}_0(w)]^2_{0,\frac{3}{4},\mathbb{R}^n}
+C[e^{\tau\psi_{\varepsilon}(\cdot,0)}\widetilde{h}_1(w)]^2_{0,\frac{1}{4},\mathbb{R}^{n}}+CR_0
\end{aligned}
\end{equation}
where
\begin{equation*}
\begin{aligned}
R_0=&\left(\sqrt{\varepsilon\tau}+\frac{\delta^2}{\sqrt{\varepsilon\tau}}\right)\sum_{\pm}\int_{\mathbb{R}^n}\left(
\tau^{2}|w_{\pm}(x,t,0)|^2+|D_{x,y}w_{\pm}(x,t,0)|^2\right)e^{2\tau\psi_{\varepsilon}(x,0)}dX\\
+&\left(\varepsilon\tau+
\frac{\delta^2\tau}{\varepsilon}\right)\sum_{\pm}[(e^{\tau\psi_{\varepsilon}}w_{\pm})(\cdot,\cdot,0)]^2_{0,\frac{1}{4},\mathbb{R}^n}+
\frac{\delta^2}{\varepsilon\tau}\sum_{\pm}[(e^{\tau\psi_{\varepsilon}}w_{\pm})(\cdot,\cdot,0)]^2_{0,\frac{3}{4},\mathbb{R}^n}\\
+&\frac{\delta^2}{\varepsilon\tau}\sum_{\pm}[D_{x}(e^{\tau\psi_{\varepsilon}}w_{\pm})(\cdot,\cdot,0)]^2_{\frac{1}{2},0,\mathbb{R}^n}
+\frac{\delta^2}{\varepsilon\tau}\sum_{\pm}[\partial_y(e^{\tau\psi_{\varepsilon}}w_{\pm})(\cdot,\cdot,0)]^2_{0,\frac{1}{4},\mathbb{R}^n}.
\end{aligned}
\end{equation*}

By \eqref{ellgen}, \eqref{6.7} and \eqref{lipgen-delta} we obtain that
\begin{equation*}
\begin{aligned}
&|\mathcal{L}_{\delta,g}(w\eta_{g,\mu})|\\
\leq &|\mathcal{L}_\delta(w\eta_{g,\mu})|+
|\mathcal{L}_\delta(w\eta_{g,\mu})-\mathcal{L}_{\delta,g}(w\eta_{g,\mu})|\\
\leq &\eta_{g,\mu}|\mathcal{L}_\delta(w)|+C\chi_{Q_{\frac{2}{\mu}}(X_g)}\left(\delta(\varepsilon\tau)^{-1/2}|D_{x,y}^2w|+
(\varepsilon\tau)^{1/2}|D_{x,y}w|\,
+\varepsilon\tau|w_\pm|\right),
\end{aligned}
\end{equation*}
which, together with \eqref{cardAg}, \eqref{aggiunta} and \eqref{7.17}, gives
\begin{equation}\label{8.3}
\sum_{g\in \mathbb{Z}^{n}}d^{(1)}_{g,\mu}\leq
C\int_{\mathbb{R}^{n+1}}|\mathcal{L}_\delta(w)|^2\,e^{2\tau\psi_{\varepsilon}(x,y)}dXdy+CR_1,
\end{equation}
where
\begin{equation*}
\begin{aligned}
R_1=\sum_{\pm}\int_{\mathbb{R}^{n+1}_{\pm}}\left(\frac{\delta^2}{\varepsilon\tau}|D_{x,y}^2w_{\pm}|^2+
\varepsilon\tau|D_{x,y}w_{\pm}|^2+\varepsilon^2\tau^2|w_{\pm}|^2\right)e^{2\tau\psi_{\varepsilon}(x,y)}dXdy.
\end{aligned}
\end{equation*}

By \eqref{cardAg}, \eqref{6.16}, \eqref{4.32e}, \eqref{theta0g} and \eqref{6.16} we have

\begin{equation}\label{8.9}
\sum_{g\in \mathbb{Z}^{n}}d^{(2)}_{g,\mu}\\
\leq C\left( [D_x(e^{\tau\psi_{\varepsilon}(\cdot,0)}\widetilde{h}_0(w))]^2_{\frac{1}{2},0,\mathbb{R}^{n}}+\tau^{3}\int_{\mathbb{R}^{n}}
|\widetilde{h}_0(w)|^2e^{2\tau\psi_{\varepsilon}(x,0)}dX+R_2\right),
\end{equation}
where
\begin{equation*}
\begin{aligned}
R_2=&\sum_\pm\left(\varepsilon^2\tau^2[e^{\tau\psi_{\varepsilon}(\cdot,0)}w_\pm(\cdot,\cdot,0)]^2_{\frac{1}{2},0,\mathbb{R}^{n}}+
(\varepsilon\tau)^{1/2}\int_{\mathbb{R}^{n}}|D_xw_\pm(x,t,0)|^2e^{2\tau\psi_{\varepsilon}(x,0)}dX\right.
\\
&\left.+(\varepsilon\tau)^{5/2}\int_{\mathbb{R}^{n}}|w_\pm(x,t,0)|^2e^{2\tau\psi_{\varepsilon}(x,0)}dX\right).
\end{aligned}\end{equation*}

Similarly, by \eqref{ellgen}, \eqref{cardAg}, \eqref{6.7}, \eqref{6.20}, \eqref{4.32e} and \eqref{lipgen-delta} we have

\begin{equation}\label{8.21}
\sum_{g\in \mathbb{Z}^{n}}d^{(3)}_{g,\mu}\leq C\left(\tau\int_{\mathbb{R}^{n}}|\widetilde{h}_1(w)|^2e^{2\tau\psi_{\varepsilon}(x,0)}dX+
[e^{\tau\psi_{\varepsilon}(\cdot,0)}\widetilde{h}_1(w)]^2_{\frac{1}{2},0,\mathbb{R}^{n}}+R_3\right),
\end{equation}
where
\begin{equation*}
\begin{aligned}
R_3=&\frac{\delta^{2}}{\varepsilon\tau}\sum_{\pm}[D_{x,y}(w_{\pm}e^{\tau\psi_{\varepsilon}})(\cdot,\cdot,0)]^2_{\frac{1}{2},0,\mathbb{R}^{n}}
+\tau(\varepsilon+\delta^{2}\varepsilon^{-1})\sum_{\pm}[e^{\tau\psi_{\varepsilon}(\cdot,0)}w_{\pm}(\cdot,0)]^2_{\frac{1}{2},0,\mathbb{R}^{n}}\\
&+\delta^2\varepsilon^{-1}\sum_{\pm}\int_{\mathbb{R}^{n}}|D_{x,y}w_{\pm}(x,t,0)|^2e^{2\tau\psi_{\varepsilon}(x,0)}dX\\
&+\tau^2(\varepsilon +\delta^{2}\varepsilon^{-1/2})\sum_{\pm}\int_{\mathbb{R}^{n}}|w_{\pm}(x,t,0)|^2e^{2\tau\psi_{\varepsilon}(x,0)}dX.
\end{aligned}
\end{equation*}

Now we choose $\varepsilon=\delta$, so that by \eqref{7.29final}, \eqref{8.1}, \eqref{d0}, \eqref{8.3}, \eqref{8.9} and \eqref{8.21} we have

\begin{equation}\label{8.22}
\begin{aligned}
\Xi(w)\leq &C\left(\int_{\mathbb{R}^{n+1}_{\pm}}|\mathcal{L}_\delta(w)|^2\,e^{2\tau\psi_{\delta,\pm}}dXdy
+[e^{\tau\psi_{\delta}(\cdot,0)}\widetilde{h}_1(w)]^2_{\frac{1}{2},\frac{1}{4},\mathbb{R}^{n}}
\right.\\
&+[D_x(e^{\tau\psi_{\delta}(\cdot,0)}\widetilde{h}_0(w))]^2_{\frac{1}{2},0,\mathbb{R}^{n}}+
[e^{\tau\psi_{\delta}(\cdot,0)}\widetilde{h}_0(w)]^2_{0,\frac{3}{4},\mathbb{R}^n}\\
&\left.+\tau\int_{\mathbb{R}^{n}}|\widetilde{h}_1(w)|^2e^{2\tau\psi_{\delta}(x,0)}dX+
\tau^{3}\int_{\mathbb{R}^{n}}|\widetilde{h}_0(w)|^2e^{2\tau\psi_{\delta}(x,0)}dX+R_4\right),
\end{aligned}
\end{equation}
where
\begin{equation*}
\begin{aligned}
R_4=&\delta\sum_{\pm}\int_{\mathbb{R}^{n+1}_{\pm}}\left(\frac{1}{\tau}|D_{x,y}^2w_{\pm}|^2+
\tau|D_{x,y}w_{\pm}|^2+\tau^2|w_{\pm}|^2\right)e^{2\tau\psi_{\delta}(x,y)}dXdy\\
+&(\delta\tau)^{1/2}\sum_{\pm}\int_{\mathbb{R}^{n}}\left((|D_{x,y}w_{\pm}(x,t,0)|^2
+\tau^2|w_{\pm}(x,t,0)|^2\right)e^{2\tau\psi_{\delta}(x,0)}dX\\
+&\delta\tau^2\sum_{\pm}[(e^{\tau\psi_{\delta}}w_{\pm})(\cdot,\cdot,0)]^2_{\frac{1}{2},\frac{1}{4},\mathbb{R}^n}+
\frac{\delta}{\tau}\sum_{\pm}[(e^{\tau\psi_{\delta}}w_{\pm})(\cdot,\cdot,0)]^2_{0,\frac{3}{4},\mathbb{R}^n}\\
+&\frac{\delta}{\tau}\sum_{\pm}[D_{x}(e^{\tau\psi_{\delta}}w_{\pm})(\cdot,\cdot,0)]^2_{\frac{1}{2},0,\mathbb{R}^n}
+\frac{\delta}{\tau}\sum_{\pm}[\partial_y(e^{\tau\psi_{\delta}}w_{\pm})(\cdot,\cdot,0)]^2_{0,\frac{1}{4},\mathbb{R}^n}
\end{aligned}
\end{equation*}
and $C$ depends on $\lambda_0, M_0, n$.
Now it is easy to note that there exists a sufficiently small $\delta_0$  and a sufficiently large $\tau_0$, both depending on $\lambda_0,M_0,n$  such that if $\delta\leq\delta_0$ and $\tau\geq\tau_0$, then $R_4$ on the right hand side of \eqref{8.22} can be absorbed by $\Xi(w)$ (defined in \eqref{l.s.}). In other words, we have proved that

\begin{equation}\label{Final}
\begin{aligned}
&\sum_{\pm}\sum_{k=0}^2\tau^{3-2k}\int_{\mathbb{R}^{n+1}_{\pm}}|D_{x,y}^k{w}_{\pm}|^2e^{2\tau\psi_{\delta}(x,y)}dXdy+
\tau^{-1}\sum_{\pm}\int_{\mathbb{R}^{n+1}_{\pm}}|\partial_t{w}_{\pm}|^2e^{2\tau\psi_{\delta}(x,y)}dXdy\\
&+\sum_{\pm}\sum_{k=0}^1\tau^{3-2k}\int_{\mathbb{R}^{n}}|D_{x,y}^k{w}_{\pm}(x,t,0)|^2e^{2\tau\psi_{\delta}(x,0)}dX
+\tau^2\sum_{\pm}[(e^{\tau\psi_\delta}w_{\pm})(\cdot,\cdot,0)]^2_{\frac{1}{2},\frac{1}{4},\mathbb{R}^{n}}\\
&+\sum_{\pm}[(e^{\tau\psi_\delta}w_{\pm})(\cdot,\cdot,0)]^2_{0,\frac{3}{4},\mathbb{R}^{n}}
+\sum_{\pm}[\partial_y(e^{\tau\psi_{\delta,\pm}}w_{\pm})(\cdot,\cdot,0)]^2_{\frac{1}{2},\frac{1}{4},\mathbb{R}^{n}}
+\sum_{\pm}[D_x(e^{\tau\psi_{\delta,\pm}}w_{\pm})(\cdot,\cdot,0)]^2_{\frac{1}{2},0,\mathbb{R}^{n}}\\
\leq &C\left(\int_{\mathbb{R}^{n+1}}|\mathcal{L}_{\delta}(w)|^2\,e^{2\tau\psi_{\delta}(x,y)}dXdy+
[e^{\tau\psi_\delta(\cdot,0)}\widetilde{h}_1(w)]^2_{\frac{1}{2},\frac{1}{4},\mathbb{R}^{n}}
+[e^{\tau\psi_\delta (\cdot,0)}\widetilde{h}_0(w)]^2_{0,\frac{3}{4},\mathbb{R}^{n}}\right.\\
&\left.+[D_x(e^{\tau\psi_\delta(\cdot,0)}\widetilde{h}_0(w))]^2_{\frac{1}{2},0,\mathbb{R}^{n}}+
\tau^3\int_{\mathbb{R}^{n}}|\widetilde{h}_0(w)|^2e^{2\tau\psi_\delta(x,0)}dX
+\tau\int_{\mathbb{R}^{n}}|\widetilde{h}_1(w)|^2e^{2\tau\psi_\delta(x,0)}dX\right).
\end{aligned}
\end{equation}
Now, applying \eqref{Final} to the function $w(x,t,y)=u(\delta x,\delta^2t,\delta y)$, by a standard change of variable  we obtain \eqref{8.24} with $b=0$.

\section{Step 3 - Carleman estimate with weight depending on $t$}\label{step3}
In the previous step we have proved that
\begin{equation}\label{(1)-1}
\begin{aligned}
\sum_{\pm}&\left(\sum_{k=0}^2\tau^{3-2k}\int_{\mathbb{R}^{n+1}_{\pm}}|D_{x,y}^k{u}_{\pm}|^2e^{2\tau\phi_{\delta_0}}dXdy+
\tau^{-1}
\int_{\mathbb{R}^{n+1}_{\pm}}|\partial_t{u}_{\pm}|^2e^{2\tau\phi_{\delta_0}}dXdy\right)
+Y^L(u;\phi_{\delta_0})\\
&\leq C\left(\int_{\mathbb{R}^{n+1}}|\mathcal{L}(u)|^2\,e^{2\tau\phi_{\delta_0}}dXdy+
Y^R(u;\phi_{\delta_0})\right).
\end{aligned}
\end{equation}
Let us now define
\begin{equation}\label{phi1}
	\varphi_1(t)=-\frac{b}{2}\,t^2
\end{equation}
and insert in \eqref{(1)-1} the function $u=ve^{\tau\varphi_1(t)}$. It is easy to see that
\begin{equation}\label{(2)-1}
\begin{aligned}
	\int_{\mathbb{R}^{n+1}}|\mathcal{L}(ve^{\tau\varphi_1}|^2\,e^{2\tau\phi_{\delta_0}}dXdy\leq&
	C \tau^2\sum_\pm \int_{\mathbb{R}^{n+1}_{\pm}}b^2|v_{\pm}|^2e^{2\tau(\phi_{\delta_0}+\varphi_1)}dXdy
	\\&+2\int_{\mathbb{R}^{n+1}}|\mathcal{L}(v)|^2\,e^{2\tau(\phi_{\delta_0}+\varphi_1)}dXdy,
\end{aligned}\end{equation}
\begin{equation}\label{(1a)-2}
	 \int_{\mathbb{R}^{n+1}_{\pm}}|D_{x,y}^k({v}_{\pm}e^{\tau\varphi_1})|^2e^{2\tau\phi_{\delta_0}}=
	 \int_{\mathbb{R}^{n+1}_{\pm}}|D_{x,y}^k{v}_{\pm}|^2e^{2\tau(\phi_{\delta_0}+\varphi_1)},
\end{equation}
and
\begin{equation}\label{(1b)-2}
	 \int_{\mathbb{R}^{n+1}_{\pm}}|\der_t({v}_{\pm}e^{\tau\varphi_1})|^2e^{2\tau\phi_{\delta_0}}\geq
	\int_{\mathbb{R}^{n+1}_{\pm}}\left[\frac{1}{2}|\der_t v_\pm|^2-C\tau^2b^2|v_\pm|^2\right]e^{2\tau(\phi_{\delta_0}+\varphi_1)}
\end{equation}
Moreover, by \eqref{BL} and \eqref{BR}, we have
\begin{equation*}
	Y^L(v e^{\tau\varphi_1};\phi_{\delta_0})=Y^L(v ;\phi_{\delta_0}+\varphi_1)\mbox{ and }	Y^R(v e^{\tau\varphi_1};\phi_{\delta_0})=Y^R(v ;\phi_{\delta_0}+\varphi_1),
\end{equation*}
hence, for large enough $\tau$ the extra terms appearing in \eqref{(2)-1} and \eqref{(1b)-2} can be absorbed and  Theorem \ref{thm8.2} is finally proved.
\section{Three-region inequality}\label{treregioni}

\begin{theo}\label{thtrereg}
Let $A(x,t,y)=\sum_\pm A_\pm(x,t,y)$ where $A_\pm$ satisfy assumptions \eqref{ellgen} and \eqref{lipgen}  and let $V$ be bounded function and $W$ a bounded vector valued function such that
\begin{equation*}
	\|V\|_{L^\infty(\RR^{n+1},\RR)}+\|W\|_{L^\infty(\RR^{n+1},\RR^n)}\leq \lambda_0^{-1}.
\end{equation*}
Let $\tau_0$, $\alpha_+$, $\alpha_-$, $\beta$, $b$ and $\delta_0$ given by Theorem \ref{thm8.2}, and let
\begin{equation}\label{z}
	z(x,t,y)=\frac{\alpha_-y}{\delta_0}+\frac{\beta y^2}{2\delta_0^2}-\frac{|x|^2}{2\delta_0}-\frac{bt^2}{2}.
\end{equation}
There exist $C$ and $R$ depending only on $\lambda_0$, $M_0$ and $n$, such that, if $u$ is a weak solution to the equation
\begin{equation}\label{eqreg}
	-\der_tu+\dive_{x,y}\left(A(x,t,y)D_{x,y} u\right)+W\cdot D_{x,y} u+Vu=0\,\mbox{ in }\,Q_{\delta_0/2}\times(-\delta_0 r_0,\delta_0 r_0),
\end{equation}
then, for $0<R_1, R_2\leq R$,
\begin{equation}\label{trereg}
	\int_{U_2}|u|^2dXdy\leq C \left(\int_{U_1}|u|^2dXdy\right)^{\frac{R_2}{2R_1+3R_2}}\left(\int_{U_3}|u|^2dXdy\right)^{\frac{2R_1+2R_2}{2R_1+3R_2}}
\end{equation}
where
\begin{equation*}
	\begin{aligned}
	U_1&=\left\{(x,t,y)\in\RR^{n+1}\,:\,z(x,t,y)\geq -4R_2,\, \frac{R_1}{8a}<y<\frac{R_1}{a}\right\}\\
		U_2&=\left\{(x,t,y)\in\RR^{n+1}\,:\,z(x,t,y)\geq -R_2,\, y<\frac{R_1}{4a}\right\}\\
			U_3&=\left\{(x,t,y)\in\RR^{n+1}\,:\,z(x,t,y)\geq -4R_2,\, y<\frac{R_1}{a}\right\}
\end{aligned}
\end{equation*}
for
\[a=\alpha_+/\delta_0.\]
\end{theo}
Notice that function $z$ coincides on $\RR^{n-1}\times\RR\times\RR_{-}$ with the weight function $\Phi$ appearing in Theorem \ref{thm8.2}.
\begin{proof}
We prove the Theorem with the additional assumption $u_{\pm}\in H^{2,1}(Q^{\pm}_{\delta_0/2}\times(-\delta_0 r_0,\delta_0 r_0))$, where $Q^{\pm}_{\delta_0/2}\times(-\delta_0 r_0,\delta_0 r_0))=Q_{\delta_0/2}\times(-\delta_0 r_0,\delta_0 r_0))\cap (\RR^{n}\times\RR_{\pm})$, in appendix we show that indeed the weak solution $u$ satisfy such an additional assumption. After performing a standard density argument we can apply Theorem \ref{thm8.2} to the function $u\theta$ where $\theta$ is a cut off function such that ${\rm supp}\, \theta\subset Q_{\delta_0/2}\times(-\delta_0 r_0,\delta_0 r_0)$.
We can assume, that $\alpha_+>\alpha_-$. By following the calculations in the proof of the three-region inequality in the elliptic case (\cite[Theorem 3.1]{flvw}), let us choose
\begin{equation}\label{erre}
	R=\frac{\alpha_-}{16}\min\left\{r_0, \frac{13\alpha_-}{8\beta},\frac{2\delta_0}{19\alpha_-+8\beta}\right\}.
\end{equation}
Given $0<R_1<R_2\leq R$, let $\theta_1(s)\in C^\infty(\RR)$ such that $0\leq\theta_1(s)\leq 1$ and
\begin{equation*}
	\theta_1(s)=\left\{\begin{array}{rl}1,& s>-2R_2\\0,& s\leq-3R_2,\end{array}\right.
\end{equation*}
and let  $\theta_2(y)\in C^\infty(\RR)$ such that $0\leq\theta_2(y)\leq 1$ and
\begin{equation*}
	\theta_2(y)=\left\{\begin{array}{rl}0,& y\geq\frac{R_1}{2a},\\1,& y<\frac{R_1}{4a}.\end{array}\right.
\end{equation*}
Define
\begin{equation*}
	\theta(x,t,y)=\theta_1(z(x,t,y))\theta_2(y)
\end{equation*}
for $z$ as in \eqref{z}.

The support of function $\theta$ is contained in the set
\[\left\{\,z(x,t,y)<-3R_2,\,\,\, y<\frac{R_1}{2a}\right\};\]
notice that, for $R$ given by \eqref{erre}, the support of  $\theta$ is contained in $Q_{\delta_0/2}\times[-\delta_0r_0,\delta_0r_0]$ (see details in \cite{flvw}).

We can, hence, apply estimate \eqref{8.24} (see also remark \ref{bassi}) to the function $\theta(x,t,y)u(x,t,y)$.
Let us calculate
\begin{equation*}
\begin{aligned}
	\tilde{\mathcal{L}}(\theta u)=&\theta\tilde{\mathcal{L}}(u)-(\der_t\theta) u+2\sum_\pm H_\pm A_\pm D_{x,y}\theta\cdot D_{x,y}u_\pm\\
	&+\sum_\pm u_\pm\left(\dive\left(A_\pm D_{x,y}\theta\right)\right)+u\,W\cdot D_{x,y}\theta.
	\end{aligned}
\end{equation*}
Since $\tilde{\mathcal{L}}(u)=0$ and since the derivatives of $\theta$ are nonzero only on the set
\[\tilde{U}=\left\{ -3R_2\leq z\leq -2R_2,\,\,y\leq \frac{R_1}{2a}\right\},\]
we have
\begin{equation}\label{1-5t}
	\left|\tilde{\mathcal{L}}(\theta u)\right|\leq C\left(\left|D_{x,y}u\right|+|u|\right)\chi_{\tilde{U}},
\end{equation}
where $\left|D_{x,y}u\right|^2=\sum_\pm H_\pm \left|D_{x,y}u_\pm\right|^2$.

Since $h_0(u)=h_1(u)=0$, we also have
\begin{equation}\label{2-5t}
h_0(\theta u)=\theta h_0(u)=0,
\end{equation}
 \begin{equation}\label{3-5t}
h_1(\theta u)=\theta h_1(u)+u_+(x,t,0)\eta(x,t)=u_+(x,t,0)\eta(x,t)
\end{equation}
for
\begin{equation*}
\eta(x,t)=\left[A_+(x,t,0)-A_-(x,t,0)\right]D_{x,y}\theta(x,t,0)\cdot \nu.
\end{equation*}
By explicit calculations it is easy to see that
\[D_{x,y}\theta(x,t,0)=\left(-\frac{x}{\delta_0},\frac{\alpha_-}{\delta_0}\right){\theta_1}^{\!\prime}\!\left(-\frac{|x|^2}{2\delta_0}-\frac{b}{2}t^2\right) ,\]
hence, $\eta(x,t)$ is different from zero only in
\[\mathcal{G}=\left\{2R_2<\frac{|x|^2}{2\delta_0}+\frac{b}{2}t^2<3R_2\right\}.\]

By \eqref{8.24}, \eqref{1-5t}, \eqref{2-5t} and \eqref{3-5t}, we have
\begin{equation}\label{3-7t}
	\begin{aligned}\sum_\pm &\sum_{k=0}^1\tau^{3-2k}\int_{\RR^{n+1}_\pm}\left|D^k_{x,y}(\theta u_\pm)\right|^2e^{2\tau\Phi}dXdy\leq C\int_{\tilde{U}}
	\left|D_{x,y}u\right|^2e^{2\tau\Phi}dXdy\\
&	+C\int_{\tilde{U}} 	 \left|u\right|^2e^{2\tau\Phi}dXdy+C\left[e^{\tau\Phi(x,t,0)}\eta(x,t)u_+(x,t,0)\right]^2_{\frac{1}{2},\frac{1}{4},\RR^n}\\
&+C\tau\int_{\mathcal{G}}\left|e^{\tau\Phi(x,t,0)}\eta(x,t)u_+(x,t,0) \right|^2dX.
\end{aligned}\end{equation}

Notice that $\tilde{U}
\subset \tilde{U}_1\cup\tilde{U}_3$, where
\[\tilde{U}_1=\left\{-3R_2\leq z,\,\,\frac{R_1}{4a}\leq y\leq \frac{R_1}{2a}\right\},\]
and
\[\tilde{U}_3=\left\{-3R_2\leq z\leq -2R_2,\,\,y<\frac{R_1}{4a}\right\}.\]
The weight function $\Phi$ can be written as
\[\Phi(x,t,y)=\left\{\begin{array}{rcl}\frac{(\alpha_+-\alpha_-)y}{\delta_0}+z(x,t,y)&\mbox{for}&y\geq 0\\
z(x,t,y)&\mbox{for}&y<0\end{array},\right.\]
hence, by \eqref{erre},
\begin{equation}\label{11t}
	\Phi(x,t,y)\leq\left\{\begin{array}{lcl} \frac{R_1}{2}+\frac{\beta}{2}\frac{R_1^2}{4\alpha_+^2}\leq R_1&\mbox{in}&\tilde{U}_1
	\\ \frac{R_1}{4a}\left(\frac{\alpha_+-\alpha_-}{\delta_0}\right)-2R_2\leq \frac{R_1}{4}-2R_2&\mbox{in}&\tilde{U}_3\end{array}.\right.
\end{equation}
By \eqref{11t},
\begin{equation}\label{12t}
	\int_{\tilde{U}} 	\left|u\right|^2e^{2\tau\Phi}dXdy\leq C\left(e^{2\tau R_1}\int_{\tilde{U}_1} 	\left|u\right|^2dXdy+ e^{\tau(\frac{R_1}{2}-4R_2)}
	\int_{\tilde{U}_3} 	\left|u\right|^2dXdy\right).
\end{equation}
A similar estimate can be written for $D_{x,y}u$ instead of $u$. We can, then, use a parabolic Caccioppoli-type inequality and, observing that $\tilde{U}_1\subset\subset U_1$ and $\tilde{U}_3\subset\subset U_3$, we can write
\begin{equation}\label{1-12t}
	\int_{\tilde{U}} 	\left|D_{x,y}u\right|^2e^{2\tau\Phi}dXdy\leq  C\left(e^{2\tau R_1}\int_{U_1} 	\left|u\right|^2dXdy+ e^{\tau(\frac{R_1}{2}-4R_2)}
	\int_{U_3} 	\left|u\right|^2dXdy\right).
\end{equation}
In the set $\mathcal{G}$, that contains the support of $\eta$, we have
\[\Phi(x,t,0)=-\left(\frac{|x|^2}{2\delta_0}+\frac{b}{2}t^2\right)\leq-2R_2,\]
hence
\begin{equation}\label{1-9t}
	\int_{\mathcal{G}}\left|e^{\tau\Phi(x,t,0)}\eta(x,t)u_+(x,t,0) \right|^2dX\leq
Ce^{-4\tau R_2}	\int_{\mathcal{G}}\left|u_+(x,t,0) \right|^2dX
\end{equation}
and, by Proposition \ref{pr6.4} (applying \eqref{6.26} and also \eqref{6.26t} for $\alpha=1/4$ and $\beta=1/2$), we have
\begin{equation}\label{1-8t}
	\left[e^{\tau\Phi(x,t,0)}\eta(x,t)u_+(x,t,0)\right]^2_{\frac{1}{2},\frac{1}{4},\RR^n}
\leq C e^{-4\tau R_2}\left(\left[u_+(x,t,0)\right]^2_{\frac{1}{2},\frac{1}{4},\mathcal{G}}+\tau^2\int_{\mathcal{G}}\left|u_+(x,t,0)\right|^2dX\right).
\end{equation}
By putting together \eqref{3-7t}, \eqref{12t}, \eqref{1-12t}, \eqref{1-9t} and \eqref{1-8t}, we get
\begin{equation}\label{10t}
	\begin{aligned}\sum_\pm& \sum_{k=0}^1\tau^{3-2k}\int_{\RR^{n+1}_\pm}\left|D^k_{x,y}(\theta u_\pm)\right|^2e^{2\tau\Phi}dXdy\leq
	C\left(e^{2\tau R_1}\int_{U_1} 	\left|u\right|^2dXdy\right.\\
	&\left.+ e^{\tau(\frac{R_1}{2}-4R_2)}
	\int_{U_3} 	\left|u\right|^2dXdy\right) +C e^{-4\tau R_2}\left(\left[u_+(x,t,0)\right]^2_{\frac{1}{2},\frac{1}{4},\mathcal{G}}+\tau^2\int_{\mathcal{G}}\left|u_+(x,t,0)\right|^2dX\right).
	\end{aligned}
\end{equation}
On the other hand, since $\mathcal{G}\subset\subset U_3$, by using traces and regularity estimates (see \cite[Theorem 5.1]{LSU}), we have
\begin{equation}\label{2-15t}
		\left(\left[u_+(x,t,0)\right]^2_{\frac{1}{2},\frac{1}{4},\mathcal{G}}+\tau^2\int_{\mathcal{G}}\left|u_+(x,t,0)\right|^2dX\right)\leq
	\tau^2\int_{U_3}\left|u\right|^2dXdy.
\end{equation}
Let us now consider the left-hand side of \eqref{10t}. On the set
\[U_2=\left\{z\geq -R_2,\, y<\frac{R_1}{4a}\right\}\]
we have $\theta\equiv 1$ and $\Phi\geq -R_2$, hence
\begin{equation}\label{13t}
\sum_\pm \sum_{k=0}^1\tau^{3-2k}\int_{\RR^{n+1}_\pm}\left|D^k_{x,y}(\theta u_\pm)\right|^2e^{2\tau\Phi}dXdy
\geq \tau^3 e^{-2\tau R_2}\int_{U_2}\left|u\right|^2dXdy.
\end{equation}
By \eqref{10t}, \eqref{2-15t} and \eqref{13t} we get, for $\tau\geq\tau_0$
\begin{equation}\label{16t}
	\int_{U_2}\left|u\right|^2dXdy\leq C\left(e^{2\tau (R_1+R_2)}\int_{U_1}\left|u\right|^2dXdy+e^{-\tau R_2}\int_{U_3}\left|u\right|^2dXdy\right).
\end{equation}
Now we want to choose $\tau$ in order to get \eqref{trereg}.
Let us denote by $\varepsilon^2=\int_{U_1}\left|u\right|^2dXdy$, $E^2=\int_{U_3}\left|u\right|^2dXdy$
so that \eqref{16t} becomes
\begin{equation}\label{16tbis}
	\int_{U_2}\left|u\right|^2dXdy\leq C\left(e^{2\tau (R_1+R_2)}\varepsilon^2+e^{-\tau R_2}E^2\right).
\end{equation}
Let $\tau_1$ be such that
\[e^{2\tau_1 (R_1+R_2)}\varepsilon^2=e^{-\tau_1 R_2}E^2\]
that is
\[\tau_1=\frac{1}{2R_1+3R_2}\log\left(\frac{E^2}{\varepsilon^2}\right)\]
If $\tau_1>\tau_0$, then w can choose $\tau=\tau_1$ and \eqref{16tbis} gives
\[\int_{U_2}\left|u\right|^2dXdy\leq 2C (\varepsilon^2)^{\frac{R_2}{2R_1+3R_2}}(E^2)^{\frac{2R_1+2R_2}{2R_1+3R_2}}\]
that is \eqref{trereg}.

On the other hand, if  $\tau_1\leq\tau_0$, it means that
\[\frac{1}{2R_1+3R_2}\log\left(\frac{E^2}{\varepsilon^2}\right)\leq\tau_0,\]
that is
\[E^2\leq e^{\frac{\tau_0}{R_1+2R_2}}\varepsilon^2,\]
and, hence, we can write
\[\int_{U_2}\left|u\right|^2dXdy\leq E^2=(E^2)^{\frac{2R_1+2R_2}{2R_1+3R_2}}(E^2)^{1-\frac{2R_1+2R_2}{2R_1+3R_2}}
\leq C (\varepsilon^2)^{\frac{R_2}{2R_1+3R_2}}(E^2)^{\frac{2R_1+2R_2}{2R_1+3R_2}}\]
that is, again, \eqref{trereg}.

\end{proof}
\section{appendix}\label{appendice}
In what follows we assume that $A$, $V$, $W$ satisfy the same assumptions of Theorem \ref{thtrereg}. Moreover, for the sake of brevity, we denote $\Omega_{\delta_0}=Q_{\delta_0/2}\times(-\delta_0 r_0,\delta_0 r_0)$, $I_{\delta_0}=[-\delta_0,\delta_0]$ and $\Omega_{\delta_0/2}^{\pm}=\Omega_{\delta_0/2}\cap (\RR^{n}\times\RR_{\pm})$. We recall that
$H^{1,\frac{1}{2}}(\Omega_{\delta_0})$ the space of functions $f\in L^2(\Omega_{\delta_0})$ satisfying $\|f\|_{H^{1,\frac{1}{2}}(\Omega_{\delta_0})}<\infty$   where

$$\|f\|_{H^{1,\frac{1}{2}}(\Omega_{\delta_0})}=\left(\|f\|^2_{L^{2}(\Omega_{\delta_0})}+\|D_{x,y}f\|^2_{L^{2}(\Omega_{\delta_0})}
+[f]^2_{0,\frac{1}{2},\Omega_{\delta_0}}\right)^{1/2},$$
and
$$[f]_{0,\frac{1}{2},\Omega_{\delta_0}}=\left[\int_{\Omega_{\delta_0}}\int_{I_{\delta_0}}\int_{I_{\delta_0}}\frac{|f(x,t_1,y)-f(x,t_2,y)|^2}
{|t_1-t_2|^{2}}dt_1dt_2dxdy\right]^{1/2}.$$

\begin{pr}\label{regolarit}
Let $u\in H^{1,\frac{1}{2}}(\Omega_{\delta_0})$ be a weak solution to the equation \eqref{eqreg} then $u_{\pm}\in H^{2,1}(\Omega_{\delta_0/2}^{\pm})$.

\end{pr}

 We give a sketch of the proof. We limit ourselves to the case $W=0$, $V=0$. Let $\{A^{(m)}_{\pm}\}_{m\in\mathbb{N}}$ be a sequence of smooth symmetric matrix-valued functions that approximate $A_{\pm}$ in $L^{\infty}(\Omega_{\delta_0})$, and that satisfies
\begin{equation*}
\widetilde{\lambda}_0|z|^2\leq A^{(m)}_{\pm}(x,t,y)z\cdot z\leq \widetilde{\lambda}^{-1}_0|z|^2,\, \quad \forall (x,t,y)\in \Omega_{\delta_0},\,\forall\, z\in \mathbb{R}^n,\,
\end{equation*}
and
\begin{equation*}
|A^{(m)}_{\pm}(x',t',y')-A^{(m)}_{\pm}(x,t,y)|\leq \widetilde{M}_0\left(|x'-x|^2+|y'-y|^2+|t'-t|\right)^{1/2},
\end{equation*}
where $\widetilde{\lambda}_0$ and $\widetilde{M}_0$ depend only on $\lambda_0$ and $M_0$. For every $m\in \mathbb{N}$, let $u^{(m)}$ be a weak solution to

\begin{equation*}
\begin{cases}
\begin{array}{l}
-\der_tu^{(m)}+\dive_{x,y}\left(A^{(m)}(x,t,y)D_{x,y} u^{(m)}\right)=0,\quad \mbox{in } \Omega_{3\delta_0/4},\\
u^{(m)}=u, \quad \mbox{on } \Gamma,
\end{array}
\end{cases}
\end{equation*}
where $\Gamma$ is the parabolic boundary of $\Omega_{3\delta_0/4}$. Now, see \cite[Chapter III, Section 13]{LSU}, we have $u^{(m)}_{\pm}\in H^{2,1}(\Omega^{\pm}_{\delta_1})$ for every $0<\delta_1<\delta_0$ and

\begin{equation}\label{reg2}
\begin{aligned}
&\|u^{(m)}-u\|:=\max_{t\in I_{3\delta_0/4}}
\|u^{(m)}-u(\cdot,t,\cdot)\|_{L^2(Q_{3\delta_0/4})}+\|D_{x,y}(u^{(m)}-u)\|_{L^2(\Omega_{3\delta_0/4})}\\
&\leq C\|D_{x,y}u\|_{L^2(\Omega_{3\delta_0/4})}\|A-A^{(m)}\|_{L^{\infty}(\Omega_{3\delta_0/4})},
\end{aligned}
\end{equation}
where $C$ doesn't depend by $m$ (it depends only by $\widetilde{\lambda}_0$). By  \eqref{reg2}
we have

\begin{equation}\label{reg-conv}
u^{(m)}\rightarrow u \mbox{, as } m\rightarrow\infty \quad \mbox{, in } L^2(\Omega_{\delta_0/2}).
\end{equation}
Now we apply the Carleman estimate \eqref{8.24}, for a fixed $\tau$, to the operator $\mathcal{L}$. Since $u^{(m)}_{\pm}\in H^{2,1}(\Omega^{\pm}_{3\delta_0/4})$, for every $m\in \mathbb{N}$, after performing a standard density argument we can apply \eqref{8.24} to $u^{(m)}\zeta$ where $\zeta\in C^{\infty}_0 (\Omega_{3\delta_0/4})$ is a cut off function such that $0\leq \zeta\leq 1$ and $\zeta=1$ in $\Omega_{\delta_0/2}$. It is easy to check that that

\begin{equation}\label{reg3}
\mathcal{L}(u^{(m)}\zeta)\leq C\left(|u^{(m)}|+|D_{x,y}u^{(m)}|\right), \quad \mbox{ in } \Omega_{3\delta_0/4} \mbox{, }\quad \forall m\in \mathbb{N},
\end{equation}
where $C$ depends only on $\widetilde{\lambda}_0$ and $\widetilde{M}_0$.
Moreover by trace inequality and \eqref{reg2} we have, for every $m\in \mathbb{N}$

\begin{equation}\label{reg4}
\begin{aligned}
& Y^R(u^{(m)}\zeta;\Phi)\leq C \left(\left[u_+(x,t,0)\right]^2_{\frac{1}{2},\frac{1}{4},\Omega_{5\delta_0/8}}+
\int_{\Omega_{5\delta_0/8}}\left|u_+(x,t,0)\right|^2dX\right)\\
&\leq C'\|D_{x,y}u\|^2_{L^2(\Omega_{3\delta_0/4})}\left(1+\|A-A^{(m)}\|^2_{L^{\infty}(\Omega_{3\delta_0/4})}\right)\leq \widetilde{C},
\end{aligned}
\end{equation}
where $C$, $C'$ and $\widetilde{C}$ are independent of $m$.

Now by \eqref{8.24}, \eqref{reg3}, \eqref{reg4} we get

\begin{equation*}
\sum_{\pm}\left(\sum_{k=0}^2\int_{\Omega^{\pm}_{\delta_0/2}}|D_{x,y}^k{u}^{(m)}_{\pm}|^2dXdy+
\int_{\Omega^{\pm}_{\delta_0/2}}|\partial_t{u}^{(m)}_{\pm}|^2dXdy\right)\leq C,
\end{equation*}
where $C$ is independent of $m$. Therefore there exists a subsequence $\{u^{(j_m)}\}_{m\in\mathbb{N}}$ such that $\{D_{x,y}^ku^{(j_m)}\}_{m\in\mathbb{N}}$ and $\{\partial_tu^{(j_m)}\}_{m\in\mathbb{N}}$ weakly converge in $L^2(\Omega_{\delta_0/2})$ so that, taking into account \eqref{reg-conv}, we conclude the proof.
\qed

\section*{Acknowledgement}
The authors were partially supported by Gruppo Nazionale per l'Analisi Matematica, la Probabilit\`{a} e le loro Applicazioni (GNAMPA) of the Istituto Nazionale di Alta Matematica (INdAM).
EF was partially supported by FIR Project Geometrical and Qualitative Aspects of PDE's.


\begin{thebibliography}{999}
\bibitem[AM]{alma} G. Alessandrini, R. Magnanini, \emph{Elliptic equations in divergence form, geometrical critical points of solutions and Stekloff eigenfunctions}, SIAM J. Math. Anal., \textbf{25} (1994), 1259--1268.


\bibitem[ARRV]{ARRV} G. Alessandrini, L. Rondi, E. Rosset, S. Vessella,
The stability for the Cauchy problem for elliptic equations.
Inverse Problems 25 (2009), 123004.

\bibitem[AV1]{AlVe} G. Alessandrini, S. Vessella, Remark on the strong
unique continuation property for parabolic operators, Proc. of AMS, 132, (2004), 499--501.

\bibitem[AV2]{AV} G. Alessandrini, S. Vessella,
 \emph{Lipschitz stability for the inverse conductivity problem},
Adv. in Appl. Math., 35 (2005), 207--241.

\bibitem[A]{Am} B.K. Amonov, \emph{The stability of solution of Cauchy
problem for a second order equation of parabolic type with data on a
time-like manifold}, Funkcional. Anal. i Prilo\v{z}en 6, (3), (1972), 1--9.

\bibitem[AS]{AmSh} B.K. Amonov, S.P. Shishatskii, \emph{An a priori estimate
of the solution of the Cauchy problem with data on a time-like surface for a
second order parabolic equation, and related uniqueness theorems}, Dokl.
Akad Nauk\textit{\ SSSR} 206 (1972),  11--12.

\bibitem[AKS]{AKS} N. Aronszajn, A. Krzywicki and J. Szarski, \emph{ A unique continuation theorem for exterior differential forms on riemannian manifolds}, Ark. for Matematik, 4, (34), (1962), 417--453.

\bibitem[BLR1]{Be-LR-1} M. Bellassoued, J. Le Rousseau,
 \emph{Carleman estimates for elliptic operators with complex coefficients. Part I: Boundary value problems},
J. Math. Pures Appl. (9) 104 (2015), 657--728.

\bibitem[BLR2]{Be-LR-2} M. Bellassoued, J. Le Rousseau,\emph{ Carleman estimates for elliptic operators with complex coefficients Part II: transmission problems}, 	arXiv:1605.02535.

\bibitem[BY]{Be-Ya} M. Bellassoued, M. Yamamoto, \emph{Inverse source problem for a transmission problem for a parabolic equation}, J. Inverse Ill-Posed Probl. 14 (2006), 47--56.

\bibitem[BDLR]{Be-De-LR} A. Benabdallah, Y. Dermenjian, J. Le Rousseau, \emph{Carleman estimates for stratified media}, J. Funct. Anal. 260 (2011), 3645--3677.

\bibitem[BDT]{Be-De-Th} A. Benabdallah, Y. Dermenjian, L. Thevenet,
\emph{Carleman estimates for some non-smooth anisotropic media},
Comm. Partial Differential Equations 38 (2013), 1763--1790

\bibitem[BJS]{bjs} L. Bers, F. John and M. Schechter, Partial Differential Equations, Interscience, 1964.


\bibitem[Cal]{cal} A. Calder\'on, \emph{Uniqueness in the Cauchy  problem for partial differential equations}, Amer. J. Math., \textbf{80} (1958), 16--36.

\bibitem[Car]{ca} T. Carleman, \emph{Sur un probl\`{e}me d'unicit\'{e} pur les syst\`{e}mes d'\'{e}quations aux d\'{e}riv\'{e}es partielles \`{a} deux variables ind\'{e}pendantes}, Ark. Mat., Astr. Fys., \textbf{26} (1939), no. 17, 1--9.

\bibitem[CY]{CY} M. Choulli, M. Yamamoto, \emph{Logarithmic stability of parabolic Cauchy problems}, arXiv:1702.06299

\bibitem[DCFLVW]{dcflvw}
M. Di Cristo, E. Francini,  C.-L. Lin, S. Vessella,  J.-N. Wang, \emph{Carleman estimate for second order elliptic equations with Lipschitz leading coefficients and jumps at an interface}, in print on J. Math. Pures Appl.

\bibitem[DOP]{Do-Os-Pu} A. Doubova, A. Osses, J.-P. Puel,
\emph{Exact controllability to trajectories for semilinear heat equations with discontinuous diffusion coefficients},
ESAIM Control Optim. Calc. Var. 8 (2002), 621--661.


\bibitem[FLVW]{flvw} E. Francini, C.-L. Lin, S. Vessella, J.-N. Wang, \emph{Three-region inequalities for the second order elliptic equation with discontinuous coefficients and size estimate}, J. Differential Equations 261 (2016), no. 10, 5306--5323.

\bibitem[EF]{EsFe} L. Escauriaza and J. Fernandez, \emph{Unique continuation
for parabolic operator}, Ark. f\"{o}r Matematik 41, (1), (2003), 35--60.


\bibitem[EFV]{EsFeVe} L. Escauriaza, F. J. Fernandez, S. Vessella,
\emph{Doubling properties of caloric functions}, Applicable Analysis, 85, (1-3),
(2006) 205--223, Special issue dedicated to the memory of Carlo Pucci, editors R.
Magnanini and G. Talenti.


\bibitem[EV]{EsVe} L. Escauriaza, S. Vessella, \emph{Optimal three cylinder
inequalities for solutions to parabolic equations with Lipschitz leading
coefficients}, in G. Alessandrini and G. Uhlmann eds, \textquotedblright
Inverse Problems: Theory and Applications\textquotedblright\ Contemporary
Mathematics 333, (2003), American Mathematical Society, Providence R. I. ,
79--87.


\bibitem[F]{Fe} F. J. Fernandez, \emph{Unique continuation for parabolic
operator II}, Comm. Part. Diff. Equat., 28, (9\&10), (2003), 1597-1604.

\bibitem[FI]{FI} A. V. Fursikov and O. Yu. Imanuvilov, Controllability of
parabolic equations, Lecture Notes vol. 34 (Seoul Korea: Seoul
National University), 1996.

\bibitem[H1]{Ho63} L. H\"{o}rmander, Linear partial differential operators,
Springer, New York, 1963.

\bibitem[H2]{ho} L. H\"ormander, \emph{Uniqueness theorems for second order elliptic equations}, Comm. PDE, 8, (1983), 21-63.

\bibitem[H3]{Ho85} L. H\"{o}rmander, The analysis of linear partial differential
operators, Vol. III, Vol. IV Springer, New York, 1985.


\bibitem[KSU]{KSU} C. Kenig C, J. Sj\"{o}strand,  G. Uhlmann, \emph{The Calder\'{o}n problem with partial data}, Ann. Math. (2007) 165, 567--591.


\bibitem[Kl]{Kl1} M. V. Klibanov, Inverse problems and Carleman estimates
, Inverse Problems, 8, (1992), 575-596.

\bibitem[KlTi]{KlTi} M. V. Klibanov, A. Timonov, \emph{Carleman estimates for
coefficient inverse problems and numerical applications}, Inverse and
Ill-Posed Problems Series, VSP, Utrecht 2004.


\bibitem[KoTa]{KoTa} H. Koch, D. Tataru, \emph{Carleman estimates and unique
continuation for second order parabolic equations with nonsmooth coefficients}
, Comm. Partial Differential Equations 34 (2009), no. 4-6, 305--366.



\bibitem[I1]{Is1} V. Isakov, \emph{Carleman type estimate in an anisotropic
Case and applications}, J. of Diff. Equat., 105, (2), (1993), 217--238.

\bibitem[I2]{Is4} V. Isakov, \emph{Carleman type estimates and their
applications}, (2000), in "New Analytic and Geometric Methods in Inverse
Problems", K. Bingham, Ya. V. Kurylev, E. Sommersalo (editors), Lectures
given at the EMS Summer School and Conference, Edinburgh, 2000, Springer.

\bibitem[I3]{isakov} V. Isakov, Inverse problems for partial differential equations, volume 12 of Applied Mathematical Sciences, Springer, New York, second edition, 2006.

\bibitem[IY]{ItYa} S. Ito, H. Yamabe, \emph{A unique continuation theorem for
solutions of a parabolic differential equation}, J. Math. Soc. Japan 10,
(1958), 314--321.


\bibitem[LSU]{LSU} O. A. Ladyzhenskaya, V. A. Solonnikov, N. N. Ural'tseva,
Linear and quasilinear equations of parabolic type, Transl. Math.
Monogr. 23, AMS, Providence, R. I., 1968.

\bibitem[La]{Lan} E. M. Landis, \emph{A three sphere theorem}, Dokl. Akad.
Nauk SSSR 148 (1963), 277-279, Engl. trans. Soviet Math. Dokl. 4 (1963),
76--78.


\bibitem[LO]{LanO} E. M. Landis, O. A. Oleinik, \emph{Generalized analyticity
and some related properties of solutions of elliptic and parabolic equations}
, Russian Math. Surveys 29, (1974), 195-212.

\bibitem[LRS]{LRS} M. M. Lavrentiev, V. G. Romanov and S. P. Shishatskij,
Ill-posed problems of mathematical physics and analysis, Transl.
Math. Monogr. 64, AMS, Providence, R. I., 1986.

\bibitem[LRLeb]{LeRosseau-Lebeau} J. Le Rousseau, G. Lebeau, \emph{On Carleman estimates for elliptic and parabolic operators. Applications to unique continuation and control of parabolic equations}, ESAIM Control Optim. Calc. Var. 18 (2012), 712--747.

\bibitem[LRLer]{ll}  J. Le Rousseau, N. Lerner, \emph{Carleman estimates for anisotropic elliptic operators with jumps at an interface}, Analysis \& PDE, (6) (2013), 1601--1648.

\bibitem[LRR1]{lr1} J. Le Rousseau, L. Robbiano, \emph{Carleman estimate for elliptic operators with coefficients with jumps at an interface in arbitrary dimension and application to the null controllability of linear parabolic equations}, Arch. Rational Mech. Anal., \textbf{195} (2010), 953--990.

\bibitem[LRR2]{lr2} J. Le Rousseau, L. Robbiano, \emph{Local and global Carleman estimates for parabolic operators with coefficients with jumps at interfaces, Inventiones Math.}, \textbf{183} (2011), 245--336.

\bibitem[LeP]{LP} M. Lees,  M. H. Protter, \emph{Unique continuation for
parabolic equations}, Duke Math. J. 28, (1961), 369--382.

\bibitem[Li]{Lin} F. H. Lin, \emph{A uniqueness theorem for parabolic equations}, Comm. Pure Appl. Math. XLIII, (1990), 127--136.

\bibitem[LM]{LiMa}  J. L. Lions, E. Magenes, Non-homogeneous boundary value
problems and applications, I, II, \textit{Springer, New York, 1972.}


\bibitem[Ma]{hMan} N. Mandache, \emph{On a counterexample concerning unique
continuation for elliptic equations in divergence form}, Math. Phys. Anal.
Geom. (1), (1998), 273--292.


\bibitem[Mil]{Mi} K. Miller, \emph{Nonunique continuation for uniformly parabolic and elliptic equations in self-adjoint divergence form with H\"older continuous coefficients}, Arch. Rational Mech. Anal., 54, (1974), 105--117.

\bibitem[Miz]{Miz} S. Mizohata, \emph{Unicit\'{e} du prolongement des solutions pour quelques operateurs differentiels paraboliques}, Mem. Coll. Sci. Univ. Kyoto. Ser. A. Math. 31, (1958), 219--239.

\bibitem[N]{Nir} L. Nirenberg, \emph{ Uniqueness in Cauchy problems for
differential operators with constant leading coefficients}, Comm. Pure Appl.
Math. 10, (1957), 89--105.

\bibitem[Pl]{pl} A. Pli\'s, \emph{On non-uniqueness in Cauchy problem for an elliptic second order differential equation}, Bull. Acad. Polon. Sci. S\'er. Sci. Math. Astronom. Phys., 11, (1963), 95--100.

\bibitem[Pr]{Pr} M. H. Protter, \emph{Properties of solutions of parabolic
equations and inequalities}, Canad. J. Math., 13, (1961), 331--345.

\bibitem[SS]{SS} J. C. Saut, B. Scheurer, \emph{Unique continuation for
some evolution equations}, J. of Diff. Equat. 66, (1987), 118--139.

\bibitem[Sc]{s} F. Schulz, \emph{On the unique continuation property of elliptic divergence form equations in the plane}, Math. Z., 228, (1998), 201--206.


\bibitem[So]{So}  C. D. Sogge,
\emph{A unique continuation theorem for second order parabolic differential operators},
Ark. Mat. 28 (1990),  159--182.

\bibitem[T]{hTat} D. Tataru, \emph{Unique continuation problems for partial
differential equations}, in Geometric methods in inverse problems and PDE, IMA Vol. Math. Appl., 137,  239-255, Springer, New York, 2003.

\bibitem[Tr]{Tr} F. Tr\`{e}ves, Linear partial differential equations, Gordon and Breach, New York, 1970.


\bibitem[V1]{Ve3} S. Vessella, \emph{Carleman estimates, optimal three
cylinder inequality and unique continuation properties for solutions to
parabolic equations}, Comm. Part. Diff. Equat.28 (3\&4), (2003), 637--627.

\bibitem[V2]{Ve1} S. Vessella, \emph{Quantitative estimates of unique continuation for parabolic equations, determination of unknown time-varying boundaries and optimal stability estimates}, Inverse Problems 24 (2008), 023001.

\bibitem[V3]{Ve4} S. Vessella, \emph{Unique continuation properties and quantitative estimates of unique continuation for parabolic equations}, Handbook of differential equations: evolutionary equations. Vol. V, 421--500, Handb. Differ. Equ., Elsevier/North-Holland, Amsterdam, 2009.

\bibitem[Ya]{Yamabe} H. Yamabe, \emph{ A unique continuation theorem of a diffusion equation}. Ann. of Math. (2) 69 (1959), 462--466.

\bibitem[Yam]{Ya1} M. Yamamoto, \emph{Carleman estimates for parabolic equations and applications}, Inverse Problems 25 (2009), 123013.


\bibitem[Z]{Zu} C. Zuily, Uniqueness and non-uniqueness in the Cauchy problem,
Birkh\"{a}user, Boston 1983.











































\end{thebibliography}
\end{document}